\documentclass[a4paper]{amsart}

\usepackage[latin1]{inputenc}
\usepackage{amssymb,amsmath,amsthm}
\usepackage{amsfonts}
\usepackage{paralist}
\usepackage{enumitem}
\usepackage{booktabs}
\usepackage{ifthen}

\usepackage{url}
\usepackage{hyphenat}

\newcommand{\vect}[1]{\ensuremath{\mathbf{#1}}}
\newcommand{\card}[1]{\ensuremath{\lvert{#1}\rvert}}

\newcommand{\IN}{\ensuremath{\mathbb{N}}}
\newcommand{\IF}{\ensuremath{\mathbb{F}}}
\newcommand{\allpermutations}{\ensuremath{\mathbb{P}}}

\newcommand{\nset}[1]{\ensuremath{[{#1}]}}

\newcommand{\subperm}[2]{\ensuremath{{#1}_{#2}}}
\newcommand{\subpermsimple}[2]{\ensuremath{{#1}{\upharpoonright}_{#2}}}
\newcommand{\substring}[2]{\ensuremath{{#1}[{#2}]}}
\newcommand{\maxintervals}[1]{\ensuremath{I_{#1}}}
\newcommand{\maxintervalsmiddle}[1]{\ensuremath{I_{#1}^{-}}}
\newcommand{\interwoven}{\ensuremath{\propto}}
\newcommand{\notinterwoven}{\ensuremath{\not\propto}}
\newcommand{\maxblocksize}{\ensuremath{M}}

\newcommand{\symm}[1]{\ensuremath{S_{#1}}}         
\newcommand{\alt}[1]{\ensuremath{A_{#1}}}          
\newcommand{\cycl}[1]{\ensuremath{Z_{#1}}}         
\newcommand{\dihed}[1]{\ensuremath{D_{#1}}}        
\newcommand{\compalt}[1]{\ensuremath{O_{#1}}}      
\newcommand{\gensg}[1]{\ensuremath{\langle {#1} \rangle}}
\newcommand{\asc}[1]{\ensuremath{\iota_{#1}}}      
\newcommand{\desc}[1]{\ensuremath{\delta_{#1}}}    
\newcommand{\natcycle}[1]{\ensuremath{\zeta_{#1}}} 
\newcommand{\dja}[2]{\ensuremath{\nu^{({#1})}_{#2}}}
\newcommand{\ajd}[2]{\ensuremath{\lambda^{({#1})}_{#2}}}

\newcommand{\patt}[2]{\ensuremath{\Patl[{#1}]{#2}}}

\newcommand{\interval}[2]{\ensuremath{[{#1},{#2}]}}
\newcommand{\eqclass}[2]{\ensuremath{{#1}/{#2}}}
\newcommand{\jumps}[1]{\ensuremath{J({#1})}}

\newcommand{\interchange}[1]{\ensuremath{W_{#1}}}
\newcommand{\oddeven}[1]{\ensuremath{\text{\textit{\OE}}_{#1}}}
\newcommand{\grOE}[1]{\ensuremath{\Xi_{#1}}}

\DeclareMathOperator{\Sub}{Sub}                  
\DeclareMathOperator{\red}{red}                  
\DeclareMathOperator{\Pat}{Pat}
\DeclareMathOperator{\Comp}{Comp}
\DeclareMathOperator{\gPat}{gPat}
\DeclareMathOperator{\gComp}{gComp}
\DeclareMathOperator{\Av}{Av}
\DeclareMathOperator{\Aut}{Aut}
\DeclareMathOperator{\orbits}{Orb}
\DeclareMathOperator{\reversals}{rev}
\DeclareMathOperator{\Part}{Part}
\DeclareMathOperator{\IntPart}{IntPart}

\newcommand{\AGL}[2]{\ensuremath{\mathit{AGL}_{#1}({#2})}}
\newcommand{\AGammaL}[2]{\ensuremath{\mathit{A \Gamma L}_{#1}({#2})}}
\newcommand{\PGL}[2]{\ensuremath{\mathit{PGL}_{#1}({#2})}}
\newcommand{\PGammaL}[2]{\ensuremath{\mathit{P \Gamma L}_{#1}({#2})}}
\newcommand{\PSL}[2]{\ensuremath{\mathit{PSL}_{#1}({#2})}}

\newcommand{\Compn}[2][n]{\ifthenelse{\equal{#2}{}}{\Comp^{({#1})}}{\Comp^{({#1})}{#2}}}
\newcommand{\Patl}[2][\ell]{\ifthenelse{\equal{#2}{}}{\Pat^{({#1})}}{\Pat^{({#1})}{#2}}}
\newcommand{\gCompn}[2][n]{\ifthenelse{\equal{#2}{}}{\gComp^{({#1})}}{\gComp^{({#1})}{#2}}}
\newcommand{\gPatl}[2][\ell]{\ifthenelse{\equal{#2}{}}{\gPat^{({#1})}}{\gPat^{({#1})}{#2}}}

\theoremstyle{plain}
\newtheorem{theorem}{Theorem}[section]
\newtheorem{proposition}[theorem]{Proposition}
\newtheorem{lemma}[theorem]{Lemma}
\newtheorem{corollary}[theorem]{Corollary}

\newtheorem{claim}{Claim}
\numberwithin{claim}{theorem}

\theoremstyle{definition}
\newtheorem{definition}[theorem]{Definition}
\newtheorem{example}[theorem]{Example}

\newtheorem{fact}[theorem]{Fact}

\theoremstyle{remark}
\newtheorem{remark}[theorem]{Remark}

\numberwithin{equation}{section}

\begin{document}
\title[Permutation groups arising from pattern involvement]{Permutation groups \\ arising from pattern involvement}
\author{Erkko Lehtonen}
\address{Technische Universit\"at Dresden,
Institut f\"ur Algebra \\
01062 Dresden,
Germany}


\begin{abstract}
\noindent
For an arbitrary finite permutation group $G$, subgroup of the symmetric group $\symm{\ell}$, we determine
the permutations involving only members of $G$ as $\ell$\hyp{}patterns,
i.e., avoiding all patterns in the set $\symm{\ell} \setminus G$.
The set of all $n$\hyp{}permutations with this property constitutes again a permutation group.
We consequently refine and strengthen the classification of sets of permutations closed under pattern involvement and composition that is due to Atkinson and Beals.
\end{abstract}

\maketitle

\section{Introduction}

This study combines two different points of view to permutations:
the algebraic one of permutation groups and the combinatorial one of permutation patterns.
Permutation groups are a fundamental and extensively studied topic in classical algebra, and it requires no further introduction here.
The theory of permutation patterns and pattern avoidance has been an active field of research over the past decades.
The basic building block of the theory is the pattern involvement relation, which is a partial order on the set of all finite permutations (see Section~\ref{sec:preliminaries} for precise definitions).
Sets of permutations that are downwards closed with respect to the pattern involvement order are called permutation classes.
Permutation classes can be equivalently described by certain pattern avoidance conditions.

A recurrent theme in the theory of permutation patterns is that of describing and enumerating specific permutation classes, especially classes avoiding a small number of patterns.
Following this line of research, we focus in this paper on permutation classes that are motivated by algebraic considerations.
Namely, there is a perhaps surprising connection between permutation patterns and permutation groups that was observed by P\"oschel and the current author in~\cite{LehPos}:
for any permutation group $G$, a subgroup of the symmetric group $\symm{\ell}$, and for every $n \geq \ell$, the set of $n$\hyp{}permutations involving only members of $G$ as $\ell$\hyp{}patterns is a subgroup of $\symm{n}$; in other words, the class $\Av(\symm{\ell} \setminus G)$ of permutations avoiding the complement of $G$ is comprised of levels that are permutation groups.

All this is closely related to the work of Atkinson and Beals~\cite{AtkBea1999,AtkBea2001} on group classes, that is, permutation classes in which every level is a permutation group.
They determined
how level sequences of group classes eventually behave
(see Theorem~\ref{thm:AtkBea-asymptotic}).
Moreover, they completely and explicitly described those group classes in which every level is a transitive group (see Theorem~\ref{thm:AtkBea-T4}).

The main goal of the current paper is, in other words, to refine and strengthen Atkinson and Beals's results on group classes.
On the one hand, we would like to focus on the local behaviour of the level sequence of $\Av(\symm{\ell} \setminus G)$, for an arbitrary group $G \leq \symm{\ell}$.
On the other hand, we would like to find out how fast this level sequence reaches the
eventual
behaviour predicted by Atkinson and Beals's results.

Our main technical tool is the monotone Galois connection $(\Compn{}, \Patl{})$ between the symmetric groups $\symm{\ell}$ and $\symm{n}$ ($\ell \leq n$), which was introduced by P\"oschel and the current author in \cite{LehPos}.
For subsets $S \subseteq \symm{\ell}$ and $T \subseteq \symm{n}$, $\Patl{T}$ is the set of all $\ell$\hyp{}patterns of members of $T$, and $\Compn{S}$ is the set of all $n$\hyp{}permutations all $\ell$\hyp{}patterns of which belong to $S$ (in other words, all $n$\hyp{}permutations avoiding $\symm{\ell} \setminus S$).
Our goal is to determine for an arbitrary permutation group $G$, a subgroup of $\symm{\ell}$, and for every $n \geq \ell$, the set $\Compn{G}$.
With the help of a ``transitive property'' of the operators $\Compn{}$, $\Patl{}$, it is sufficient to determine only $\Compn[n+1]{G}$ for any $G \leq \symm{n}$ (see Lemma~\ref{lem:Comp-Pat-transitive}).
The analysis is split into several cases corresponding to certain classes of permutation groups (such as intransitive, imprimitive, and primitive groups), which require different proof techniques.

Accordingly, the main results of this paper are of the following form:
\emph{If $G \leq \symm{n}$ is a permutation group belonging to a certain class $\mathcal{C}$ of permutation groups, then $\Compn[n+1]{G}$ equals a certain subgroup of $\symm{n+1}$.}
Sometimes also $\Compn[n+2]{G}$ is described directly.
Repeated applications of such theorems, together with the ``transitive property'' then yield the sequence $G, \Compn[n+1]{G}, \dots, \Compn[n+i]{G}, \dots$ for any group $G \leq \symm{n}$.

The material is organized as follows.
In Section~\ref{sec:preliminaries}, we provide basic definitions related to permutations, permutation patterns, and other notions that will be used in this paper.
Then, in Section~\ref{sec:AtkBea}, we quote the main results from the paper by Atkinson and Beals~\cite{AtkBea2001} and briefly explain our main goals.
We first introduce some notation and tools for the proofs and deal with a few special groups (symmetric and trivial groups, as well as the group generated by the descending permutation) in Section~\ref{sec:tools}.
Each one of the remaining sections deals with a particular type of groups:
alternating groups (Section~\ref{sec:alternating}),
groups containing the natural cycle (Section~\ref{sec:cycle}),
intransitive groups (Section~\ref{sec:intransitive}),
imprimitive groups (Section~\ref{sec:imprimitive}),
and, finally,
primitive groups (Section~\ref{sec:primitive}).
We make some concluding remarks in Section~\ref{sec:concluding}.


\section{Permutation patterns}
\label{sec:preliminaries}

The set of nonnegative integers is denoted by $\IN$, and $\IN_+ := \IN \setminus \{0\}$.
For any $a, b \in \IN$, the \emph{interval} $\{i \in \IN \mid a \leq i \leq b\}$ is denoted by $\interval{a}{b}$.
(Note that $\interval{a}{b} = \emptyset$ if $a > b$.)
For any $n \in \IN_+$, the interval $\interval{1}{n}$ is denoted simply by $\nset{n}$.

We recall some standard terminology and notation related to permutations, permutation groups, and permutation patterns.
For more background on these topics, we refer the reader, e.g., to the books by B\'ona~\cite{Bona}, Dixon and Mortimer~\cite{DixMor}, and Kitaev~\cite{Kitaev}.
The set of all permutations of the set $\nset{n}$ (also referred to as \emph{$n$\hyp{}permutations}) is denoted by $\symm{n}$.
Together with the operation of functional composition, it constitutes a group called the \emph{symmetric group \textup{(}of degree $n$\textup{)}.}
The subgrops of the symmetric group $\symm{n}$ are called \emph{permutation groups \textup{(}of degree $n$\textup{)}.}
The subgroup generated by a subset $S \subseteq \symm{n}$ is denoted by $\gensg{S}$.
We write $G \leq H$ to denote that $G$ is a subgroup of $H$.
To be precise, when we speak about ``permutation groups'' or ``subgroups'', we usually really mean subuniverses of the symmetric group $\symm{n}$.
It should also be emphasized here that we will always compose functions (in particular, permutations) from right to left, and we often denote functional composition simply by juxtaposition. Thus $fg(x) = (f \circ g)(x) = f(g(x))$.

Any permutation $\pi \in \symm{n}$ induces a bijective map $\hat{\pi}$ on the power set $\mathcal{P}(\nset{n})$, defined as follows: $\hat{\pi} \colon \mathcal{P}(\nset{n}) \to \mathcal{P}(\nset{n})$, $\hat{\pi}(X) := \{\pi(x) \mid x \in X\}$ for all $X \in \mathcal{P}(\nset{n})$.
This further induces a bijection $\hat{\hat{\pi}}$ on $\mathcal{P}(\mathcal{P}(\nset{n}))$ defined as $\hat{\hat{\pi}}(\Pi) \colon \mathcal{P}(\mathcal{P}(\nset{n})) \to \mathcal{P}(\mathcal{P}(\nset{n}))$, $\hat{\hat{\pi}}(\Pi) := \{\hat{\pi}(B) \mid B \in \Pi\}$ for all $\Pi \in \mathcal{P}(\mathcal{P}(\nset{n}))$.
Without any risk of confusion, we will denote the maps $\hat{\pi}$ and $\hat{\hat{\pi}}$ simply by $\pi$.

We will often write permutations $\pi \in \symm{n}$ as strings $\pi_1 \pi_2 \dots \pi_n$, where $\pi_i = \pi(i)$ for all $i \in \nset{n}$.
We will also use the conventional cycle notation for permutations: if $a_1, a_2, \dots, a_\ell$ are distinct elements of $\nset{n}$, then $(a_1 \; a_2 \; \cdots \; a_\ell)$ denotes the permutation that maps $a_\ell$ to $a_1$ and $a_i$ to $a_{i+1}$ for $1 \leq i \leq \ell - 1$ and keeps the remaining elements fixed.
Such a permutation is called a \emph{cycle}, or an \emph{$\ell$\hyp{}cycle.}
Every permutation is a product of pairwise disjoint cycles.

\begin{definition}
The following permutations will play an important role throughout the paper:
\begin{itemize}
\item the \emph{identity permutation,} or the \emph{ascending permutation} $\asc{n} := 12 \dots n$,
\item the \emph{descending permutation} $\desc{n} := n (n-1) \dots 1$,
\item the \emph{natural cycle} $\natcycle{n} := 2 3 \dots n 1 = (1 \; 2 \; \cdots \; n)$.
\end{itemize}

The subgroup $\gensg{\natcycle{n}}$ of $\symm{n}$ generated by the natural cycle $\natcycle{n}$ is called the \emph{natural cyclic group} of degree $n$ and is denoted by $\cycl{n}$.
The subgroup $\gensg{\natcycle{n}, \desc{n}}$ is called the \emph{natural dihedral group} of degree $n$ and is denoted by $\dihed{n}$.
\end{definition}

The \emph{direct sum} $\sigma \oplus \tau$ and the \emph{skew sum} $\sigma \ominus \tau$ of two permutations $\sigma \in \symm{m}$ and $\tau \in \symm{n}$ are permutations of the set $\nset{m+n}$ and they are given by the following rules:
\begin{align*}
(\sigma \oplus \tau)(i) &=
\begin{cases}
\sigma(i), & \text{if $1 \leq i \leq m$,} \\
m + \tau(i - m), & \text{if $m+1 \leq i \leq m+n$,}
\end{cases}
\\
(\sigma \ominus \tau)(i) &=
\begin{cases}
\makebox[0pt][l]{$n + \sigma(i)$,}
\phantom{m + \tau(i - m),} & \text{if $1 \leq i \leq m$,} \\
\tau(i - m), & \text{if $m+1 \leq i \leq m+n$.}
\end{cases}
\end{align*}

For any string $\vect{a} = a_1 a_2 \dots a_n$ and an index set $I \subseteq \nset{n}$, $I = \{i_1, i_2, \dots, i_\ell\}$ with $i_1 < i_2 < \dots < i_\ell$, we denote the scattered substring $a_{i_1} a_{i_2} \dots a_{i_\ell}$ of $\vect{a}$ by $\substring{\vect{a}}{I}$.
For any string $\vect{u} = u_1 u_2 \dots u_\ell$ of distinct integers, the \emph{reduction} or \emph{reduced form} of $\vect{u}$, denoted by $\red(\vect{u})$, is the permutation obtained from the string $\vect{u}$ by replacing its $i$\hyp{}th smallest entry with $i$, for $1 \leq i \leq \ell$.

A permutation $\tau = \tau_1 \tau_2 \dots \tau_\ell \in \symm{\ell}$ is a \emph{pattern} (or an \emph{$\ell$\hyp{}pattern,} if we want to emphasize the number $\ell$) of a permutation $\pi = \pi_1 \pi_2 \dots \pi_n \in \symm{n}$, or $\pi$ \emph{involves} $\tau$, denoted $\tau \leq \pi$, if there exists a substring $\substring{\pi}{I} = \pi_{i_1} \pi_{i_2} \dots \pi_{i_\ell}$ of $\pi$ (where $I = \{i_1, i_2, \dots, i_\ell\}$, $i_1 < i_2 < \dots < i_\ell$) such that $\red(\substring{\pi}{I}) = \tau$.
We denote by $\subperm{\pi}{I}$ the pattern of $\pi$ corresponding to the index set $I \subseteq \nset{n}$, i.e., $\subperm{\pi}{I} := \red(\substring{\pi}{I})$.
If $\tau \nleq \pi$, the permutation $\pi$ is said to \emph{avoid} $\tau$.
The pattern involvement relation $\leq$ is a partial order on the set $\allpermutations := \bigcup_{n \geq 1} \symm{n}$ of all finite permutations.
Furthermore, every covering relationship in this order links permutations of two consecutive degrees.

\begin{fact}
\label{fact:between}
If $\ell \leq m \leq n$ and $\sigma \in \symm{\ell}$, $\tau \in \symm{n}$ and $\sigma \leq \tau$, then there exists $\pi \in \symm{m}$ such that $\sigma \leq \pi \leq \tau$.
(For a proof of this well\hyp{}known fact, see, e.g., \cite[Lemma~2.8]{LehPos}.)
\end{fact}

Downward closed subsets of $\allpermutations$ under the pattern involvement order are called \emph{permutation classes.}
For a permutation class $C$ and for $n \in \IN_+$, the set $C^{(n)} := C \cap \symm{n}$ is called the \emph{$n$\hyp{}th level} of $C$.
For any set $B \subseteq \allpermutations$, let $\Av(B)$ be the set of all permutations that avoid every member of $B$.
It is clear that $\Av(B)$ is a permutation class and, conversely, every permutation class is of the form $\Av(B)$ for some $B \subseteq \allpermutations$.

\begin{definition}
Let $\ell, n \in \IN_+$ with $\ell \leq n$.
We denote by $\patt{\ell}{\pi}$ the set of all $\ell$\hyp{}patterns of $\pi$, i.e., $\patt{\ell}{\pi} := \{\tau \in \symm{\ell} : \tau \leq \pi\}$. 
We say that a permutation $\tau \in \symm{n}$ is \emph{compatible} with a set $S \subseteq \symm{\ell}$ of $\ell$\hyp{}permutations if $\patt{\ell}{\tau} \subseteq S$. For $S \subseteq \symm{\ell}$, $T \subseteq \symm{n}$, we write
\begin{align*}
\Compn{S} & := \{\tau \in \symm{n} \mid \patt{\ell}{\tau} \subseteq S\}
= (\Av(\symm{\ell} \setminus S))^{(n)}
\\
\Patl{T} & := \bigcup_{\tau \in T} \patt{\ell}{\tau}.
\end{align*}
\end{definition}

\begin{example}
\label{ex:asc-desc-cycle}
For any $\ell, n \in \IN_+$ with $\ell < n$,
\[
\patt{\ell}{\asc{n}} = \{\asc{\ell}\},
\qquad
\patt{\ell}{\desc{n}} = \{\desc{\ell}\},
\qquad
\patt{\ell}{\natcycle{n}} = \{\asc{\ell}, \natcycle{\ell}\}.
\]
\end{example}

\begin{remark}
\label{rem:level-sequences}
For any permutation class $C$ and for any $\ell, m, n \in \IN_+$ with $\ell \leq m \leq n$, it obviously holds that
\[
\Patl[\ell]{C^{(m)}} \subseteq C^{(\ell)},
\qquad
C^{(n)} \subseteq \Compn[n]{C^{(m)}}.
\]
In other words, given just the $m$\hyp{}th level of a permutation class $C$, the operators $\Patl{}$ and $\Compn{}$ provide lower bounds for lower levels and upper bounds for higher levels of $C$, respectively.
Keeping in mind that every permutation class is of the form $\Av(B) = \bigcap_{m \in \IN_+} \Av(B^{(m)})$ for some $B \subseteq \allpermutations$, and for every $n \in \IN_+$ it holds that $(\Av{B^{(m)}})^{(n)} = \Compn[n]{(\symm{m} \setminus B^{(m)})}$,
we see that studying the operators $\Patl{}$, $\Compn{}$ is essentially the same thing as studying permutation classes.
\end{remark}

\begin{lemma}
\label{lem:Comp-intersection}
For any $S, S' \subseteq \symm{\ell}$, $\Compn[n]{(S \cap S')} = \Compn[n]{S} \cap \Compn[n]{S'}$.
\end{lemma}

\begin{proof}
For any $\pi \in \symm{n}$, it holds that
\begin{align*}
&
\pi \in \Compn[n]{(S \cap S')}
\iff
\Patl[\ell]{\pi} \subseteq S \cap S'
\\ &
\iff
\Patl[\ell]{\pi} \subseteq S \wedge \Patl[\ell]{\pi} \subseteq S'
\\ &
\iff
\pi \in \Compn[n]{S} \wedge \pi \in \Compn[n]{S'}
\\ &
\iff
\pi \in \Compn[n]{S} \cap \Compn[n]{S'}.
\qedhere
\end{align*}
\end{proof}

The operators $\Compn{}$ and $\Patl{}$ are clearly monotone, i.e., if $S \subseteq S'$ then $\Compn{S} \subseteq \Compn{S'}$, and if $T \subseteq T'$ then $\Patl{T} \subseteq \Patl{T'}$.
It was shown in~\cite[Section~3]{LehPos} that $\Compn{}$ and $\Patl{}$ are the upper and lower adjoints of the monotone Galois connection (residuation) between the power sets $\mathcal{P}(\symm{\ell})$ and $\mathcal{P}(\symm{n})$ induced by the pattern avoidance relation $\nleq$.
This means that $\Patl{\Compn{}}$ and $\Compn{\Patl{}}$ are kernel and closure operators, respectively, i.e., for all $S \subseteq \symm{\ell}$ and $T \subseteq \symm{n}$, it holds that
\[
\begin{array}{c@{\qquad}c}
\Patl{\Compn{S}} \subseteq S, & \Compn{S} = \Compn{\Patl{\Compn{S}}}, \\
T \subseteq \Compn{\Patl{T}}, & \Patl{T} = \Patl{\Compn{\Patl{T}}}.
\end{array}
\]

Our work builds on a few simple yet crucial properties of the operators $\Patl{}$ and $\Compn{}$ that are described in Lemma~\ref{lem:Ppitau-PpiPtau}, Proposition~\ref{prop:S-subgroup}, and Lemma~\ref{lem:Comp-Pat-transitive}.
It is exactly these properties that make a connection between permutation groups and permutation patterns.

\begin{lemma}[{Lehtonen, P\"oschel~\cite[Lemma~2.6(ii)]{LehPos}}]
\label{lem:Ppitau-PpiPtau}
Let $\pi, \tau \in \symm{n}$, and let $\ell \leq n$.
Then $\Patl{\pi \tau} \subseteq (\Patl{\pi}) (\Patl{\tau}) := \{\sigma \sigma' \mid \sigma \in \Patl{\pi},\, \sigma' \in \Patl{\tau}\}$.
\end{lemma}

\begin{proposition}[{Lehtonen, P\"oschel \cite[Proposition~3.1]{LehPos}}]
\label{prop:S-subgroup}
If $S$ is a subgroup of $\symm{\ell}$, then $\Compn{S}$ is a subgroup of $\symm{n}$.
\end{proposition}

Using the standard terminology of the theory of permutation patterns, we can rephrase Proposition~\ref{prop:S-subgroup} as follows:
the set of $n$\hyp{}permutations avoiding the complement of a subgroup of $\symm{\ell}$ is a subgroup of $\symm{n}$.

For $S \subseteq \symm{\ell}$ and $T \subseteq \symm{n}$, write
\[
\gCompn{S} = \gensg{\Compn{S}},
\qquad
\gPatl{T} = \gensg{\Patl{T}}.
\]
By Proposition~\ref{prop:S-subgroup}, the equality $\gCompn{S} = \Compn{S}$ holds whenever $S$ is a subgroup of $\symm{\ell}$.
In fact, it was shown in \cite[Lemma~3.3]{LehPos} that $(\gCompn{}, \gPatl{})$ constitutes a monotone Galois connection between the subgroup lattices $\Sub(\symm{\ell})$ and $\Sub(\symm{n})$ of the symmetric groups $\symm{\ell}$ and $\symm{n}$.
This means that $\gPatl{\gCompn{}}$ and $\gCompn{\gPatl{}}$ are kernel and closure operators, respectively, i.e., for all subgroups $G \leq \symm{\ell}$ and $H \leq \symm{n}$, it holds that
\[
\begin{array}{c@{\qquad}c}
\gPatl{\gCompn{G}} \leq G, & \gCompn{G} = \gCompn{\gPatl{\gCompn{G}}}, \\
H \leq \gCompn{\gPatl{H}}, & \gPatl{H} = \gPatl{\gCompn{\gPatl{H}}}.
\end{array}
\]

The operators $\Compn{}$ and $\Patl{}$, as well as $\gCompn{}$ and $\gPatl{}$ satisfy the following ``transitive property''.

\begin{lemma}
\label{lem:Comp-Pat-transitive}
\label{lem:gComp-gPat-transitive}
Assume that $\ell \leq m \leq n$.
Then for all subsets $S \subseteq \symm{\ell}$, $T \subseteq \symm{n}$,
\[
\Compn[n]{\Compn[m]{S}} = \Compn[n]{S},
\qquad
\Patl[\ell]{\Patl[m]{T}} = \Patl[\ell]{T},
\]
and for all subgroups $G \leq \symm{\ell}$, $H \leq \symm{n}$,
\[
\gCompn[n]{\gCompn[m]{G}} = \gCompn[n]{G},
\qquad
\gPatl[\ell]{\gPatl[m]{H}} = \gPatl[\ell]{H}.
\]
\end{lemma}

\begin{proof}
We prove first the equality $\Patl[\ell]{\Patl[m]{T}} = \Patl[\ell]{T}$.
Assume that $\sigma \in \Patl[\ell]{\Patl[m]{T}}$.
Then there exists $\pi \in \Patl[m]{T}$ such that $\sigma \leq \pi$.
This in turn implies that there exists $\tau \in T$ such that $\pi \leq \tau$.
Since the pattern involvement relation is transitive, we have $\sigma \leq \tau$, so $\sigma \in \Patl[\ell]{T}$.
Thus the inclusion $\Patl[\ell]{\Patl[m]{T}} \subseteq \Patl[\ell]{T}$ holds.
For the converse inclusion, assume that $\sigma \in \Patl[\ell]{T}$.
Then there exists $\tau \in T$ such that $\sigma \leq \tau$.
By Fact~\ref{fact:between}, there exists $\pi \in \symm{m}$ such that $\sigma \leq \pi \leq \tau$.
This implies that $\pi \in \Patl[m]{T}$; hence $\sigma \in \Patl[\ell]{\Patl[m]{T}}$, so $\Patl[\ell]{T} \subseteq \Patl[\ell]{\Patl[m]{T}}$.

Next we prove the equality $\Compn[n]{\Compn[m]{S}} = \Compn[n]{S}$.
Assume that $\tau \in \Compn[n]{\Compn[m]{S}}$, and
let $\sigma$ be an $\ell$\hyp{}pattern of $\tau$.
By Fact~\ref{fact:between}, there exists $\pi \in \symm{m}$ such that $\sigma \leq \pi \leq \tau$.
Since $\tau \in \Compn[n]{\Compn[m]{S}}$, we have $\pi \in \Compn[m]{S}$. Consequently, $\sigma \in S$.
Thus $\tau \in \Compn[n]{S}$, so the inclusion $\Compn[n]{\Compn[m]{S}} \subseteq \Compn[n]{S}$ holds.
For the converse inclusion, assume that $\tau \in \Compn[n]{S}$.
Let $\pi \in \symm{m}$ such that $\pi \leq \tau$.
Then for every $\sigma \in \symm{\ell}$ such that $\sigma \leq \pi$, we have $\sigma \leq \tau$ by the transitivity of pattern involvement, so $\sigma \in S$.
It follows that $\pi \in \Compn[m]{S}$.
Hence $\tau \in \Compn[n]{\Compn[m]{S}}$, so $\Compn[n]{S} \subseteq \Compn[n]{\Compn[m]{S}}$.

The equality $\gCompn[n]{\gCompn[m]{G}} = \gCompn[n]{G}$ follows from Proposition~\ref{prop:S-subgroup} and what we have shown above.
Namely, for any subgroup $G \leq \symm{\ell}$, we have
\begin{multline*}
\gCompn[n]{G}
= \gensg{\Compn[n]{G}}
= \gensg{\Compn[n]{\Compn[m]{G}}} \\
= \gensg{\Compn[n]{\gensg{\Compn[m]{G}}}}
= \gCompn[n]{\gCompn[m]{G}}.
\end{multline*}

For the last equality $\gPatl[\ell]{\gPatl[m]{H}} = \gPatl[\ell]{H}$,
the inclusion $\gPatl[\ell]{H} \subseteq \gPatl[\ell]{\gPatl[m]{H}}$ can be proved by making use of the monotonicity of $\Patl{}$, as follows:
\begin{align*}
\gPatl[\ell]{H}
&= \gensg{\Patl[\ell]{H}}
= \gensg{\Patl[\ell]{\Patl[m]{H}}} \\
&\subseteq \gensg{\Patl[\ell]{\gensg{\Patl[m]{H}}}}
= \gPatl[\ell]{\gPatl[m]{H}}.
\end{align*}
In order to prove the converse inclusion $\gPatl[\ell]{\gPatl[m]{H}} \subseteq \gPatl[\ell]{H}$, assume that $\sigma \in \gPatl[\ell]{\gPatl[m]{H}}$.
Then there exist $\sigma^1, \sigma^2, \dots, \sigma^p \in \Patl[\ell]{\gPatl[m]{H}}$ such that $\sigma = \sigma^1 \circ \sigma^2 \circ \cdots \circ \sigma^p$.
This in turn means that, for each $i \in \{1, \dots, p\}$, there exists $\tau^i \in \gPatl[m]{H}$ such that $\sigma^i \leq \tau^i$.
Consequently, for each $i \in \{1, \dots, p\}$, there exist $\tau^{i1}, \tau^{i2}, \dots, \tau^{ij_i} \in \Patl[m]{H}$ such that $\tau^i = \tau^{i1} \circ \tau^{i2} \circ \cdots \circ \tau^{ij_i}$.
Now, by applying Lemma~\ref{lem:Ppitau-PpiPtau} and the monotonicity of $\Patl{}$, we obtain, for every $i \in \{1, \dots, p\}$, that
\begin{align*}
\sigma^i
& \in \Patl[\ell]{\tau^i}
= \Patl[\ell]{\tau^{i1} \tau^{i2} \cdots \tau^{ij_i}}
\subseteq (\Patl[\ell]{\tau^{i1}}) (\Patl[\ell]{\tau^{i2}}) \cdots (\Patl[\ell]{\tau^{ij_i}}) \\
& \subseteq (\Patl[\ell]{\Patl[m]{H}}) (\Patl[\ell]{\Patl[m]{H}}) \cdots (\Patl[\ell]{\Patl[m]{H}}) \\
& = (\Patl[\ell]{H}) (\Patl[\ell]{H}) \cdots (\Patl[\ell]{H})
\subseteq \gensg{\Patl[\ell]{H}}
= \gPatl[\ell]{H}.
\end{align*}
Since $\gPatl[\ell]{H}$ is a group, we have $\sigma = \sigma^1 \circ \sigma^2 \circ \cdots \circ \sigma^p \in \gPatl[\ell]{H}$.
\end{proof}


\section{Group classes}
\label{sec:AtkBea}

Atkinson and Beals~\cite{AtkBea1999,AtkBea2001} investigated \emph{group classes,} or \emph{group closed sets,} i.e., permutation classes in which every level is a permutation group.
They determined the eventual behaviour of level sequences of group classes, and we rephrase this result in Theorem~\ref{thm:AtkBea-asymptotic} below.
As the theorem reveals, the level sequence of any group class eventually coincides with one of only a few possible sequences of permutation groups, which we call \emph{stable sequences.}
In this description, $\symm{n}^{a,b}$ (for $n \geq a+b$) denotes the group of all permutations in $\symm{n}$ that map each one of the intervals $\interval{1}{a}$ and $\desc{n}(\interval{1}{b}) = \interval{n-b+1}{n}$ onto itself and fix the points in $\interval{a+1}{n-b}$.
Note that $\symm{n}^{1,1}$ is the trivial subgroup of $\symm{n}$.
(Since we are only concerned about the eventual behaviour of the level sequence, it does not matter how $\symm{n}^{a,b}$ is defined when $n < a + b$, so we may just take $\symm{n}^{a,b} := \symm{n}$ in this case.)

\begin{theorem}[Atkinson, Beals~\cite{AtkBea1999,AtkBea2001}]
\label{thm:AtkBea-asymptotic}
If $C$ is a permutation class in which every level $C^{(n)}$ is a permutation group, then the level sequence $C^{(1)}, C^{(2)}, \dots, C^{(i)}, \dots$ eventually coincides with one of the following sequences of permutation groups:
\begin{enumerate}[label=\upshape{(\arabic*)}]
\item\label{asymptotic-families-1} the groups $\symm{n}^{a,b}$ for some fixed $a, b \in \IN_+$,
\item the natural cyclic groups $\cycl{n}$,
\item the full symmetric groups $\symm{n}$,
\item the groups $\gensg{G_n, \desc{n}}$, where $(G_n)_{n \in \IN}$ is one of the above sequences \textup{(}with $a = b$ in \ref{asymptotic-families-1}\textup{)}.
\end{enumerate}
\end{theorem}

Moreover, Atkinson and Beals completely and explicitly described those group classes in which every level is a transitive group.
Following the terminology of \cite{AtkBea2001}, a permutation group $G \leq \symm{n}$ is \emph{anomalous} if $\natcycle{n} \in G$ and $\gensg{\Patl[n-1]{G}} \neq \symm{n-1}$.
The natural cyclic and dihedral groups are anomalous.
Alternating groups are not anomalous, because $A_n$ contains the $3$\hyp{}cycles $(i \; i+1 \; i+2)$ for $1 \leq i \leq n-2$, the adjacent transposition $(i \; i+1)$ is an $(n-1)$\hyp{}pattern of $(i \; i+1 \; i+2)$, and these adjacent transpositions generate the full symmetric group $\symm{n-1}$.
Note that the alternating group $\alt{n}$ contains the natural cycle $\natcycle{n}$ if and only if $n$ is odd.
(For more insight on anomalous groups, see Proposition~\ref{prop:anomalous}.)

\begin{theorem}[{Atkinson, Beals~\cite[Theorem~2]{AtkBea2001}}]
\label{thm:AtkBea-T4}
Let $C$ be a permutation class in which every level $C^{(n)}$ is a transitive group.
Then, with the exception of at most two levels, one of the following holds.
\begin{enumerate}[label=\upshape{(\arabic*)}]
\item $C^{(n)} = \symm{n}$ for all $n \in \IN_+$.
\item For some $M \in \IN$, $C^{(n)} = \symm{n}$ for $1 \leq n \leq M$, and $C^{(n)} = \dihed{n}$ for $n > M$.
\item For some $M, N \in \IN$ with $M \leq N$, $C^{(n)} = \symm{n}$ for $1 \leq n \leq M$, $C^{(n)} = \dihed{n}$ for $M + 1 \leq n \leq N$, and $C^{(n)} = \cycl{n}$ for $n > N$.
\end{enumerate}
The exceptions, if any, may occur in the second and third cases and are of the following two possible types:
\begin{enumerate}[label=\upshape{(\roman*)}]
\item $C^{(M+1)} = \alt{M+1}$ and $C^{(M+2)}$ is an anomalous group that is neither $\dihed{M+2}$ nor $\cycl{M+2}$, or
\item $C^{(M+1)}$ is a proper overgroup of $\cycl{M+1}$ but is not $\dihed{M+1}$.
\end{enumerate}
\end{theorem}

It should be noted that every group mentioned in Theorem~\ref{thm:AtkBea-T4} contains the natural cycle (of the appropriate degree).
For group classes with an intransitive group at some level, only the
eventual behaviour of the level sequence is revealed by
the above theorems.
This raises the question how the level sequence of an arbitrary group class behaves before reaching one of the possible eventual stable sequences and how fast a stable sequence is reached.

The main goal of this paper is to refine and strengthen
Theorems~\ref{thm:AtkBea-asymptotic} and \ref{thm:AtkBea-T4}.
More precisely, we would like to describe the sequence
\begin{equation}
\label{eq:comp-sequence}
G, \Compn[n+1]{G}, \Compn[n+2]{G}, \dots, \Compn[n+i]{G}, \dots
\end{equation}
for an arbitrary permutation group $G \leq \symm{n}$.
In view of Remark~\ref{rem:level-sequences} and Proposition~\ref{prop:S-subgroup}, studying such sequences is essentially the same thing as studying the levels of group classes.
By Lemma~\ref{lem:Comp-Pat-transitive}, it would suffice to determine $\Compn[n+1]{G}$ for every group $G \leq \symm{n}$ and for every $n \in \IN_+$.
The entire sequence \eqref{eq:comp-sequence} could then be determined by applying such results repeatedly, one level at a time.
Unfortunately, in some cases, the exact description of $\Compn[n+1]{G}$ eludes us.
Taking a suitable subgroup $H_1$ and an overgroup $H_2$ of $G$, the monotonicity of $\Compn[n+1]{}$ implies that $\Compn[n+1]{H_1} \leq \Compn[n+1]{G} \leq \Compn[n+1]{H_2}$, and we may nevertheless be able to obtain good lower and upper bounds for $\Compn[n+1]{G}$.
With these results, it is also possible to find the exact point (or a good estimate) where the sequence \eqref{eq:comp-sequence} reaches
one of the eventual stable sequences
predicted by Theorem~\ref{thm:AtkBea-asymptotic}.
In fact, we will see that, unless $\natcycle{n} \notin G$ and $G$ is intransitive or imprimitive, the smallest $i \in \IN$ for which $\Compn[n+i]{G}$
coincides with one of the stable sequences
is at most $2$.

Our study is divided into six main sections, each dealing with a certain type of groups and employing particular proof techniques.
We first consider separately some special groups, namely, the symmetric and trivial groups, as well as the group generated by the descending permutation $\desc{n}$.
Second, we consider alternating groups.
For the remaining groups, we make a first distinction between those groups that contain the natural cycle and those that do not.
The groups not containing the natural cycle are divided into intransitive and transitive ones.
The transitive groups without a natural cycle are further subdivided into imprimitive and primitive ones.
(Recall that for a permutation group $G \leq \symm{n}$ and an element $i \in \nset{n}$, the \emph{$G$\hyp{}orbit} of $i$ is the set $\{\pi(i) \mid \pi \in G\}$.
The orbits of $G$ partition $\nset{n}$.
A permutation group is \emph{transitive} if it has only one orbit; otherwise it is \emph{intransitive.}
A transitive subgroup of $\symm{n}$ is \emph{primitive} if it preserves no nontrivial partition of $\nset{n}$; otherwise it is \emph{imprimitive.})
This case distinction is summarized in Table~\ref{table:roadmap}, which also provides an index to the main results.

\begin{table}
\begin{center}
\begin{tabular}{ll}
\toprule
Type of group & Result \\
\midrule
$\bullet$ $\symm{n}$, $\gensg{\desc{n}}$, trivial & Theorem~\ref{thm:symmetric-trivial} \\
$\bullet$ $\alt{n}$ & Theorems~\ref{thm:Comp-alt}, \ref{thm:Comp-alt-n+2} \\
$\bullet$ $\natcycle{n} \in G$ and $\alt{n} \nleq G$ & Theorem~\ref{thm:Comp-cycle} \\
$\bullet$ $\natcycle{n} \notin G$: & \\
\quad -- intransitive & Theorem~\ref{thm:CompSPi-general}, \ref{thm:SPi-asymptotic}, \ref{thm:general-intransitive} \\
\quad -- transitive: & \\
\quad\quad -- imprimitive & Theorems~\ref{thm:CompAutPi}, \ref{thm:CompAutPi-2}, \ref{thm:general-imprimitive} \\
\quad\quad -- primitive & Theorem~\ref{thm:Comp-primitive} \\
\bottomrule
\end{tabular}
\end{center}

\bigskip
\caption{Index of main results.}
\label{table:roadmap}
\end{table}

We would like to emphasize that we are dealing with concrete subgroups of the symmetric group $\symm{n}$, not with abstract groups.
Even if two subgroups $G$ and $H$ of $\symm{n}$ are isomorphic, or even conjugate, the groups $\Compn[n+1]{G}$ and $\Compn[n+1]{H}$ may be markedly different.


\section{Tools and notation, symmetric and trivial groups}
\label{sec:tools}
\label{sec:simple}

Recall that the \emph{reverse} of a permutation $\pi \in \symm{n}$ is $\pi^\mathrm{r} = \pi \circ \desc{n}$, and the \emph{complement} of $\pi$ is $\pi^\mathrm{c} = \desc{n} \circ \pi$.
It is well known that pattern involvement is preserved under reverses, complements, and inverses of permutations, i.e.,
\[
\tau \leq \pi \iff \tau^\mathrm{r} \leq \pi^\mathrm{r}, \qquad
\tau \leq \pi \iff \tau^\mathrm{c} \leq \pi^\mathrm{c}, \qquad
\tau \leq \pi \iff \tau^{-1} \leq \pi^{-1}.
\]
The \emph{reverse-complement} of $\pi$, i.e., $\pi^\mathrm{rc} = \desc{n} \circ \pi \circ \desc{n}$, equals the conjugate of $\pi$ with respect to the descending permutation $\desc{n}$ (note that $\desc{n}^{-1} = \desc{n}$).
Consequently, for any $S \subseteq \symm{n}$, the statement $T = \Compn[m]{S}$ is true if and only if the statement $\desc{m} T \desc{m} = \Compn[m]{\desc{n} S \desc{n}}$ is true.
Similarly, $T = \Patl{S}$ if and only if $\desc{\ell} T \desc{\ell} = \Patl{\desc{n} S \desc{n}}$.
Whenever we have proved a statement of the form $T = \Compn[m]{S}$ or $T = \Patl{S}$, we obtain also the statement $\desc{m} T \desc{m} = \Compn[m]{\desc{n} S \desc{n}}$ or $\desc{\ell} T \desc{\ell} = \Patl{\desc{n} S \desc{n}}$, whichever applies, for free, and we indicate this by saying that we ``consider conjugates with respect to the descending permutation''.

In order to simplify notation, for any $\pi \in \symm{n}$ and $i \in \nset{n}$, we write $\subpermsimple{\pi}{i} := \subperm{\pi}{\nset{n} \setminus \{i\}}$.
We say that integers $y$ and $z$ are \emph{consecutive modulo $n$} if $y + 1 \equiv z \pmod{n}$ or $y - 1 \equiv z \pmod{n}$.

\begin{fact}
\label{fact:dihedral}
Let $\pi \in \symm{n}$. Then $\pi(x)$ and $\pi(x+1)$ are consecutive modulo $n$ for all $x \in \nset{n-1}$ if and only if $\pi \in \dihed{n}$.
\end{fact}

\begin{lemma}[{Atkinson, Beals~\cite[Lemma~1]{AtkBea2001}}]
\label{lem:AtkBea-L11}
Let $\pi \in \symm{n}$ and let $i \in \nset{n-1}$.
Let $j := \pi(i)$ and $k := \pi(i+1)$, and let $\gamma := \subpermsimple{\pi}{i+1} \circ (\subpermsimple{\pi}{i})^{-1}$.
\textup{(}Note that $\gamma \in \gPatl[n-1]{\pi}$.\textup{)}
Then $\gamma$ is the cycle $(j \; j+1 \; \cdots \; k-1)$ if $j < k$ and the cycle $(k \; k+1 \; \cdots \; j-1)$ if $j > k$.
In addition, if $j$ and $k$ are not consecutive modulo $n$, then $\gPatl[n-2]{\pi}$ contains a transposition of the form $(t \;\; t+1)$.
\end{lemma}

Before going to the main results, let us make a couple of simple observations that reveal the importance of descending permutations and natural cycles for group classes.
In many of the results that follow, we make statements of the form $\Compn[m]{G} \leq \gensg{\desc{m}}$.
It is then clear from Lemma~\ref{lem:descending} which one of the two subgroups of $\gensg{\desc{m}}$ (the trivial subgroup or $\gensg{\desc{m}}$ itself) equals $\Compn[m]{G}$.

\begin{lemma}
\label{lem:descending}
Let $n, m \in \IN_+$ with $n \leq m$.
Let $G \leq \symm{n}$.
Then $\desc{m} \in \Compn[m]{G}$ if and only if $\desc{n} \in G$.
\end{lemma}

\begin{proof}
As noted in Example~\ref{ex:asc-desc-cycle}, $\Patl[n]{\desc{m}} = \{\desc{n}\}$, so $\desc{m} \in \Compn[m]{G}$ if and only if $\desc{n} \in G$.
\end{proof}

\begin{lemma}
\label{lem:ZD}
For $n \geq 2$ and $i \in \nset{n-1}$,
\begin{align*}
& \Patl[n-1]{\natcycle{n}^i} = \{\natcycle{n-1}^i, \natcycle{n-1}^{i-1}\},
&& \Patl[n-1]{(\natcycle{n}^i \circ \desc{n})} = \{\natcycle{n-1}^i \circ \desc{n-1}, \natcycle{n-1}^{i-1} \circ \desc{n-1}\},
\\
& \Patl[n-1]{\cycl{n}} = \cycl{n-1},
&& \Patl[n-1]{\dihed{n}} = \dihed{n-1}.
\end{align*}
\end{lemma}

\begin{proof}
Recall from Example~\ref{ex:asc-desc-cycle} that $\Patl[n-1]{\asc{n}} = \{\asc{n-1}\}$ and $\Patl[n-1]{\desc{n}} = \{\desc{n-1}\}$.
It is easy to verify that for $i \in \nset{n-1}$, $\Patl[n-1]{\natcycle{n}^i} = \{\natcycle{n-1}^i, \natcycle{n-1}^{i-1}\}$ and $\Patl[n-1]{(\natcycle{n}^i \circ \desc{n})} = \{\natcycle{n-1}^i \circ \desc{n-1}, \natcycle{n-1}^{i-1} \circ \desc{n-1}\}$.
Consequently,
$\Patl[n-1]{\cycl{n}} = \cycl{n-1}$ and
$\Patl[n-1]{\dihed{n}} = \dihed{n-1}$,
as claimed.
\end{proof}

\begin{lemma}
\label{lem:cyclic-dihedral}
Let $G \leq \symm{n}$.
Then the following statements hold.
\begin{enumerate}[label=\upshape{(\roman*)}]
\item\label{cd1} The following statements are equivalent.
\begin{enumerate}[label=\upshape{(\alph*)}]
\item $\cycl{n} \leq G$.
\item $\cycl{n+1} \leq \Compn[n+1]{G}$.
\item $\Compn[n+1]{G}$ contains a permutation $\pi \in \cycl{n+1} \setminus \{\asc{n+1}\}$.
\end{enumerate}
\item\label{cd2} The following statements are equivalent.
\begin{enumerate}[label=\upshape{(\alph*)}]
\item $\dihed{n} \leq G$.
\item $\dihed{n+1} \leq \Compn[n+1]{G}$.
\item $\Compn[n+1]{G}$ contains a permutation $\pi \in \dihed{n+1} \setminus (\cycl{n+1} \cup \{\desc{n+1}\})$.
\end{enumerate}
\end{enumerate}
\end{lemma}

\begin{proof}
\ref{cd1}
The implication $\text{(a)} \Rightarrow \text{(b)}$ holds by Lemma~\ref{lem:ZD}, and $\text{(b)} \Rightarrow \text{(c)}$ is obvious.
It remains to show that (c) implies (a).
Assume that $\pi \in \Compn[n+1]{G}$ and $\pi \in \cycl{n+1} \setminus \{\asc{n+1}\}$.
Then $\pi = \natcycle{n+1}^i$ for some $i \in \{1, \dots, n\}$.
By Lemma~\ref{lem:ZD}, the $n$\hyp{}patterns of $\natcycle{n+1}^i$ are $\natcycle{n}^i$ and $\natcycle{n}^{i-1}$, and they are members of $G$.
Consequently, $G$ contains $\natcycle{n}^i \circ (\natcycle{n}^{i-1})^{-1} = \natcycle{n}$.
We conclude that $\cycl{n} = \gensg{\natcycle{n}} \leq G$.

\ref{cd2}
As above, the implication $\text{(a)} \implies \text{(b)}$ holds by Lemma~\ref{lem:ZD}, and $\text{(b)} \implies \text{(c)}$ is obvious.
It remains to show that (c) implies (a).
Assume that $\pi \in \Compn[n+1]{G}$ and $\pi \in \dihed{n+1} \setminus (\cycl{n+1} \cup \{\desc{n+1}\})$.
Then $\pi = \natcycle{n+1}^i \circ \desc{n+1}$ for some $i \in \{1, \dots, n\}$.
By Lemma~\ref{lem:ZD}, the $n$\hyp{}patterns of $\natcycle{n+1}^i \circ \desc{n+1}$ are $\natcycle{n}^i \circ \desc{n}$ and $\natcycle{n}^{i-1} \circ \desc{n}$, and they are members of $G$.
Consequently, $G$ contains
$(\natcycle{n}^i \circ \desc{n}) \circ (\natcycle{n}^{i-1} \circ \desc{n})^{-1} = \natcycle{n}$
and hence also
$(\natcycle{n}^i)^{-1} \circ (\natcycle{n}^i \circ \desc{n}) = \desc{n}$.
We conclude that $\dihed{n} = \gensg{\natcycle{n}, \desc{n}} \leq G$.
\end{proof}

As the first point in our classification of groups, we have a few special groups:
the symmetric and trivial groups, and the group generated by the descending permutation.
These just give rise to some of the stable sequences of Theorem~\ref{thm:AtkBea-asymptotic}.

\begin{theorem}
\label{thm:symmetric-trivial}
The following statements hold for all $n \in \IN_+$.
\begin{enumerate}[label=\upshape{(\roman*)}]
\item\label{st:i} $\Compn[n+1]{\symm{n}} = \symm{n+1}$.
\item If $n \geq 2$, then $\Compn[n+1]{\{\asc{n}\}} = \{\asc{n+1}\}$.
\item If $n \geq 3$, then $\Compn[n+1]{\gensg{\desc{n}}} = \gensg{\desc{n+1}}$.
\end{enumerate}
\end{theorem}

\begin{proof}
Statement \ref{st:i} is trivial.
The other statements are easy to verify.
\end{proof}


\section{Alternating groups}
\label{sec:alternating}

The second class of groups we investigate are the alternating groups.
Recall that the \emph{alternating group \textup{(}of degree $n$\textup{)},} denoted $\alt{n}$, is the group of all even permutations of the set $\nset{n}$.
A permutation is \emph{even} if it is the product of an even number of transpositions; otherwise it is \emph{odd.}
The oddness or evenness of a permutation is referred to as its \emph{parity.}

\begin{lemma}
\label{lem:parities}
Let $\pi \in \symm{n}$.
Then $\pi$ and $\subpermsimple{\pi}{1}$ have the same parity if and only if $\pi(1)$ is odd.
\end{lemma}

\begin{proof}
Recall that a couple $\{i, j\} \subseteq \nset{n}$ is a \emph{reversal} of a permutation $\pi \in \symm{n}$ if $i < j$ and $\pi(i) > \pi(j)$.
Let us denote the number of reversals of $\pi$ by $\reversals(\pi)$.
The parity of a permutation equals the parity of the number of its reversals.
It is clear that for all $\{i, j\} \subseteq \nset{n-1}$, $\{i, j\}$ is a reversal of $\subpermsimple{\pi}{1}$ if and only if $\{i + 1, j + 1\}$ is a reversal of $\pi$.
Furthermore, there are exactly $\pi(1) - 1$ reversals of $\pi$ of the form $\{1, j\}$ for some $j$.
Thus $\reversals(\pi) = \reversals(\subpermsimple{\pi}{1}) + \pi(1) - 1$.
From this equality it follows that $\reversals(\pi) \equiv \reversals(\subpermsimple{\pi}{1}) \pmod{2}$ if and only if $\pi(1) \equiv 1 \pmod{2}$.
\end{proof}

For $n \geq 2$, denote by $\compalt{n}$ the set of all odd permutations in $\symm{n}$, i.e., $\compalt{n} := \symm{n} \setminus \alt{n}$.
Let $\oddeven{n}$ be the partition of $\nset{n}$ into odd and even numbers, i.e.,
\[
\oddeven{n} :=
\bigl\{ \{i \in \nset{n} \mid \text{$i$ odd}\}, \{i \in \nset{n} \mid \text{$i$ even}\} \bigr\}.
\]
Denote by $\symm{\oddeven{n}}$ the set of permutations in $\symm{n}$ that preserve the blocks of $\oddeven{n}$, i.e., $\pi \in \symm{\oddeven{n}}$
if and only if $\pi(i) \equiv i \pmod{2}$ for all $i \in \nset{n}$.
Denote by $\interchange{\oddeven{n}}$ the set of permutations in $\symm{n}$ that interchange the blocks of $\oddeven{n}$, i.e., $\pi \in \interchange{\oddeven{n}}$
if and only if $\pi(i) \not\equiv i \pmod{2}$ for all $i \in \nset{n}$.
Note that $\interchange{\oddeven{n}} = \emptyset$ if $n$ is odd.
Let
\[
\grOE{n} := (\symm{\oddeven{n}} \cap \alt{n}) \cup (\interchange{\oddeven{n}} \cap \compalt{n}).
\]
It is not difficult to verify that $\grOE{n}$ is a permutation group.

\begin{theorem}
\label{thm:Comp-alt}
For $n \geq 1$,
$\Compn[n+1]{\alt{n}} = \grOE{n+1}$.
\end{theorem}

\begin{proof}
For $i \in \nset{n}$, let $\gamma(\pi,i) := \subpermsimple{\pi}{i+1} \circ \subpermsimple{\pi}{i}^{-1}$ as in Lemma~\ref{lem:AtkBea-L11}.
Then $\gamma(\pi,i)$ is a cycle of length $m(\pi,i) := \card{\pi(i) - \pi(i+1)}$.

Let $\pi \in \Compn[n+1]{\alt{n}}$.
Since $\pi \in \Compn[n+1]{\alt{n}}$, the $n$\hyp{}patterns $\subpermsimple{\pi}{i}$ and $\subpermsimple{\pi}{i+1}$ of $\pi$ are even; hence also $\gamma(\pi,i)$ is an even permutation, so every $m(\pi,i)$ must be odd ($i \in \nset{n}$).
This implies that the sequence $\pi(1), \pi(2), \dots, \pi(n+1)$ alternates between odd and even numbers.
Hence either $\pi$ maps odd numbers to odd numbers and even numbers to even numbers, or $\pi$ interchanges the set of odd numbers with the set of even numbers.
In other words, $\pi \in \symm{\oddeven{n+1}} \cup \interchange{\oddeven{n+1}}$.

Since $\subpermsimple{\pi}{1}$ is an even permutation, it follows from Lemma~\ref{lem:parities} that $\pi(1)$ is odd if and only if $\pi$ is an even permutation.
This means that $\pi \in \symm{\oddeven{n+1}}$ if and only if $\pi \in \alt{n}$.
We conclude that $\pi \in \grOE{n+1}$.
Therefore we have proved the inclusion $\Compn[n+1]{\alt{n}} \subseteq \grOE{n+1}$.

In order to show the converse inclusion, assume that $\pi \in \grOE{n+1}$.
Then the sequence $\pi(1), \pi(2), \dots, \pi(n+1)$ alternates between odd and even numbers.
Therefore the cycles $\gamma(\pi,i)$ are all of odd length, i.e., they are all even permutations.
It follows that all $n$\hyp{}patterns of $\pi$ have the same parity.
If $\pi \in \symm{\oddeven{n+1}} \cap \alt{n+1}$, then $\pi$ is even and $\pi(1)$ is odd, so Lemma~\ref{lem:parities} implies that also $\subpermsimple{\pi}{1}$ is even.
If $\pi \in \interchange{\oddeven{n+1}} \cap \compalt{n+1}$, then $\pi$ is odd and $\pi(1)$ is even, so Lemma~\ref{lem:parities} implies that $\subpermsimple{\pi}{1}$ is even also in this case.
Thus all $n$\hyp{}patterns of $\pi$ are even, so $\pi \in \Compn[n+1]{\alt{n}}$.
This shows that $\grOE{n+1} \subseteq \Compn[n+1]{\alt{n}}$.
\end{proof}

Let us find out how the level sequence of $\alt{n}$ continues after $\grOE{n+1}$.

\begin{lemma}
\label{lem:desc-in-alt}
The descending permutation $\desc{n}$ is a member of the alternating group $\alt{n}$ if and only if $n \equiv 0 \pmod{4}$ or $n \equiv 1 \pmod{4}$.
\end{lemma}

\begin{proof}
In the case where $n$ is even, say $n = 2k$, it holds that the descending permutation $\desc{n} = (1 \; n) (2 \; n-1) \cdots (k \; k+1)$ is even if and only if $k$ is even, i.e., $n = 2k \equiv 0 \pmod{4}$.
In the case where $n$ is odd, say $n = 2k + 1$, it holds that $\desc{n} = (1 \; n) (2 \; n-1) \cdots (k \; k+2)$ is even if and only if $k$ is even, i.e., $n = 2k + 1 \equiv 1 \pmod{4}$.
\end{proof}

\begin{theorem}
\label{thm:Comp-alt-n+2}
For $n \geq 2$,
\[
\Compn[n+2]{\alt{n}} =
\begin{cases}
\gensg{\desc{n+2}}, & \text{if $n \equiv 0 \pmod{4}$,} \\
\dihed{n+2},        & \text{if $n \equiv 1 \pmod{4}$,} \\
\{\asc{n+2}\},      & \text{if $n \equiv 2 \pmod{4}$,} \\
\cycl{n+2},         & \text{if $n \equiv 3 \pmod{4}$.} \\
\end{cases}
\]
\end{theorem}

\begin{proof}
Observe first that $\Compn[n+2]{\alt{n}} \leq \dihed{n+2}$.
For, assume that $\pi \in \symm{n+2} \setminus \dihed{n+2}$.
By Fact~\ref{fact:dihedral} there exists $x \in \nset{n+1}$ such that $\pi(x)$ and $\pi(x+1)$ are not consecutive modulo $n+2$, so by Lemma~\ref{lem:AtkBea-L11}, $\gPatl[n]{\pi}$ contains a transposition, an odd permutation; hence $\pi \notin \Compn[n+2]{\alt{n}}$.

Assume first that $n$ is even.
Then $\natcycle{n} \notin \alt{n}$, which implies
\begin{gather*}
(\Compn[n+2]{\alt{n}}) \cap (\cycl{n+2} \setminus \{\asc{n+2}\}) = \emptyset, \\
(\Compn[n+2]{\alt{n}}) \cap (\dihed{n+2} \setminus (\cycl{n+2} \cup \{\desc{n+2}\})) = \emptyset
\end{gather*}
by Lemma~\ref{lem:cyclic-dihedral}.
Consequently, $\Compn[n+2]{\alt{n}} \subseteq \{\asc{n+2}, \desc{n+2}\}$.
By Lemma~\ref{lem:descending}, we have $\desc{n+2} \in \Compn[n+2]{\alt{n}}$ if and only if $\desc{n} \in \alt{n}$.
It then follows from Lemma~\ref{lem:desc-in-alt} that $\Compn[n+2]{\alt{n}} = \{\asc{n+2}, \desc{n+2}\} = \gensg{\desc{n+2}}$ if $n \equiv 0 \pmod{4}$ and $\Compn[n+2]{\alt{n}} = \{\asc{n+2}\}$ if $n \equiv 2 \pmod{4}$.

Assume then that $n$ is odd.
Then $\natcycle{n} \in \alt{n}$.
If $n \equiv 1 \pmod{4}$, then $\desc{n} \in \alt{n}$ by Lemma~\ref{lem:desc-in-alt},
and it follows from Lemma~\ref{lem:cyclic-dihedral} that
$\dihed{n+2} \leq \Compn[n+2]{\alt{n}}$;
hence $\Compn[n+2]{\alt{n}} = \dihed{n+2}$.
If $n \equiv 3 \pmod{4}$, then $\desc{n} \notin \alt{n}$ by Lemma~\ref{lem:desc-in-alt},
which implies $\desc{n+2} \notin \Compn[n+2]{\alt{n}}$ by Lemma~\ref{lem:descending}.
Moreover, $\cycl{n+2} \leq \Compn[n+2]{\alt{n}}$ and
$(\Compn[n+2]{\alt{n}}) \cap (\dihed{n+2} \setminus (\cycl{n+2} \cup \{\desc{n+2}\})) = \emptyset$
by Lemma~\ref{lem:cyclic-dihedral}.
Hence $\Compn[n+2]{\alt{n}} = \cycl{n+2}$.
\end{proof}


\section{Overgroups of the natural cyclic group}
\label{sec:cycle}

In this section, we consider permutation groups containing the natural cycle $\natcycle{n}$, in other words, overgroups of the natural cyclic group $\cycl{n}$.
We will make use of a well\hyp{}known technical lemma that describes the group generated by the natural cycle and another cycle.

\begin{lemma}[{Isaacs, Zieschang~\cite[Theorem~B]{IsaZie}}]
\label{lem:IsaZie}
Let $m$ and $n$ be integers such that $1 < m < n$, and let $\tau_m := (1 \; 2 \; \cdots \; m) \in \symm{n}$.
Then $\gensg{\natcycle{n}, \tau_m} = \symm{n}$ unless both $m$ and $n$ are odd, in which case $\gensg{\natcycle{n}, \tau_m} = \alt{n}$.
\end{lemma}

\begin{lemma}
\label{lem:cycle-dihedr-alt}
Let $G \leq \symm{n}$.
If $\natcycle{n} \in G$ and $\Compn[n+1]{G} \nleq \dihed{n+1}$, then $\alt{n} \leq G$.
\end{lemma}

\begin{proof}
According to Fact~\ref{fact:dihedral}, $\Compn[n+1]{G}$ contains a permutation $\pi$ such that $\pi(t)$ and $\pi(t+1)$ are not consecutive modulo $n$ for some $t$.
By Lemma~\ref{lem:AtkBea-L11}, $\gPatl[n]{\Compn[n+1]{G}}$ contains a cycle of the form $(y \; y+1 \; \cdots \; y + m - 1)$ for some $y \in \nset{n}$ and $1 < m < n$.
Since $\gPatl[n]{\Compn[n+1]{G}}$ also contains $\natcycle{n}$ by Lemma~\ref{lem:cyclic-dihedral} and Example~\ref{ex:asc-desc-cycle}, we have
\[
(\natcycle{n}^{y-1})^{-1} \circ (y \; y+1 \; \cdots \; y + m - 1) \circ \natcycle{n}^{y-1} 
= (1 \; 2 \; \cdots \; m)
\in \gPatl[n]{\Compn[n+1]{G}}.
\]
It follows from Lemma~\ref{lem:IsaZie} that $\alt{n} \leq \gPatl[n]{\Compn[n+1]{G}} \leq G$.
\end{proof}

We are now ready to describe $\Compn[n+1]{G}$ for any overgroup $G$ of $\cycl{n}$.
We exclude the symmetric group $\symm{n}$ and the alternating group $\alt{n}$, because they have already been taken care of by Theorems~\ref{thm:symmetric-trivial} and \ref{thm:Comp-alt}, respectively.

\begin{theorem}
\label{thm:Comp-cycle}
Assume that $G \leq \symm{n}$, $G \notin \{\symm{n}, \alt{n}\}$, and $\natcycle{n} \in G$.
\begin{enumerate}[label=\upshape{(\roman*)}]
\item\label{prop:Comp-cycle:dihed} If $\dihed{n} \leq G$, then $\Compn[n+1]{G} = \dihed{n+1}$.
\item\label{prop:Comp-cycle:other} If $\dihed{n} \nleq G$, then $\Compn[n+1]{G} = \cycl{n+1}$.
\end{enumerate}
\end{theorem}

\begin{proof}
Since $\natcycle{n} \in G$, we have $\cycl{n} \leq G$, so Lemma~\ref{lem:cyclic-dihedral} implies $\cycl{n+1} \leq \Compn[n+1]{G}$.
Since $\alt{n} \nleq G$, Lemma~\ref{lem:cycle-dihedr-alt} implies $\Compn[n+1]{G} \leq \dihed{n+1}$.
If $\dihed{n} \leq G$, then $\dihed{n+1} \leq \Compn[n+1]{G}$ by Lemma~\ref{lem:cyclic-dihedral}; hence $\Compn[n+1]{G} = \dihed{n+1}$.
If $\dihed{n} \nleq G$, then $\Compn[n+1]{G} \cap (\dihed{n+1} \setminus (\cycl{n+1} \cup \{\desc{n+1}\})) = \emptyset$ by Lemma~\ref{lem:cyclic-dihedral}.
Moreover, since $\desc{n} \notin G$, Lemma~\ref{lem:descending} implies $\desc{n+1} \notin \Compn[n+1]{G}$.
Consequently, $\Compn[n+1]{G} = \cycl{n+1}$.
\end{proof}

Compare the above with the description of group classes in which every level is transitive (Theorem~\ref{thm:AtkBea-T4}).
In fact, with the results we have obtained so far, we can describe precisely the anomalous groups, which play an important role in Theorem~\ref{thm:AtkBea-T4}.

\begin{proposition}
\label{prop:anomalous}
A permutation group $G \leq \symm{n}$ is anomalous if and only if
$G = \cycl{n}$, $G = \dihed{n}$, or
$\cycl{n} \leq G \leq \grOE{n}$.
\textup{(}Note that the last case is possible only for even $n$.\textup{)}
\end{proposition}

\begin{proof}
Assume first that $G$ is anomalous.
Then $\cycl{n} \leq G$ and $\gPatl[n-1]{G} \neq \symm{n-1}$ hold by definition.
By Example~\ref{ex:asc-desc-cycle}, $\Patl[n-1]{\natcycle{n}} = \{\asc{n-1}, \natcycle{n-1}\}$, so $\cycl{n-1} \leq \gPatl[n-1]{G}$.
Since $\gPatl[n-1]{G} \neq \symm{n-1}$, it follows from Theorems~\ref{thm:Comp-alt} and \ref{thm:Comp-cycle} that $\Compn[n]{\gPatl[n-1]{G}}$ must be one of the groups $\cycl{n}$, $\dihed{n}$ and $\grOE{n}$.
Since the inclusion $G \leq \Compn[n]{\gPatl[n-1]{G}}$ holds by the general properties of Galois connections, we have either $G = \cycl{n}$, $G = \dihed{n}$, or $\cycl{n} \leq G \leq \grOE{n}$.

Assume then that $G = \cycl{n}$, $G = \dihed{n}$, or $\cycl{n} \leq G \leq \grOE{n}$.
Since $\gPatl[n-1]{\cycl{n}} = \cycl{n-1}$ and $\gPatl[n-1]{\dihed{n}} = \dihed{n-1}$ by Lemma~\ref{lem:ZD}, we have $\gPatl[n-1]{G} \neq \symm{n-1}$ in the cases where $G = \cycl{n}$ or $G = \dihed{n}$.
If $\cycl{n} \leq G \leq \grOE{n}$, then
\[
\gPatl[n-1]{G}
\leq \gPatl[n-1]{\grOE{n}}
= \gPatl[n-1]{\Compn[n]{\alt{n-1}}}
\leq \alt{n-1}
\]
by Theorem~\ref{thm:Comp-alt} and by the fact that $\gPatl[n-1]{\Compn[n]{}}$ is a kernel operation.
Hence $\gPatl[n-1]{G} \neq \symm{n-1}$ also in this case, and we conclude that $G$ is anomalous.
\end{proof}


\section{Intransitive groups}
\label{sec:intransitive}

\subsection{Partitions}

Before moving on to the next class of permutation groups,
let us recall some terminology related to partitions of sets.
Let $\Pi$ be a partition of $\nset{n}$.
The blocks of $\Pi$ are also referred to as \emph{$\Pi$\hyp{}blocks.}
The $\Pi$\hyp{}block containing the element $x$ is denoted by $\eqclass{x}{\Pi}$.
We write $x \sim_\Pi y$ to denote the fact that $x$ and $y$ are members of the same $\Pi$\hyp{}block, i.e., $\eqclass{x}{\Pi} = \eqclass{y}{\Pi}$.
Blocks of cardinality $1$ are \emph{trivial.}
A partition is \emph{trivial} if all its blocks are trivial
or it has only one block.
A partition is called an \emph{interval partition} if its blocks are intervals.
We say that blocks $B$ and $B'$ are \emph{consecutive} if there exist $t \in \nset{n}$ such that $t-1 \in B$ and $t \in B'$.
A partition $\Pi$ is a \emph{refinement} of a partition $\Gamma$, and $\Gamma$ is a \emph{coarsening} of $\Pi$, denoted $\Pi \sqsubseteq \Gamma$, if every $\Pi$\hyp{}block is a subset of some $\Gamma$\hyp{}block.
The set $\Part(n)$ of all partitions of $\nset{n}$, ordered by the refinement relation $\sqsubseteq$, constitutes a lattice, and we denote by $\Pi \vee \Gamma$ and $\Pi \wedge \Gamma$ the finest common coarsening of $\Pi$ and $\Gamma$ and the coarsest common refinement of $\Pi$ and $\Gamma$, respectively.
The set $\IntPart(n)$ of all interval partitions of $\nset{n}$ is a sublattice of $\Part(n)$.

\subsection{Intransitive groups}
Let $G$ be a permutation group, a subgroup of $\symm{n}$, and let $i \in \nset{n}$.
The \emph{$G$\hyp{}orbit} (or simply \emph{orbit}) of $i$ is the set $\{\pi(i) \mid \pi \in G\}$.
The set of all $G$\hyp{}orbits is denoted by $\orbits G$; it is a partition of $\nset{n}$.
A permutation group is \emph{transitive} if it has only one orbit; otherwise it is \emph{intransitive.}
In this section, we investigate the level sequences of intransitive groups.

At this point, let us introduce an important family of groups that will recur in what follows.

\begin{definition}
Let $\Pi$ be a partition of $\nset{n}$, and define $\symm{\Pi}$ to be the set of all $n$\hyp{}permutations that map each $\Pi$\hyp{}block onto itself, i.e.,
\[
\symm{\Pi} := \{\pi \in \symm{n} \mid \text{$\pi(B) = B$ for every block $B \in \Pi$}\}.
\]
It is clear that $\symm{\Pi}$ is a group, and if $\Pi$ has at least two blocks, then $\symm{\Pi}$ is intransitive.
Note that the group $\symm{n}^{a,b}$ of Theorem~\ref{thm:AtkBea-asymptotic} equals $\symm{\Pi_n^{a,b}}$, where
\[
\Pi_n^{a,b} := \{\interval{1}{a}, \desc{n}(\interval{1}{b})\} \cup \{\{i\} \mid a < i < \desc{n}(b)\}.
\]
\end{definition}

If $G$ is a intransitive group, then $G \leq \symm{\orbits G}$.
Moreover, $\orbits G$ is the finest partition $\Pi$ of $\nset{n}$ such that $G \leq \symm{\Pi}$.
Therefore, in order to describe $\Compn[n+i]{G}$, or at least to bound it from above, it is useful to determine $\Compn[n+1]{\symm{\orbits G}}$.
Therefore, we focus first on groups of the form $\symm{\Pi}$ where $\Pi$ is a partition of $\nset{n}$.

\begin{definition}
For an arbitrary partition $\Pi$ of $\nset{n}$, let $\maxintervals{\Pi}$ be the coarsest interval partition of $\nset{n}$ that is a refinement of $\Pi$, and let $\maxintervalsmiddle{\Pi} := \maxintervals{\Pi} \setminus \{\eqclass{1}{\maxintervals{\Pi}}, \eqclass{n}{\maxintervals{\Pi}}\}$.
Define the partition $\Pi'$ of the set $\nset{n+1}$ as $\Pi' := \maxintervals{\Pi}^1 \wedge \maxintervals{\Pi}^{n+1}$, where $\maxintervals{\Pi}^1$ and $\maxintervals{\Pi}^{n+1}$ are the partitions of $\nset{n+1}$ given by
\begin{align*}
\maxintervals{\Pi}^1 &:= \{B + 1 \mid B \in \maxintervals{\Pi}, \; 1 \notin B\} \cup \{((\eqclass{1}{\maxintervals{\Pi}}) + 1) \cup \{1\}\}, \\
\maxintervals{\Pi}^{n+1} &:= \{B \mid B \in \maxintervals{\Pi}, \; n \notin B\} \cup \{\eqclass{n}{\maxintervals{\Pi}} \cup \{n + 1\}\}.
\end{align*}
In other words, $\maxintervals{\Pi}^{n+1}$ is obtained from $\maxintervals{\Pi}$ by adjoining the element $n + 1$ to the block $\eqclass{n}{\maxintervals{\Pi}}$,
and $\maxintervals{\Pi}^1$ is obtained from $\maxintervals{\Pi}$ by adding $1$ to all elements and then adjoining the element $1$ to the set $(\eqclass{1}{\maxintervals{\Pi}}) + 1$.
Note that $\maxintervals{\Pi}^1$ and $\maxintervals{\Pi}^{n+1}$ are interval partitions by construction, and, consequently, their meet, $\Pi'$, is also an interval partition.

We write $\Pi^{(1)} := \Pi'$ and $\Pi^{(i+1)} := (\Pi^{(i)})'$ for $i \geq 1$.
We also write $\desc{n}(\Pi) := \{\desc{n}(B) \mid B \in \Pi\}$.
\end{definition}

\begin{example}
Let
$\Pi = \bigl\{ \{1,2,3,7,8,9,10\}, \{4,5,6,12,13,14\}, \{11\} \bigr\}$.
Then
\begin{align*}
\maxintervals{\Pi} &= \bigl\{ \{1,2,3\}, \{4,5,6\}, \{7,8,9,10\}, \{11\}, \{12,13,14\} \bigr\}, \\
\maxintervals{\Pi}^1 &= \bigl\{ \{1,2,3,4\}, \{5,6,7\}, \{8,9,10,11\}, \{12\}, \{13,14,15\} \bigr\}, \\
\maxintervals{\Pi}^{n+1} &= \bigl\{ \{1,2,3\}, \{4,5,6\}, \{7,8,9,10\}, \{11\}, \{12,13,14,15\} \bigr\}, \\
\Pi^{(1)} &= \Pi' = \bigl\{ \{1,2,3\}, \{4\}, \{5,6\}, \{7\}, \{8,9,10\}, \{11\}, \{12\}, \{13,14,15\} \bigr\}, \\
\Pi^{(2)} &= \bigl\{ \{1,2,3\}, \{4\}, \{5\}, \{6\}, \{7\}, \{8\}, \{9,10\}, \{11\}, \{12\}, \{13\}, \{14,15,16\} \bigr\}, \\
\Pi^{(3)} &= \bigl\{ \{1,2,3\}, \{4\}, \{5\}, \{6\}, \{7\}, \{8\}, \{9\}, \{10\}, \{11\}, \{12\}, \{13\}, \{14\}, \\ &\phantom{{}=\big\{} \{15,16,17\} \bigr\}.
\end{align*}
\end{example}

\begin{fact}
\label{fact:IPi}
It is easy to verify that if $\interval{a}{b} \in \maxintervals{\Pi}$, then the following statements hold:
\begin{enumerate}
\item If $a = 1$ and $b < n$, then $\interval{a}{b} \in \Pi'$.
\item If $a > 1$ and $b < n$, then $\{a\} \in \Pi'$; moreover, if $a < b$, then also $\interval{a+1}{b} \in \Pi'$.
\item If $a > 1$ and $b = n$, then $\{a\}, \interval{a+1}{n+1} \in \Pi'$.
\item If $a = 1$ and $b = n$, then $\interval{1}{n+1} \in \Pi'$.
\end{enumerate}
\end{fact}

\begin{lemma}[{Atkinson, Beals~\cite[Lemma~3]{AtkBea2001}}]
\label{lem:AtkBea-L6}
Let $\Pi$ be a partition of $\nset{n}$, and let $\pi \in \Compn[n+1]{\symm{\Pi}}$.
Assume that $t \in \nset{n}$ is a number such that $t-1 \not\sim_\Pi t$.
Then either
\begin{enumerate}[label=\upshape{(\alph*)}]
\item\label{AB-L6-1-a} $\pi(t) = t$ and $\pi$ maps each one of the intervals $\interval{1}{t-1}$ and $\interval{t+1}{n+1}$ onto itself, or
\item\label{AB-L6-1-b} $\pi(t) = \desc{n+1}(t) = n - t + 2$, $\pi(n - t + 2) = \desc{n+1}(n - t + 2) = t$, and $\pi$ interchanges the interval $\interval{1}{t-1}$ with $\desc{n+1}(\interval{1}{t-1}) = \interval{n - t + 3}{n+1}$ and the interval $\interval{t+1}{n+1}$ with $\desc{n+1}(\interval{t+1}{n+1}) = \interval{1}{n - t + 1}$.
\end{enumerate}
\end{lemma}

\begin{lemma}
\label{lem:AtkBea-L6-2}
If $\pi \in \symm{n+1}$ and $t, t' \in \nset{n}$ are numbers satisfying condition \ref{AB-L6-1-a} or \ref{AB-L6-1-b} of Lemma~\ref{lem:AtkBea-L6}, then both $t$ and $t'$ satisfy the same condition \textup{(}\ref{AB-L6-1-a} or \ref{AB-L6-1-b}\textup{)}.
\end{lemma}

\begin{proof}
Suppose, to the contrary, that $t$ satisfies condition \ref{AB-L6-1-a} and $t'$ satisfies condition \ref{AB-L6-1-b}.
Then $\pi(\interval{1}{t-1}) = \interval{1}{t-1}$ and $\pi(\interval{1}{t'-1}) = \interval{n-t'+3}{n+1}$.
One of the intervals $\interval{1}{t-1}$ and $\interval{1}{t'-1}$ is included in the other, but their images under $\pi$, namely $\interval{1}{t-1}$ and $\interval{n-t'+3}{n+1}$, are incomparable under the subset inclusion order, which is clearly impossible, so we have reached a contradiction.
\end{proof}

\begin{lemma}
\label{lem:CompSPi}
Let $\Pi$ be a partition of $\nset{n}$.
Then $\Compn[n+1]{\symm{\Pi}} \subseteq \symm{\Pi'} \cup \desc{n+1} \symm{\Pi'}$.
\end{lemma}

\begin{proof}
Let $\pi \in \Compn[n+1]{\symm{\Pi}}$.
By Lemma~\ref{lem:AtkBea-L6-2}, all numbers $t \in \nset{n}$ with $t-1 \not\sim_\Pi t$ satisfy the same condition (\ref{AB-L6-1-a} or \ref{AB-L6-1-b}) of Lemma~\ref{lem:AtkBea-L6}.
Assume first that condition \ref{AB-L6-1-a} is satisfied by all such numbers $t$.
It follows from Fact~\ref{fact:IPi} and Lemma~\ref{lem:AtkBea-L6} that for each $\interval{a}{b} \in \maxintervals{\Pi}$, the following holds:
\begin{enumerate}
\item if $a = 1$ and $b < n$, then $\pi$ maps the interval $\interval{a}{b}$ onto itself;
\item if $a > 1$ and $b < n$, then $\pi(a) = a$; moreover, if $a < b$, then $\pi$ maps $\interval{a+1}{b}$ onto itself;
\item if $a > 1$ and $b = n$, then $\pi(a) = a$ and $\pi$ maps $\interval{a+1}{n+1}$ onto itself;
\item If $a = 1$ and $b = n$, then $\pi$ maps $\interval{1}{n+1}$ onto itself.
\end{enumerate}
In other words, $\pi$ maps every $\Pi'$\hyp{}block onto itself, that is, $\pi \in \symm{\Pi'}$.

Assume then that condition \ref{AB-L6-1-b} is satisfied.
Arguing as above, we see that $\pi$ maps each block $B$ of $\Pi'$ onto $\desc{n+1}(B)$.
Consequently, $\desc{n+1} \circ \pi$ maps every $\Pi'$\hyp{}block onto itself, that is, $\desc{n+1} \circ \pi \in \symm{\Pi'}$.
Therefore, $\pi \in \desc{n+1} \symm{\Pi'}$.
We conclude that $\Compn[n+1]{\symm{\Pi}} \subseteq \symm{\Pi'} \cup \desc{n+1} \symm{\Pi'}$.
\end{proof}

\begin{lemma}
\label{lem:Pi-prime-in-Comp-Pi}
Let $\Pi$ be a partition of $\nset{n}$.
Then $\symm{\Pi'} \subseteq \Compn[n+1]{\symm{\Pi}}$.
\end{lemma}

\begin{proof}
The claim clearly holds when $\Pi = \{\nset{n}\}$ (in which case $\Pi' = \{\nset{n+1}\}$), so we may assume that $\Pi$ has at least two blocks and hence so does $\Pi'$.
Let $\pi \in S_{\Pi'}$ and $i \in \nset{n+1}$.
We want to show that $\subpermsimple{\pi}{i} \in \symm{\Pi}$.
Let $B := \eqclass{i}{\Pi'}$.
We have $\pi(i) \in B$, because $\pi$ maps every $\Pi'$\hyp{}block onto itself.
Let $j \in \nset{n}$.
We shall consider different cases according to the position of $j$ relative to $i$ and $B$.

Case 1: $j < i$ and $j \notin B$.
Since the blocks of $\Pi'$ are intervals and $\Pi'$ has at least two blocks, it necessarily holds that $\pi(j) < \pi(i)$ and
$j \not\sim_{\Pi'} n+1$.
Therefore, $\subpermsimple{\pi}{i}(j) = \pi(j) \in \eqclass{j}{\Pi'} \subseteq \eqclass{j}{\Pi}$.

Case 2: $j \geq i$ and $j + 1 \notin B$. Then necessarily $\pi(j+1) > \pi(i)$ and
$1 \not\sim_{\Pi'} j+1$.
Therefore, $\subpermsimple{\pi}{i}(j) = \pi(j+1) - 1 \in \eqclass{(j+1)}{\Pi'} - 1 \subseteq \eqclass{j}{\Pi}$.

Case 3: $j < i$ and $j \in B$, or $j \geq i$ and $j + 1 \in B$.
In either case, both $j$ and $j+1$ are in $B$, because the block $B$ is an interval, and therefore also $\pi(j), \pi(j+1) \in B$.

Subase 3.1: $j < i$ and $j \in B$.
If $\pi(j) < \pi(i)$, then we also have $\pi(j) + 1 \in B$, i.e., $\pi(j) \in B - 1$.
Consequently, $\subpermsimple{\pi}{i}(j) = \pi(j) \in B \cap (B - 1)$.
If $\pi(j) > \pi(i)$, then we also have $\pi(j) - 1 \in B$.
Consequently, $\subpermsimple{\pi}{i}(j) = \pi(j) - 1 \in B \cap (B - 1)$.

Subcase 3.2: $j \geq i$ and $j + 1 \in B$.
If $\pi(j+1) < \pi(i)$, then we also have $\pi(j+1) + 1 \in B$, i.e., $\pi(j+1) \in B - 1$.
Consequently, $\subpermsimple{\pi}{i}(j) = \pi(j+1) \in B \cap (B - 1)$.
If $\pi(j+1) > \pi(i)$, then we also have $\pi(j+1) - 1 \in B$.
Consequently, $\subpermsimple{\pi}{i}(j) = \pi(j+1) - 1 \in B \cap (B - 1)$.

In both subcases (3.1 and 3.2), we have arrived to the conclusion that $\subpermsimple{\pi}{i}(j) \in B \cap (B - 1)$.
Depending on whether $n+1 \notin B$ or $1 \notin B$ (at least one of these must hold, because $\Pi'$ has at least two blocks), we have $\eqclass{j}{\Pi'} \subseteq \eqclass{j}{\Pi}$ or $\eqclass{(j+1)}{\Pi'} - 1 \subseteq \eqclass{j}{\Pi}$.
Since $B = \eqclass{j}{\Pi'} = \eqclass{(j+1)}{\Pi'}$, we have $B \cap (B - 1) \subseteq \eqclass{j}{\Pi}$. Consequently $\subpermsimple{\pi}{i}(j) \in \eqclass{j}{\Pi}$.

We conclude that $\subpermsimple{\pi}{i}(j) \in \eqclass{j}{\Pi}$ for all $j \in \nset{n}$, that is, $\subpermsimple{\pi}{i} \in \symm{\Pi}$.
From this it follows that $\pi \in \Compn[n+1]{\symm{\Pi}}$.
\end{proof}

\begin{lemma}
\label{lem:desc-CompSPi}
Let $\Pi$ be a partition of $\nset{n}$.
Then the following statements hold.
\begin{enumerate}[label=\upshape{(\roman*)}]
\item\label{dCSPi:1} If $\desc{n} \in \symm{\Pi}$, then $\desc{n+1} \symm{\Pi'} \subseteq \Compn[n+1]{\symm{\Pi}}$.
\item\label{dCSPi:2} If $\desc{n} \notin \symm{\Pi}$, then $\desc{n+1} \symm{\Pi'} \cap \Compn[n+1]{\symm{\Pi}} = \emptyset$.
\end{enumerate}
\end{lemma}

\begin{proof}
\ref{dCSPi:1}
Let $\pi \in \desc{n+1} \symm{\Pi'}$.
Then $\pi = \desc{n+1} \circ \tau$ for some $\tau \in \symm{\Pi'}$.
By Lemma~\ref{lem:Pi-prime-in-Comp-Pi}, $\tau \in \Compn[n+1]{\symm{\Pi}}$.
We also have $\desc{n+1} \in \Compn[n+1]{\symm{\Pi}}$ by Lemma~\ref{lem:descending}.
Since $\Compn[n+1]{\symm{\Pi}}$ is a permutation group by Proposition~\ref{prop:S-subgroup}, it follows that $\pi = \desc{n+1} \circ \tau \in \Compn[n+1]{\symm{\Pi}}$.

\ref{dCSPi:2}
Suppose, to the contrary, that there exists a permutation $\pi \in \desc{n+1} \symm{\Pi'} \cap \Compn[n+1]{\symm{\Pi}}$.
Then $\pi = \desc{n+1} \circ \tau$ for some $\tau \in \symm{\Pi'}$,
and we have $\tau \in \Compn[n+1]{\symm{\Pi}}$ by Lemma~\ref{lem:Pi-prime-in-Comp-Pi}.
Since $\Compn[n+1]{\symm{\Pi}}$ is a group, we also have that $\desc{n+1} = \pi \circ \tau^{-1} \in \Compn[n+1]{\symm{\Pi}}$.
Lemma~\ref{lem:descending} implies $\desc{n} \in \symm{\Pi}$, and we have reached a contradiction.
\end{proof}

\begin{lemma}
\label{lem:Pi-PiDeltan}
Let $\Pi$ be a partition of $\nset{n}$.
Then the following statements hold.
\begin{enumerate}[label=\upshape{(\roman*)}]
\item\label{item:desc-in-Pi}
$\desc{n} \in \symm{\Pi}$ if and only if $B = \desc{n}(B)$ for all $B \in \Pi$.

\item\label{item:desc-Pi=Pir}
If $\desc{n} \in \symm{\Pi}$, then $\Pi = \desc{n}(\Pi)$.

\item\label{item:Pi=Pir-Pi'=Pi'r}
If $\Pi = \desc{n}(\Pi)$, then $\Pi' = \desc{n+1}(\Pi')$.
\end{enumerate}
\end{lemma}

\begin{proof}
\ref{item:desc-in-Pi}
Assume first that $\desc{n} \in \symm{\Pi}$.
Let $B \in \Pi$ and $x \in B$.
Then $\desc{n}(x) \in B$, so $x \in \desc{n}^{-1}(B) = \desc{n}(B)$; hence $B \subseteq \desc{n}(B)$.
Let then $x \in \desc{n}(B)$.
Then $\desc{n}(x) = \desc{n}^{-1}(x) \in B$.
Since $\desc{n} \in \symm{\Pi}$, we have $x = \desc{n}(\desc{n}(x)) \in B$; hence $\desc{n}(B) \subseteq B$.

Assume then that $B = \desc{n}(B)$ for all $B \in \Pi$.
Then for all $x \in \nset{n}$, we have $\desc{n}(x) \in \eqclass{\desc{n}(x)}{\Pi} = \eqclass{x}{\Pi}$, so $\desc{n} \in \symm{\Pi}$.

\ref{item:desc-Pi=Pir}
If $\desc{n} \in \symm{\Pi}$, then for every $B \in \Pi$, we have $\desc{n}(B) \in \Pi$
by part \ref{item:desc-in-Pi}.
This implies $\Pi = \desc{n}(\Pi)$.

\ref{item:Pi=Pir-Pi'=Pi'r}
Assume that $\Pi = \desc{n}(\Pi)$. Then we have $\desc{n}(\interval{a}{b}) \in \maxintervals{\Pi}$ whenever $\interval{a}{b} \in \maxintervals{\Pi}$.
Let $B \in \Pi'$. Then $B = \interval{a}{b}$ for some $a \leq b$. Let us consider the different possibilities.

If $a = 1$ and $b < n+1$, then $\interval{1}{b} \in \maxintervals{\Pi}$.
We also have $\desc{n}(\interval{1}{b}) = \interval{\desc{n}(b)}{n} \in \maxintervals{\Pi}$.
Consequently, $\interval{\desc{n}(b)+1}{n+1} = \desc{n+1}(\interval{1}{b}) \in \Pi'$.

If $1 < a = b < n+1$, then $B = \{a\}$ and there exists $c \in \nset{n}$ with $a \leq c$ such that $\interval{a}{c} \in \maxintervals{\Pi}$.
We also have $\desc{n}(\interval{a}{c}) = \interval{\desc{n}(c)}{\desc{n}(a)} \in \maxintervals{\Pi}$.
Then there exists $d \in \nset{n}$ with $\desc{n}(a) < d$ such that $\interval{\desc{n}(a)+1}{d} \in \maxintervals{\Pi}$.
Consequently, $\{\desc{n}(a)+1\} = \desc{n+1}(\{a\}) \in \Pi'$.

If $1 < a < b < n+1$, then $a \geq 2$ and $\interval{a-1}{b} \in \maxintervals{\Pi}$.
We also have $\desc{n}(\interval{a-1}{b}) = \interval{\desc{n}(b)}{\desc{n}(a-1)} \in \maxintervals{\Pi}$.
Consequently, $\interval{\desc{n}(b)+1}{\desc{n}(a-1)} = \desc{n+1}(\interval{a}{b}) \in \Pi'$.

If $a > 1$ and $b = n+1$, then $\interval{a-1}{n} \in \maxintervals{\Pi}$.
We also have $\desc{n}(\interval{a-1}{n}) = \interval{1}{\desc{n}(a-1)} \in \maxintervals{\Pi}$.
Consequently, $\interval{1}{\desc{n}(a-1)} = \desc{n+1}(\interval{a}{n+1}) \in \Pi'$.

Therefore, $\desc{n+1}(B) \in \Pi'$ whenever $B \in \Pi'$, that is, $\Pi' = \desc{n+1}(\Pi')$.
\end{proof}

\begin{lemma}
\label{lem:conjugate-SPi}
For any partition $\Pi$ of $\nset{n}$, we have $\desc{n} \symm{\Pi} \desc{n} = \symm{\desc{n}(\Pi)}$.
\end{lemma}

\begin{proof}
By definition, we have $\pi \in \symm{\desc{n}(\Pi)}$
if and only if $\pi(B) = B$ for all $B \in \desc{n}(\Pi)$,
or, equivalently, $\pi(\desc{n}(C)) = \desc{n}(C)$ for all $C \in \Pi$.
Composing both sides from the left with $\desc{n} = \desc{n}^{-1}$ yields the equivalent condition $\desc{n} \circ \pi \circ \desc{n}(C) = C$ for all $C \in \Pi$.
By definition, this holds if and only if $\desc{n} \circ \pi \circ \desc{n} \in \symm{\Pi}$,
that is, $\pi \in \desc{n} \symm{\Pi} \desc{n}$.
\end{proof}

\begin{lemma}
\label{lem:Pi-deltaPi}
Let $\Pi$ be a partition of $\nset{n}$.
If $\Pi = \desc{n}(\Pi)$, then $\gensg{\symm{\Pi}, \desc{n}} = \symm{\Pi} \cup \desc{n} \symm{\Pi} = \symm{\Pi} \cup \symm{\Pi} \desc{n}$.
\end{lemma}

\begin{proof}
The inclusions $\symm{\Pi} \cup \desc{n} \symm{\Pi} \subseteq \gensg{\symm{\Pi}, \desc{n}}$ and $\symm{\Pi} \cup \symm{\Pi} \desc{n} \subseteq \gensg{\symm{\Pi}, \desc{n}}$ are obvious.
In order to prove the converse inclusions, note first that
any permutation $\pi \in \gensg{\symm{\Pi}, \desc{n}}$ can be written as
\[
\pi = \tau_1 \circ \desc{n} \circ \tau_2 \circ \desc{n} \circ \dots \circ \desc{n} \circ \tau_{q-1} \circ \desc{n} \circ \tau_q
\]
for some $\tau_i \in \symm{\Pi}$ ($1 \leq i \leq q$).
According to Lemma~\ref{lem:conjugate-SPi} and our hypothesis, we have $\desc{n} \symm{\Pi} \desc{n} = \symm{\desc{n}(\Pi)} = \symm{\Pi}$.
Setting $\tau'_i := \desc{n} \circ \tau_i \circ \desc{n}$, we have $\tau'_i \in \symm{\Pi}$ for any $\tau_i \in \symm{\Pi}$, and we can rewrite the above expression for $\pi$ as
\begin{align*}
\pi &= \tau_1 \circ \tau'_2 \circ \tau_3 \circ \dots \circ \tau'_{q-1} \circ \tau_q, & & \text{for odd $q$,} \\
\pi &= \desc{n} \circ \tau'_1 \circ \tau_2 \circ \tau'_3 \circ \dots \circ \tau'_{q-1} \circ \tau_q, & & \text{for even $q$.}
\end{align*}
This implies that $\pi \in \symm{\Pi} \cup \desc{n} \symm{\Pi}$.
Therefore $\gensg{\symm{\Pi}, \desc{n}} \subseteq \symm{\Pi} \cup \desc{n} \symm{\Pi}$.
A similar argument shows that $\gensg{\symm{\Pi}, \desc{n}} \subseteq \symm{\Pi} \cup \symm{\Pi} \desc{n}$
\end{proof}

\begin{proposition}
\label{prop:CompSPi}
Let $\Pi$ be a partition of $\nset{n}$.
Then the following statements hold.
\begin{enumerate}[label=\upshape{(\roman*)}]
\item\label{item:CompSPi:without-desc} If $\desc{n} \notin \symm{\Pi}$, then $\Compn[n+1]{\symm{\Pi}} = \symm{\Pi'}$.
\item\label{item:CompSPi:with-desc} If $\desc{n} \in \symm{\Pi}$, then $\Compn[n+1]{\symm{\Pi}} = \symm{\Pi'} \cup \desc{n+1} \symm{\Pi'} = \gensg{\symm{\Pi'}, \desc{n+1}}$; moreover, $\Pi' = \desc{n+1}(\Pi')$.
\end{enumerate}
\end{proposition}

\begin{proof}
The following inclusions hold by Lemmas~\ref{lem:CompSPi} and \ref{lem:Pi-prime-in-Comp-Pi}:
\[
\symm{\Pi'} \subseteq \Compn[n+1]{\symm{\Pi}} \subseteq \symm{\Pi'} \cup \desc{n+1} \symm{\Pi'}.
\]
Together with Lemma~\ref{lem:desc-CompSPi}, this implies that $\Compn[n+1]{\symm{\Pi}} = \symm{\Pi'}$ if $\desc{n} \notin \symm{\Pi}$, and $\Compn[n+1]{\symm{\Pi}} = \symm{\Pi'} \cup \desc{n+1} \symm{\Pi'}$ if $\desc{n} \in \symm{\Pi}$.
Moreover, if $\desc{n} \in \symm{\Pi}$, then $\Pi' = \desc{n+1}(\Pi')$ by Lemma~\ref{lem:Pi-PiDeltan}.
Consequently, $\symm{\Pi'} \cup \desc{n+1} \symm{\Pi'} = \gensg{\symm{\Pi'}, \desc{n+1}}$ by Lemma~\ref{lem:Pi-deltaPi}.
\end{proof}

Proposition~\ref{prop:CompSPi} describes $\Compn[n+1]{\symm{\Pi}}$ for an arbitrary partition $\Pi$ of $\nset{n}$.
In the case where $\desc{n} \notin \symm{\Pi}$, repeated application of Proposition~\ref{prop:CompSPi}\ref{item:CompSPi:without-desc}, together with Lemma~\ref{lem:descending}, shows that $\Compn[n+i]{\symm{\Pi}} = \symm{\Pi^{(i)}}$ for all $i \geq 1$.
We are now going to find out how the level sequence continues when $\desc{n} \in \symm{\Pi}$.

\begin{lemma}
\label{lem:PiGamma}
Let $\Pi$ and $\Gamma$ be partitions of $\nset{n}$.
Then $\symm{\Pi \vee \Gamma} = \gensg{\symm{\Pi}, \symm{\Gamma}}$.
\end{lemma}

\begin{proof}
Observe first that $\symm{\Pi}$ and $\symm{\Gamma}$ are subgroups of $\symm{\Pi \vee \Gamma}$, so $\gensg{\symm{\Pi}, \symm{\Gamma}} \subseteq \symm{\Pi \vee \Gamma}$.
In order to prove the converse inclusion, let $\pi \in \symm{\Pi \vee \Gamma}$, and write $\pi$ as a product of disjoint cycles:
$\pi = (c^1_1 \; c^1_2 \; \dots \; c^1_{p_1}) \circ \dots \circ (c^q_1 \; c^q_2 \; \dots \; c^q_{p_q})$.
Every orbit $\{c^i_1, c^i_2, \dots, c^i_{p_i}\}$ of $\pi$ is contained in a block of $\Pi \vee \Gamma$.
Thus, for each pair $(c^i_j, c^i_{j+1})$, there exists a sequence $c^i_j = d_0, d_1, \dots, d_r = c^i_{j+1}$ such that for all $j \in \{0, \dots, r-1\}$, it holds that either $d_j \sim_\Pi d_{j+1}$ or $d_j \sim_\Gamma d_{j+1}$. Therefore, the transposition $(d_j \; d_{j+1})$ is in $\symm{\Pi}$ or in $\symm{\Gamma}$.
Since $(x \; y)(y \; z)(x \; y) = (x \; z)$, we can show by an easy inductive argument that the transposition $(c^i_j \; c^i_{j+1})$ is in $\gensg{\symm{\Pi}, \symm{\Gamma}}$.
From the fact that $(c^i_1 \; c^i_2 \; \dots \; c^i_{p_i}) = (c^i_1 \; c^i_2)(c^i_2 \; c^i_3) \dots (c^i_{p_i-1} \; c^i_{p_i})$, it follows that $(c^i_1 \; c^i_2 \; \dots \; c^i_{p_i}) \in \gensg{\symm{\Pi}, \symm{\Gamma}}$.
Consequently, $\pi \in \gensg{\symm{\Pi}, \symm{\Gamma}}$.
\end{proof}

For $n \in \IN_+$, let $\Lambda_n := \{\{i, \desc{n}(i)\} \mid 1 \leq i \leq n\}$. The set $\Lambda_n$ is a partition of $\nset{n}$.

\begin{lemma}
\label{lem:Pi-Pi-prime}
Let $\Pi$ be a partition of $\nset{n}$.
Then the following statements hold.
\begin{enumerate}[label=\upshape{(\roman*)}]
\item\label{Pi-Pi-prime-B} $\desc{n} \in \symm{\Pi}$ if and only if $\Pi = \Pi \vee \Lambda_n$.
\item\label{Pi-delta} $\gensg{\symm{\Pi}, \desc{n}} \subseteq \symm{\Pi \vee \Lambda_n}$.
\end{enumerate}
\end{lemma}

\begin{proof}
\ref{Pi-Pi-prime-B}
By definition, the orbits of $\desc{n}$ are precisely the blocks of $\Lambda_n$. Therefore, it is clear that $\desc{n} \in \symm{\Pi}$ if and only if $\Pi$ is a coarsening of $\Lambda_n$, i.e., $\Pi = \Pi \vee \Lambda_n$.

\ref{Pi-delta}
Since $\desc{n} \in \symm{\Lambda_n}$, we have $\gensg{\symm{\Pi}, \desc{n}} \subseteq \gensg{\symm{\Pi}, \symm{\Lambda_n}} = \symm{\Pi \vee \Lambda_n}$ by Lemma~\ref{lem:PiGamma}.
\end{proof}

\begin{lemma}
\label{lem:middle-block}
Assume that $\Pi$ is an interval partition of $\nset{n}$ with no consecutive
nontrivial
blocks.
If $\Pi = \desc{n}(\Pi)$, then $\Pi' = (\Pi \vee \Lambda_n)'$.
\end{lemma}

\begin{proof}
Since $\Pi$ is an interval partition,
we have $\maxintervals{\Pi} = \Pi$.
The assumption $\Pi = \desc{n}(\Pi)$ implies that $\Pi \vee \Lambda_n = \{B \cup \desc{n}(B) \mid B \in \Pi\}$.
The blocks of $\Pi \vee \Lambda_n$ are
of the form $\interval{a}{b} \cup \desc{n}(\interval{a}{b})$ for some $1 \leq a \leq b \leq n$, that is, they are
unions of two
nonconsecutive
intervals, with the exception of the ``middle block'', i.e., the block containing the element $\lceil n/2 \rceil$.

If $n = 2k+1$, then the middle block coincides with a single $\Pi$\hyp{}block:
\[
\eqclass{(k+1)}{(\Pi \vee \Lambda_n)} = \eqclass{(k+1)}{\Pi} \cup \desc{n}(\eqclass{(k+1)}{\Pi}) = \eqclass{(k+1)}{\Pi}.
\]
In this case, $\maxintervals{\Pi \vee \Lambda_n} = \maxintervals{\Pi}$,
so it clearly holds that $\Pi' = (\Pi \vee \Lambda_n)'$.

If $n = 2k$, then the middle block is
\[
\eqclass{k}{(\Pi \vee \Lambda_n)}
= \eqclass{k}{\Pi} \cup \desc{n}(\eqclass{k}{\Pi})
= \eqclass{k}{\Pi} \cup \eqclass{(k+1)}{\Pi}.
\]
If $\eqclass{k}{\Pi} = \eqclass{(k+1)}{\Pi}$, then the middle block is the $\Pi$\hyp{}block $\eqclass{k}{\Pi}$.
In this case,
we have again that $\maxintervals{\Pi \vee \Lambda_n} = \maxintervals{\Pi}$,
so it holds that $\Pi' = (\Pi \vee \Lambda_n)'$.
If the blocks $\eqclass{k}{\Pi}$ and $\eqclass{(k+1)}{\Pi}$ are distinct, then these blocks must be singletons, namely $\{k\}$ and $\{k+1\}$, because we are assuming that $\Pi = \desc{n}(\Pi)$ and $\Pi$ has no consecutive
nontrivial
blocks.
Consequently, the middle block of $\Pi \vee \Lambda_n$ is $\eqclass{k}{(\Pi \vee \Lambda_n)} = \{k, k+1\}$.
Therefore,
$\maxintervals{\Pi \vee \Lambda_n}$ coincides with $\maxintervals{\Pi}$ with the exception that we have $\{k, k+1\} \in \maxintervals{\Pi \vee \Lambda_n}$ while $\{k\}, \{k+1\} \in \maxintervals{\Pi}$.
Notwithstanding this little difference between $\maxintervals{\Pi \vee \Lambda_n}$ and $\maxintervals{\Pi}$, the partitions $\Pi'$ and $(\Pi \vee \Lambda_n)'$ are easily seen to be equal, because
the singletons $\{k\}$ and $\{k+1\}$ in $\maxintervals{\Pi}$ remain as separate blocks when we construct $\Pi'$,
and
the interval $\{k, k+1\}$ in $\maxintervals{\Pi \vee \Lambda_n}$ is split into new blocks $\{k\}$ and $\{k+1\}$ when we construct $(\Pi \vee \Lambda_n)'$.
\end{proof}

\begin{proposition}
\label{prop:CompSPi-more}
Let $\Pi$ be a partition of $\nset{n}$.
Then the following statements hold.
\begin{enumerate}[label=\upshape{(\roman*)}]
\item\label{item:CompSPi:intpart} $\Pi'$ is an interval partition with no consecutive nontrivial blocks.
\item\label{item:CompSPi:SPi-desc} If $\Pi$ is an interval partition with no consecutive nontrivial blocks and $\Pi = \desc{n}(\Pi)$, then $\Compn[n+1]{\gensg{\symm{\Pi}, \desc{n}}} = \gensg{\symm{\Pi'}, \desc{n+1}}$; moreover, $\Pi' = \desc{n+1}(\Pi')$.
\end{enumerate}
\end{proposition}

\begin{proof}
\ref{item:CompSPi:intpart}
Clear from the definition of $\Pi'$.

\ref{item:CompSPi:SPi-desc}
We have $\Pi' = (\Pi \vee \Lambda_n)'$ by Lemma~\ref{lem:middle-block}.
Since $\desc{n} \in \symm{\Pi \vee \Lambda_n}$ by Lemma~\ref{lem:Pi-Pi-prime}\ref{Pi-Pi-prime-B}, we get from Proposition~\ref{prop:CompSPi} that
\[
\Compn[n+1]{\symm{\Pi \vee \Lambda_n}}
= \symm{(\Pi \vee \Lambda_n)'} \cup \desc{n+1} \symm{(\Pi \vee \Lambda_n)'}
= \symm{\Pi'} \cup \desc{n+1} \symm{\Pi'}.
\]
By Lemma~\ref{lem:Pi-PiDeltan}, we have $\Pi' = \desc{n+1}(\Pi')$, so Lemma~\ref{lem:Pi-deltaPi} implies $\symm{\Pi'} \cup \desc{n+1} \symm{\Pi'} = \gensg{\symm{\Pi'}, \desc{n+1}}$.
We also have $\symm{\Pi} \subseteq \gensg{\symm{\Pi}, \desc{n}} \subseteq \symm{\Pi \vee \Lambda_n}$ by Lemma~\ref{lem:Pi-Pi-prime}\ref{Pi-delta} and $\symm{\Pi'} \subseteq \Compn[n+1]{\symm{\Pi}}$ by Lemma~\ref{lem:Pi-prime-in-Comp-Pi}.
Putting these facts together and using the monotonicity of $\Compn[n+1]{}$, we get
\[
\symm{\Pi'} \subseteq \Compn[n+1]{\symm{\Pi}} \subseteq \Compn[n+1]{\gensg{\symm{\Pi}, \desc{n}}} \subseteq \Compn[n+1]{\symm{\Pi \vee \Lambda_n}} = \gensg{\symm{\Pi'}, \desc{n+1}}.
\]
Since $\desc{n+1} \in \Compn[n+1]{\gensg{\symm{\Pi}, \desc{n}}}$ by Lemma~\ref{lem:descending}
and $\symm{\Pi'} \subseteq \Compn[n+1]{\gensg{\symm{\Pi}, \desc{n}}}$ by the above inclusions, we have
$\gensg{\symm{\Pi'}, \desc{n+1}} \subseteq \Compn[n+1]{\gensg{\symm{\Pi}, \desc{n}}}$,
and we conclude that $\Compn[n+1]{\gensg{\symm{\Pi}, \desc{n}}} = \gensg{\symm{\Pi'}, \desc{n+1}}$.
\end{proof}

\begin{theorem}
\label{thm:CompSPi-general}
Let $\Pi$ be a partition of $\nset{n}$. Then for all $i \geq 1$, it holds
\[
\Compn[n+i]{\symm{\Pi}} =
\begin{cases}
\symm{\Pi^{(i)}}, & \text{if $\desc{n} \notin \symm{\Pi}$,} \\
\gensg{\symm{\Pi^{(i)}}, \desc{n+i}}, & \text{if $\desc{n} \in \symm{\Pi}$.}
\end{cases}
\]
\end{theorem}

\begin{proof}
If $\desc{n} \notin \symm{\Pi}$, then an easy inductive argument that uses Lemma~\ref{lem:Comp-Pat-transitive}, Proposition~\ref{prop:CompSPi}\ref{item:CompSPi:without-desc}, and Lemma~\ref{lem:descending} shows that $\Compn[n+i]{\symm{\Pi}} = \symm{\Pi^{(i)}}$ for all $i \geq 1$.

For the case where $\desc{n} \in \symm{\Pi}$, we will prove the following slightly stronger statement: for all $i \geq 1$, $\Compn[n+i]{\symm{\Pi}} = \gensg{\symm{\Pi^{(i)}}, \desc{n+i}}$, $\Pi^{(i)}$ is an interval partition with no consecutive nontrivial blocks, and $\Pi^{(i)} = \desc{n+i}(\Pi^{(i)})$.
This follows by an easy inductive argument that uses Lemma~\ref{lem:Comp-Pat-transitive}, Proposition~\ref{prop:CompSPi}\ref{item:CompSPi:with-desc}, Proposition~\ref{prop:CompSPi-more} and Lemma~\ref{lem:descending}.
\end{proof}

\subsection{Speed of reaching a steady sequence}

We can also determine the exact point where the sequence $(\Compn[n+i]{\symm{\Pi}})_{i \in \IN_+}$ reaches one of the steady sequences predicted by Theorem~\ref{thm:AtkBea-asymptotic}.
For a partition $\Pi$ of $\nset{n}$ and $a, b \in \IN_+$, let
\begin{gather*}
\maxblocksize(\Pi) := \max(\{\card{B} : B \in \maxintervalsmiddle{\Pi}\} \cup \{1\}), \\
\maxblocksize_{a,b}(\Pi) := \max(\maxblocksize(\Pi), \card{\eqclass{1}{\maxintervals{\Pi}}} - a + 1, \card{\eqclass{n}{\maxintervals{\Pi}}} - b + 1).
\end{gather*}

\begin{lemma}
\label{lem:blocks-Pi-i}
Let $\Pi$ be a partition of $\nset{n}$. Then the following statements hold.
\begin{enumerate}[label=\upshape{(\roman*)}]
\item\label{item:M-Pi'}
$\maxblocksize(\Pi') = \maxblocksize(\Pi) = 1$ or $\maxblocksize(\Pi') = \maxblocksize(\Pi) - 1$.
\item\label{item:nontrivial-blocks}
Assume that $\Pi \neq \{\nset{n}\}$.
Then for all $i \geq 1$, it holds that $\eqclass{1}{\Pi^{(i)}} = \eqclass{1}{\maxintervals{\Pi}}$ and $\eqclass{(n+i)}{\Pi^{(i)}} = (\eqclass{n}{\maxintervals{\Pi}}) + i$.
\item\label{item:B-nontrivial}
Assume that $\Pi$ is nontrivial. Then for all $m \geq \max(\maxblocksize(\Pi)-1, 1)$, it holds that every nontrivial $\Pi^{(m)}$\hyp{}block is either $\eqclass{1}{\Pi^{(m)}}$ or $\eqclass{(n+m)}{\Pi^{(m)}}$.
\end{enumerate}
\end{lemma}

\begin{proof}
\ref{item:M-Pi'}
Consider first the case where $\maxblocksize(\Pi) = 1$ and $\maxintervalsmiddle{\Pi} = \emptyset$.
Then either $\Pi = \{\nset{n}\}$ or $\Pi = \{\interval{1}{a}, \interval{a+1}{n}\}$ for some $a \in \nset{n-1}$.
In the former case, $\Pi' = \{\nset{n+1}\}$, and it holds that $\maxintervalsmiddle{\Pi'} = \emptyset$; hence $\maxblocksize(\Pi') = 1$.
In the latter case, $\Pi' = \{\interval{1}{a}, \{a+1\}, \interval{a+2}{n+1}\}$, and it holds that $\maxintervalsmiddle{\Pi'} = \{\{a+1\}\}$; hence $\maxblocksize(\Pi') = 1$.

Consider then the case where $\maxblocksize(\Pi) = 1$ and $\maxintervalsmiddle{\Pi} \neq \emptyset$.
Then every $\Pi$\hyp{}block distinct from $\eqclass{1}{\Pi}$ and $\eqclass{n}{\Pi}$ is trivial,
and it is clear from the definition of $\Pi'$ that every $\Pi'$\hyp{}block distinct from $\eqclass{1}{\Pi'}$ and $\eqclass{(n+1)}{\Pi'}$ is trivial.
Consequently, $\maxblocksize(\Pi') = 1$.

Finally, consider the case where $\maxblocksize(\Pi) > 1$.
Let $B' \in \maxintervalsmiddle{\Pi'}$.
Then either $B' = \{\min B\}$ or $B' = B \setminus \{\min B\}$ for some $B \in \maxintervalsmiddle{\Pi}$, or $B' = \{\min \eqclass{n}{\Pi}\}$,
so $\card{B'} = 1$ or $\card{B'} = \card{B} - 1$.
There exists a block $B \in \maxintervalsmiddle{\Pi}$ with $\card{B} = \maxblocksize(\Pi) > 1$.
It follows that
\[
\maxblocksize(\Pi')
= \max ( \{ \card{B'} \mid B' \in \maxintervalsmiddle{\Pi'} \} \cup \{ 1 \} )
= \max ( \{ \card{B} - 1 \mid \maxintervalsmiddle{\Pi} \} \cup \{ 1 \} )
= \maxblocksize(\Pi) - 1.
\]

\ref{item:nontrivial-blocks}
Since $\Pi \neq \{\nset{n}\}$,
the elements $1$ and $n$ are not in the same $\maxintervals{\Pi}$\hyp{}block.
By the definition of $\Pi'$, the blocks $\eqclass{1}{\maxintervals{\Pi}}$ and $\eqclass{n}{\maxintervals{\Pi}}$ of $\maxintervals{\Pi}$ give rise to the blocks $\eqclass{1}{\Pi'} = \eqclass{1}{\maxintervals{\Pi}}$ and $\eqclass{(n+1)}{\Pi'} = (\eqclass{n}{\maxintervals{\Pi}}) + 1$.
An obvious inductive argument then shows that for all $i \geq 1$, $\eqclass{1}{\Pi^{(i)}} = \eqclass{1}{\maxintervals{\Pi}}$ and $\eqclass{(n+i)}{\Pi^{(i)}} = (\eqclass{n}{\maxintervals{\Pi}}) + i$.

\ref{item:B-nontrivial}
Let $M := \maxblocksize(\Pi)$.
Then, by part~\ref{item:M-Pi'}, the sequence
\[
\maxblocksize(\Pi), \maxblocksize(\Pi^{(1)}), \maxblocksize(\Pi^{(2)}), \dots
\]
counts down from $M$ to $1$ and then it stays at $1$.
In other words, $\maxblocksize(\Pi^{(m)}) = 1$ whenever $m \geq \max(M-1, 1)$.
The claim then follows from part~\ref{item:nontrivial-blocks}.
\end{proof}

\begin{theorem}
\label{thm:SPi-asymptotic}
Let $\Pi$ be a partition of $\nset{n}$ with at least two blocks.
Let $c := \card{\eqclass{1}{\maxintervals{\Pi}}}$, $d := \card{\eqclass{n}{\maxintervals{\Pi}}}$.
Assume that $\symm{\Pi} \neq \symm{n}^{c,d}$.
Then $\Compn[n + i]{\symm{\Pi}} = \symm{n + i}^{c,d}$ or $\Compn[n + i]{\symm{\Pi}} = \gensg{\symm{n + i}^{c,d}, \desc{n + i}}$ if and only if $i \geq \max(\maxblocksize(\Pi)-1, 1)$.
\end{theorem}

\begin{proof}
Lemma~\ref{lem:blocks-Pi-i} asserts that the only nontrivial $\Pi^{(i)}$\hyp{}blocks are $\eqclass{1}{\Pi^{(i)}} = \eqclass{1}{\maxintervals{\Pi}}$ and $\eqclass{n + \ell}{\Pi^{(i)}} = (\eqclass{n}{\maxintervals{\Pi}}) + i$ whenever $i \geq \max(\maxblocksize(\Pi)-1, 1)$.
The claim then follows from Propositions~\ref{prop:CompSPi} and \ref{prop:CompSPi-more}.
\end{proof}

\subsection{Arbitrary intransitive groups}

If $G$ is an arbitrary intransitive subgroup of $\symm{n}$, $\Pi := \orbits G$, and $a$ and $b$ are the largest integers $\alpha$ and $\beta$ such that $\symm{n}^{\alpha,\beta} \leq G$, then we have $\symm{n}^{a,b} \leq G \leq \symm{\Pi}$, and the monotonicity of $\Comp$ implies that
\[
\symm{n+i}^{a,b} = \Compn[n+i]{\symm{n}^{a,b}} \leq \Compn[n+i]{G} \leq \Compn[n+i]{\symm{\Pi}}
\]
for all $i \geq 1$.
According to Theorem~\ref{thm:CompSPi-general}, $\Compn[n+i]{\symm{\Pi}}$ equals either $\symm{\Pi^{(i)}}$ or $\gensg{\symm{\Pi^{(i)}}, \desc{n+i}}$.
Furthermore, if $c := \card{\eqclass{1}{\maxintervals{\Pi}}}$, $d := \card{\eqclass{n}{\maxintervals{\Pi}}}$ and $i$ is sufficiently large, as prescribed in Theorem~\ref{thm:SPi-asymptotic}, then $\symm{n+i}^{a,b} \leq \Compn[n+i]{G} \leq \gensg{\symm{n+i}^{c,d}, \desc{n+i}}$.
Since $a$ and $b$ may be strictly less than $c$ and $d$, respectively, it is not evident from this fact alone
which steady sequence is eventually reached by the level sequence of $G$.
We are now going to clarify this point.

Let $G \leq \symm{n}$ and let $B \subseteq \nset{n}$.
We say that $G$ is \emph{transitive on $B$} if for all $x, y \in B$ there exists $\pi \in G$ such that $\pi(x) = y$.
We also write $\symm{B}$ to denote the subgroup of $\symm{n}$ comprising all permutations that map the elements of $B$ to elements of $B$ and fix every element that is not in $B$.
In other words, $\symm{B} = \symm{\Pi}$ for the partition $\Pi$ of $\nset{n}$ in which the only potentially nontrivial block is $B$.

\begin{lemma}
\label{lem:G1G2}
Let $\Pi$ be
an
interval partition of $\nset{n}$ satisfying $\Pi = \desc{n}(\Pi)$.
Let $G \leq \gensg{\symm{\Pi}, \desc{n}}$.
Let $G_1 := G \cap \symm{\Pi}$, $G_2 := G \cap \desc{n} \symm{\Pi}$.
Then the following statements hold.
\begin{enumerate}[label=\upshape{(\roman*)}]
\item\label{item:cosets}
$G_1$ is a subgroup of $G$ and either $G_2 = \emptyset$ or
$G_2 = \pi G_1$ for any $\pi \in G_2$.

\item\label{item:union-orbits}
For every block $B \in \Pi$, the set $B \cup \desc{n}(B)$ is a union of orbits of $G$.

\item\label{item:transitive}
Let $B \in \Pi$, and assume that $B \cap \desc{n}(B) = \emptyset$.
If $G_2 \neq \emptyset$, then $G$ is transitive on $B \cup \desc{n}(B)$ if and only if $G_1$ is transitive on $B$ and on $\desc{n}(B)$.
\end{enumerate}
\end{lemma}

\begin{proof}
\ref{item:cosets}
The set $G_1$ is an intersection of two subgroups of $\gensg{\symm{\Pi}, \desc{n}}$, so it is a subgroup, too. Moreover, $G_1$ is contained in $G$, so we conclude that $G_1 \leq G$.

We have $\gensg{\symm{\Pi}, \desc{n}} = \symm{\Pi} \cup \desc{n} \symm{\Pi}$ by Lemma~\ref{lem:Pi-deltaPi}.
If $G_2 \neq \emptyset$, then fix an element $\pi \in G_2$, and let $\tau \in G_2$.
Then $\pi = \desc{n} \circ \hat{\pi}$ and $\tau = \desc{n} \circ \hat{\tau}$ for some $\hat{\pi}, \hat{\tau} \in \symm{\Pi}$, and we have $\pi^{-1} \circ \tau = (\desc{n} \circ \hat{\pi})^{-1} \circ (\desc{n} \circ \hat{\tau}) = \hat{\pi}^{-1} \circ \desc{n}^{-1} \circ \desc{n} \circ \hat{\tau} = \hat{\pi}^{-1} \circ \hat{\tau} \in \symm{\Pi}$.
Clearly, $\pi^{-1} \circ \tau$ is also a member of $G$, so we have $\pi^{-1} \circ \tau \in G_1$.
Consequently, $\tau = \pi \circ \pi^{-1} \circ \tau \in \pi G_1$; hence $G_2 \subseteq \pi G_1$.
For the converse inclusion, let $\rho \in \pi G_1$.
Then $\rho = \pi \circ \hat{\rho}$ for some $\hat{\rho} \in G_1$.
Since $\pi, \hat{\rho} \in G$, we also have $\rho \in G$.
Since $\pi = \desc{n} \circ \hat{\pi}$ for some $\hat{\pi} \in \symm{\Pi}$ and $\hat{\rho} \in \symm{\Pi}$, we also have that $\rho = \pi \circ \hat{\rho} = (\desc{n} \circ \hat{\pi}) \circ \hat{\rho} = \desc{n} \circ (\hat{\pi} \circ \hat{\rho}) \in \desc{n} \symm{\Pi}$.
Consequently, $\rho \in G \cap \desc{n} \symm{\Pi} = G_2$, so $\pi G_1 \subseteq G_2$.

\ref{item:union-orbits}
Clear.

\ref{item:transitive}
Assume first that $G$ is transitive on $B \cup \desc{n}(B)$.
This implies in particular that for all $x, y \in B$, there exists $\pi \in G$ such that $\pi(x) = y$.
Since $\pi$ maps an element of a block $B$ of $\Pi$ to an element of $B$ and $B \cap \desc{n}(B) = \emptyset$, we must have $\pi \in G \cap \symm{\Pi} = G_1$.
We conclude that $G_1$ is transitive on $B$.
In a similar way, we can also conclude that $G_1$ is transitive on $\desc{n}(B)$.

Assume then that $G_1$ is transitive on $B$ and on $\desc{n}(B)$.
By part~\ref{item:cosets}, $G_2 = \pi G_1$ for some $\pi \in G_2$, and we have that $\pi = \desc{n} \circ \hat{\pi}$ for some $\hat{\pi} \in \symm{\Pi}$.
The transitivity of $G_1$ on $B$ implies that for all $x, y \in B$, there exists a permutation $\sigma^{x,y} \in G_1$ such that $\sigma^{x,y}(x) = y$.
Furthermore, for $x \in B$ and $y \in \desc{n}(B)$, the permutation $\pi \circ \sigma^{x,\pi^{-1}(y)} \in G$ satisfies $\pi \circ \sigma^{x,\pi^{-1}(y)}(x) = y$.
A similar argument based on the transitivity of $G_1$ on $\desc{n}(B)$ shows that for every $x \in \desc{n}(B)$, $y \in B \cup \desc{n}(B)$, there exists a permutation $\tau \in G$ such that $\tau(x) = y$.
We conclude that $G$ is transitive on $B \cup \desc{n}(B)$.
\end{proof}

\begin{remark}
The assumption $B \cap \desc{n}(B) = \emptyset$ is necessary in Lemma~\ref{lem:G1G2}\ref{item:transitive}.
To see this, consider, for example, the interval partition $\Pi = \{\{1\}, \{2,3\}, \{4\}\}$ and $G = \{\asc{4}, \desc{4}\}$.
It holds that $\Pi = \desc{4}(\Pi)$, $G \leq \gensg{\symm{\Pi},\desc{4}}$ and $G$ is transitive on the block $B = \{2,3\}$ of $\Pi$, but $G_1 = G \cap \symm{\Pi} = \{\asc{4}\}$ is not transitive on $B$.
\end{remark}

\begin{lemma}
\label{lem:transitiveblocks}
Let $\Pi$ be a nontrivial partition of $\nset{n}$.
Assume that $B'$ is a nontrivial $\Pi'$\hyp{}block such that $B' \cap \desc{n+1}(B') = \emptyset$, and let $B$ be the $\maxintervals{\Pi}$\hyp{}block given by
\[
B :=
\begin{cases}
B' \cup \{\min B' - 1\}, & \text{if $B' \neq \eqclass{1}{\Pi'}$ and $B' \neq \eqclass{(n+1)}{\Pi'}$,} \\
B', & \text{if $B' = \eqclass{1}{\Pi'}$,} \\
B' - 1, & \text{if $B' = \eqclass{(n+1)}{\Pi'}$.}
\end{cases}
\]
Then the following statements hold.

\begin{enumerate}[label=\upshape{(\roman*)}]
\item\label{item:transitiveblocks-without-desc}
Assume that $G \leq \symm{\Pi}$, and write $G' := \Compn[n+1]{G}$.
Then $G'$ is transitive on $B'$ if and only if $\symm{B} \leq G$.

\item\label{item:transitiveblocks-with-desc}
Assume that $\Pi = \desc{n}(\Pi)$ and $G \leq \gensg{\symm{\Pi}, \desc{n}}$, and write $G' := \Compn[n+1]{G}$.
Let $G_1 := G \cap \symm{\Pi}$, $G_2 := G \cap \desc{n} \symm{\Pi}$.
If $G_2 = \emptyset$, then $G'$ is transitive on $B'$ if and only if $\symm{B} \leq G$.
If $G_2 \neq \emptyset$, then $G'$ is transitive on $B' \cup \desc{n+1}(B')$ if and only if $\symm{B} \leq G$ and $\symm{\desc{n}(B)} \leq G$.
\end{enumerate}
\end{lemma}

\begin{proof}
\ref{item:transitiveblocks-without-desc}
The sufficiency of the condition is clear:
it is easy to see that if $\symm{B} \leq G$, then $\symm{B'} = \Compn[n+1]{\symm{B}} \leq \Compn[n+1]{G} = G'$; hence $G'$ is transitive on $B'$.

Assume then that $G'$ is transitive on $B'$.
Write $B = \interval{a}{b}$ with $1 \leq a < b \leq n$.
Consider first the case where $B' \neq \eqclass{1}{\Pi'}$ and $B' \neq \eqclass{(n+1)}{\Pi'}$; then $B' = \interval{a+1}{b}$.
By the transitivity on $B'$, there exist permutations $\pi, \tau \in G'$ such that $\pi(a+1) = b$, $\tau(b) = b-1$.
We have $G' \leq \Compn[n+1]{\symm{\Pi}}$.
By Proposition~\ref{prop:CompSPi},
$\Compn[n+1]{\symm{\Pi}} = \symm{\Pi'}$ if $\desc{n} \notin \symm{\Pi}$, and
$\Compn[n+1]{\symm{\Pi}} = \gensg{\symm{\Pi'},\desc{n+1}} = \symm{\Pi'} \cup \desc{n+1} \symm{\Pi'}$ if $\desc{n} \in \symm{\Pi}$.
In any case $G' \subseteq \symm{\Pi'} \cup \desc{n+1} \symm{\Pi'}$.
Since both $\pi$ and $\tau$ map an element of $B'$ to an element of $B'$ and $B' \cap \desc{n}(B') = \emptyset$, we must have $\pi, \tau \in \symm{\Pi'}$.
It follows that $\pi(a) = a$ and $\tau(b+1) = b+1$, because $\{a\}$ and $\{b+1\}$ are $\Pi'$\hyp{}blocks.
By Lemma~\ref{lem:AtkBea-L11}, $G$ contains the cycles $(a \; a+1 \; \dots \; b-1)$ and $(b-1 \; b)$, which generate $\symm{B}$; thus $\symm{B} \leq G$.

Consider then the case where $B' = \eqclass{1}{\Pi'}$; then $B' = B = \interval{a}{b}$.
Since $G'$ is transitive on $B'$, there exist permutations $\pi, \tau \in G'$ such that $\pi(b) = a = 1$ and $\tau(b) = b-1$.
As above, we must have $\pi, \tau \in \symm{\Pi'}$; hence $\pi(b+1) = \tau(b+1) = b+1$, because $\{b+1\}$ is a $\Pi'$\hyp{}block.
By Lemma~\ref{lem:AtkBea-L11}, $G$ contains the cycles $(a \; a+1 \; \dots \; b)$ and $(b-1 \; b)$, which generate $\symm{B}$; thus $\symm{B} \leq G$.

Consider finally the case where $B' = \eqclass{(n+1)}{\Pi'}$; then $B' = B+1 = \interval{a+1}{b+1}$.
Since $G'$ is transitive on $B'$, there exist permutations $\pi, \tau \in G'$ such that $\pi(a+1) = b+1$ and $\tau(a+1) = a+2$.
As above, we must have $\pi, \tau \in \symm{\Pi'}$; hence $\pi(a) = \tau(a) = a$, because $\{a\}$ is a $\Pi'$\hyp{}block.
By Lemma~\ref{lem:AtkBea-L11}, $G$ contains the cycles $(a \; a+1 \; \dots \; b)$ and $(a \; a+1)$, which generate $\symm{B}$; thus $\symm{B} \leq G$.

\ref{item:transitiveblocks-with-desc}
If $G_2 = \emptyset$, then $G \leq \symm{\Pi}$, so the claim follows immediately from part~\ref{item:transitiveblocks-without-desc}.
Assume then that $G_2 \neq \emptyset$.
By Lemma~\ref{lem:Pi-deltaPi}, Proposition \ref{prop:CompSPi-more}, and the monotonicity of $\Compn[n+1]{}$,
\[
G'
= \Compn[n+1]{G} \leq \Compn[n+1]{\gensg{\symm{\Pi},\desc{n}}}
= \gensg{\symm{\Pi'},\desc{n+1}}
= \symm{\Pi'} \cup \desc{n+1} \symm{\Pi'}.
\]
Write $G'_1 := G' \cap \symm{\Pi'}$, $G'_2 := G' \cap \desc{n+1} \symm{\Pi'}$.
Since $\Pi'$ is an interval partition and $\Pi' = \desc{n+1}(\Pi')$ holds by Lemma~\ref{lem:Pi-PiDeltan}\ref{item:Pi=Pir-Pi'=Pi'r}, we obtain from Lemma~\ref{lem:G1G2}\ref{item:transitive} that $G'$ is transitive on $B' \cup \desc{n+1}(B')$ if and only if $G'_1$ is transitive on $B'$ and on $\desc{n}(B')$.

It remains to show that $G'_1$ is transitive on $B'$ and on $\desc{n}(B')$ if and only if $\symm{B} \leq G$ and $\symm{\desc{n}(B)} \leq G$.
If $G'_1$ is transitive on $B'$ and on $\desc{n}(B')$, then $G'$ is transitive on $B'$ and on $\desc{n}(B')$, and part~\ref{item:transitiveblocks-without-desc} yields $\symm{B} \leq G$ and $\symm{\desc{n}(B)} \leq G$.
Assume then that $\symm{B} \leq G$ and $\symm{\desc{n}(B)} \leq G$.
Then
\begin{gather*}
\symm{B'} = \Compn[n+1]{\symm{B}} \leq \Compn[n+1]{G} = G', \\
\symm{\desc{n+1}(B')} = \Compn[n+1]{\symm{\desc{n}(B)}} \leq \Compn[n+1]{G} = G'.
\end{gather*}
A similar argument as above, using the fact that $B' \cap \desc{n+1}(B') = \emptyset$, shows that $\symm{B'} \leq \symm{\Pi'}$ and $\symm{\desc{n+1}(B')} \leq \symm{\Pi'}$.
Consequently, $\symm{B'}$ and $\symm{\desc{n+1}(B')}$ are subgroups of $G' \cap \symm{Pi'} = G'_1$, which obviously implies that $G'_1$ is transitive on $B'$ and on $\desc{n+1}(B')$.
\end{proof}

\begin{lemma}
\label{lem:Mab}
Let $G \leq \symm{n}$ be an intransitive group, and let $G' := \Compn[n+1]{G}$.
Let $\Pi_0 := \orbits G$, $\Pi_1 := \orbits G'$, and let $a$ and $b$ be the largest numbers $\alpha$ and $\beta$, respectively, such that $\symm{n}^{\alpha,\beta} \leq G$.
Then either $\maxblocksize_{a,b}(\Pi_1) = \maxblocksize_{a,b}(\Pi_0) = 1$ or $1 \leq \maxblocksize_{a,b}(\Pi_1) < \maxblocksize_{a,b}(\Pi_0)$.
\end{lemma}

\begin{proof}
We consider the case where $\desc{n} \notin G$. The case where $\desc{n} \in G$ can be proved with minor modifications.
We have $\symm{n}^{a,b} \leq G \leq \symm{\Pi_0}$, so
\[
\interval{1}{a} \subseteq \eqclass{1}{\maxintervals{\Pi_0}} \subseteq \eqclass{1}{\Pi_0}, \qquad
\desc{n}(\interval{1}{b}) = \interval{n-b+1}{n} \subseteq \eqclass{n}{\maxintervals{\Pi_0}} \subseteq \eqclass{n}{\Pi_0}.
\]
The monotonicity of $\Compn[n+1]{}$ and the fact that $\Pi_1$ is the finest partition $\Gamma$ such that $G' \leq \symm{\Gamma}$ imply that
\[
\Compn[n+1]{\symm{n}^{a,b}} = \symm{n+1}^{a,b} \leq G' \leq \symm{\Pi_1} \leq \symm{\Pi_0'}.
\]
Thus $\Pi_1 \sqsubseteq \Pi_0'$.
Furthermore, since $\maxintervals{\Pi_1} \sqsubseteq \Pi_1$, we have
\begin{gather*}
\interval{1}{a} \subseteq \eqclass{1}{\maxintervals{\Pi_1}} \subseteq \eqclass{1}{\Pi_1} \subseteq \eqclass{1}{\Pi_0'}, \\
\desc{n+1}(\interval{1}{b}) = \interval{n-b+2}{n+1} \subseteq \eqclass{(n+1)}{\maxintervals{\Pi_1}} \subseteq \eqclass{(n+1)}{\Pi_1} \subseteq \eqclass{(n+1)}{\Pi_0'}.
\end{gather*}

Let us compare the numbers
\begin{gather*}
\maxblocksize_{a,b}(\Pi_0) = \max(\maxblocksize(\Pi_0), \card{\eqclass{1}{\maxintervals{\Pi_0}}} - a + 1, \card{\eqclass{n}{\maxintervals{\Pi_0}}} - b + 1), \\
\maxblocksize_{a,b}(\Pi_1) = \max(\maxblocksize(\Pi_1), \card{\eqclass{1}{\maxintervals{\Pi_1}}} - a + 1, \card{\eqclass{(n+1)}{\maxintervals{\Pi_1}}} - b + 1).
\end{gather*}
We focus first on the quantity $\card{\eqclass{1}{\maxintervals{\Pi_1}}} - a + 1$.
Assume first that $\eqclass{1}{\maxintervals{\Pi_0}} = \interval{1}{a}$.
Since $\eqclass{1}{\Pi_0'} = \eqclass{1}{\maxintervals{\Pi_0}}$, we have $\eqclass{1}{\maxintervals{\Pi_1}} = \interval{1}{a}$ as well.
Thus it holds that $\card{\eqclass{1}{\maxintervals{\Pi_1}}} - a + 1 = 1 = \card{\eqclass{1}{\maxintervals{\Pi_0}}} - a + 1$.
Assume then that $\interval{1}{a} \subsetneq \eqclass{1}{\maxintervals{\Pi_0}}$.
Since $\Pi_0$ is nontrivial, it holds that $\eqclass{1}{\maxintervals{\Pi_0}} \cap \desc{n}(\eqclass{1}{\maxintervals{\Pi_0}}) = \emptyset$.
By the choice of the number $a$, it follows from Lemma~\ref{lem:transitiveblocks} that $G'$ is not transitive on the set $\eqclass{1}{\maxintervals{\Pi_0}} = \eqclass{1}{\Pi_0'}$.
Therefore, $\eqclass{1}{\Pi_1}$ must be a proper subset of $\eqclass{1}{\maxintervals{\Pi_0}}$.
Since we also have $\interval{1}{a} \subseteq \eqclass{1}{\maxintervals{\Pi_1}} \subseteq \eqclass{1}{\Pi_1}$, it holds that $1 \leq \card{\eqclass{1}{\maxintervals{\Pi_1}}} - a + 1 < \card{\eqclass{1}{\maxintervals{\Pi_0}}} - a + 1$.

A similar argument shows that either
$\card{\eqclass{(n+1)}{\maxintervals{\Pi_1}}} - b + 1 = 1 = \card{\eqclass{n}{\maxintervals{\Pi_0}}} - b + 1$
or
$1 \leq \card{\eqclass{(n+1)}{\maxintervals{\Pi_1}}} - b + 1 < \card{\eqclass{n}{\maxintervals{\Pi_0}}} - b + 1$.

Finally, consider the quantity $\maxblocksize(\Pi_1) = \max(\{\card{B} : B \in \maxintervalsmiddle{\Pi_1}\} \cup \{1\})$.
If $\maxintervalsmiddle{\Pi_1} = \emptyset$, then $\maxblocksize(\Pi_1) = 1$.
Assume now that $\maxintervalsmiddle{\Pi_1} \neq \emptyset$.
Let $B \in \maxintervalsmiddle{\Pi_1}$.
Since $\Pi_1$ is a refinement of $\Pi_0'$, there exists a unique block $B'$ of $\maxintervals{\Pi_0}$ such that $B \subseteq B'$.
If $B'$ is neither $\eqclass{1}{\maxintervals{\Pi_0}}$ nor $\eqclass{n}{\maxintervals{\Pi_0}}$, i.e., $B' \in \maxintervalsmiddle{\Pi_0}$, then $B$ is contained in an $\maxintervalsmiddle{\Pi_0'}$\hyp{}block.
In this case, $\card{B} \leq \maxblocksize(\Pi_0')$ and, by Lemma~\ref{lem:blocks-Pi-i}, either $\maxblocksize(\Pi_0') = 1 = \maxblocksize(\Pi_0)$ or $1 \leq \maxblocksize(\Pi_0') < \maxblocksize(\Pi_0)$.
If $B' = \eqclass{1}{\maxintervals{\Pi_0}}$, then, since $B \neq \eqclass{1}{\maxintervals{\Pi_1}}$ and $\interval{1}{a} \subseteq \eqclass{1}{\maxintervals{\Pi_1}} \subseteq \eqclass{1}{\maxintervals{\Pi_0}}$, we have $\card{B} \leq \card{\eqclass{1}{\maxintervals{\Pi_0}}} - \card{\eqclass{1}{\maxintervals{\Pi_1}}} \leq \card{\eqclass{1}{\maxintervals{\Pi_0}}} - a < \card{\eqclass{1}{\maxintervals{\Pi_0}}} - a + 1$.
We can show in a similar way that if $B' = \eqclass{n}{\maxintervals{\Pi_0}}$, then $\card{B} < \card{\eqclass{n}{\maxintervals{\Pi_0}}} - b + 1$.
Thus, either
\[
\maxblocksize(\Pi_1) = 1 = \max(\maxblocksize(\Pi_0), \card{\eqclass{1}{\maxintervals{\Pi_0}}} - a + 1, \card{\eqclass{n}{\maxintervals{\Pi_0}}} - b + 1)
\]
or
\[
1 \leq \maxblocksize(\Pi_1) < \max(\maxblocksize(\Pi_0), \card{\eqclass{1}{\maxintervals{\Pi_0}}} - a + 1, \card{\eqclass{n}{\maxintervals{\Pi_0}}} - b + 1).
\]
From the above considerations, we conclude that either $\maxblocksize_{a,b}(\Pi_1) = 1 =  \maxblocksize_{a,b}(\Pi_0)$ or $1 \leq \maxblocksize_{a,b}(\Pi_1) < \maxblocksize_{a,b}(\Pi_0)$.
\end{proof}

\begin{theorem}
\label{thm:general-intransitive}
Let $G \leq \symm{n}$ be an intransitive group, and let $\Pi := \orbits G$.
Let $a$ and $b$ be the largest numbers $\alpha$ and $\beta$, respectively, such that $\symm{n}^{\alpha,\beta} \leq G$.
Then for all $\ell \geq \maxblocksize_{a,b}(\Pi)$, it holds that $\Compn[n + \ell]{G} = \symm{n + \ell}^{a,b}$ or $\Compn[n + \ell]{G} = \gensg{\symm{n + \ell}^{a,b}, \desc{n + \ell}}$.
\end{theorem}

\begin{proof}
We consider the case where $\desc{n} \notin G$. The case where $\desc{n} \in G$ can be proved with straightforward minor modifications.
Write $G_0 := G$, and for $i \geq 1$, let $G_i := \Compn[n+i]{G} = \Compn[n+i]{G_{i-1}}$.
For $i \geq 0$, let $\Pi_i := \orbits G_i$ and $M_i := \maxblocksize_{a,b}(\Pi_i)$.
It follows from Lemma~\ref{lem:Mab} that for all $i \geq 0$, either $M_{i+1} = M_i = 1$ or $1 \leq M_{i+1} < M_i$.
Therefore the sequence $M_0, M_1, M_2, \dots$ is strictly decreasing until it reaches $1$.
This obviously happens in at most $M_0 - 1$ steps, i.e., $M_i = 1$ whenever $i \geq M_0 - 1$.
Note that it is possible that $\maxintervals{\Pi_i}$ is a proper refinement of $\Pi_i$ even if $M_i = 1$, but in this case we necessarily have
\[
\maxintervals{\Pi_i} = \{\interval{1}{a}, \desc{n+i}(\interval{1}{b})\} \cup \{ \{j\} \mid a < j < \desc{n+i}(b) \},
\]
whence
\[
\maxintervals{\Pi_{i+1}} = \Pi_{i+1} = \{\interval{1}{a}, \desc{n+i+1}(\interval{1}{b})\} \cup \{\{j\} \mid a < j < \desc{n+i+1}(b)\}.
\]
We conclude that $G_i = \symm{n+i}^{a,b}$ whenever $i \geq M_0$.
\end{proof}


\section{Imprimitive groups}
\label{sec:imprimitive}

In this section, we investigate level sequences of imprimitive groups.
Recall that a transitive subgroup of $\symm{n}$ is \emph{imprimitive} if it preserves a nontrivial partition of $\nset{n}$.
A group $G \leq \symm{n}$ \emph{preserves} a partition $\Pi$ of $\nset{n}$ if for all $\pi \in G$ we have $\pi(B) \in \Pi$ for every $B \in \Pi$.

For a partition $\Pi$ of $\nset{n}$, denote by $\Aut \Pi$ the automorphism group of $\Pi$, i.e., the set of all $\pi \in \symm{n}$ such that $\pi(B) \in \Pi$ for all $B \in \Pi$, or, equivalently, such that $\pi(x) \sim_\Pi \pi(y)$ if and only if $x \sim_\Pi y$.
Note that $\symm{\Pi} \leq \Aut \Pi$, and this inclusion is proper if the $\Pi$\hyp{}blocks are not of pairwise distinct sizes.
The imprimitive groups are thus exactly the transitive subgroups of the groups of the form $\Aut \Pi$, where $\Pi$ is a nontrivial partition with blocks of equal size.

Let $\Pi$ be a partition of $\nset{n}$.
We say that an interval $\interval{a}{b} \subseteq \nset{n}$ is a \emph{union of interwoven $\Pi$\hyp{}blocks} (in symbols, $\interval{a}{b} \interwoven \Pi$) if
there exist integers $k > 1$ and $\ell \geq 1$ such that $b - a + 1 = k \ell$ and the sets
\begin{equation}
\{j \in \interval{a}{b} \mid j \equiv i \pmod{k} \}
=
\{a + i + m k \mid 0 \leq m < \ell\}
\qquad (0 \leq i < k),
\label{eq:interval-blocks}
\end{equation}
which constitute a partition of $\interval{a}{b}$, are $\Pi$\hyp{}blocks.
In the case where $\ell = 1$, these blocks are trivial.
Note that if $\interval{a}{b} \interwoven \Pi$, then any two consecutive elements of $\interval{a}{b}$ belong to distinct $\Pi$\hyp{}blocks.
Moreover, if $\Pi$ has no trivial blocks, then distinct intervals that are unions of interwoven $\Pi$\hyp{}blocks do not overlap (see Lemma~\ref{lem:interwoven-overlap}).

For $n \in \IN_+$ and $\ell \in \nset{n}$, define the permutations $\dja{n}{\ell}, \ajd{n}{\ell} \in \symm{n}$ as
\[
\dja{n}{\ell} := \desc{\ell} \oplus \asc{n - \ell},
\qquad
\ajd{n}{\ell} := \asc{n - \ell} \oplus \desc{\ell}.
\]

Let $\pi \in \symm{n}$ and $t \in \nset{n-1}$.
If $\{\pi(t), \pi(t+1)\} = \{a, b\}$ for $a, b \in \nset{n}$ satisfying $a \leq b-2$, then we say that $\pi$ \emph{has a jump} (\emph{between $a$ and $b$}) \emph{at $t$,} or that $t$ is a \emph{jump} of $\pi$.
For $S \subseteq \nset{n-1}$, we say that $\pi$ \emph{has a jump in $S$} if there exists $t \in S$ such that $\pi$ has a jump at $t$.

\begin{lemma}
\label{lem:one-jump}
Let $\sigma \in \symm{n}$.
The permutation $\sigma$ has no jumps if and only if $\sigma$ is either $\asc{n}$ or $\desc{n}$.
The permutation $\sigma$ has exactly one jump if and only if $\sigma$ equals one of the following:
$\dja{n}{\ell}$,
$\ajd{n}{\ell}$,
$\desc{n} \circ \dja{n}{\ell}$,
$\desc{n} \circ \ajd{n}{\ell}$
for $\ell \in \{2, \dots, n-2\}$,
$\natcycle{n}^j$,
$\desc{n} \circ \natcycle{n}^j$
for $j \in \nset{n-1}$.
\end{lemma}

\begin{proof}
Straightforward verification.
\end{proof}

\begin{lemma}
\label{lem:interwoven-overlap}
Let $\Pi$ be a partition of $\nset{n}$ with no trivial blocks.
Let $a, b, c, d \in \nset{n}$ with $a < b$ and $c < d$, and assume that $\interval{a}{b}$ and $\interval{c}{d}$ are unions of interwoven $\Pi$\hyp{}blocks.
If $\interval{a}{b} \cap \interval{c}{d} \neq \emptyset$, then $\interval{a}{b} = \interval{c}{d}$.
\end{lemma}

\begin{proof}
Let $p \in \interval{a}{b} \cap \interval{c}{d}$.
Since $\interval{a}{b} \interwoven \Pi$ and $\interval{c}{d} \interwoven \Pi$, we have $\eqclass{p}{\Pi} \subseteq \interval{a}{b}$.
Since $\card{\eqclass{p}{\Pi}} > 1$, there is an element $p' \in \eqclass{p}{\Pi}$ with $p \neq p'$.
Without loss of generality, assume $p < p'$.
Then $\interval{p}{p'} \subseteq \interval{a}{b} \cap \interval{c}{d}$.

Let now $x \in \interval{a}{b}$.
Since $\interval{a}{b} \interwoven \Pi$, $\eqclass{x}{\Pi}$ has a representative $q$ in the interval $\interval{p}{p'}$.
But then $\eqclass{q}{\Pi} \subseteq \interval{c}{d}$, because $\interval{c}{d} \interwoven \Pi$ and $q \in \interval{p}{p'} \subseteq \interval{c}{d}$.
Therefore $x \in \eqclass{q}{\Pi} \subseteq \interval{c}{d}$; hence $\interval{a}{b} \subseteq \interval{c}{d}$.
A similar argument shows that $\interval{c}{d} \subseteq \interval{a}{b}$.
\end{proof}

The following is a modification of Atkinson and Beals's Lemma~\ref{lem:AtkBea-L6}.

\begin{lemma}
\label{lem:AtkBea-L6-ext}
Let $\Pi$ be a partition of $\nset{n}$, and let $\pi \in \Compn[n+1]{\Aut \Pi}$.
Assume that $t \in \nset{n}$ is a number such that $t-1 \not\sim_\Pi t$ and either $t-2 \sim_\Pi t-1$ or $t \sim_\Pi t+1$.
Then $t$ satisfies condition \ref{AB-L6-1-a} or \ref{AB-L6-1-b} of Lemma~\ref{lem:AtkBea-L6}, that is,
\begin{enumerate}[label=\upshape{(\alph*)}]
\item\label{AB-L6-ext-1-a} $\pi(t) = t$ and $\pi$ maps each one of the intervals $\interval{1}{t-1}$ and $\interval{t+1}{n+1}$ onto itself, or
\item\label{AB-L6-ext-1-b}
$\pi(t) = \desc{n+1}(t) = n - t + 2$, $\pi(n - t + 2) = \desc{n+1}(n - t + 2) = t$, and $\pi$ interchanges the interval $\interval{1}{t-1}$ with $\desc{n+1}(\interval{1}{t-1}) = \interval{n - t + 3}{n+1}$ and the interval $\interval{t+1}{n+1}$ with $\desc{n+1}(\interval{t+1}{n+1}) = \interval{1}{n - t + 1}$.
\end{enumerate}
\end{lemma}

\begin{proof}
First we prove several claims.
We are going to make repeatedly use of the fact that $\subpermsimple{\pi}{i} \in \Aut \Pi$ for all $i \in \nset{n+1}$, because $\pi \in \Compn[n+1]{\Aut \Pi}$.

\begin{quotation}
\begin{claim}
\label{clm:t2t1asc}
If $t-2 \sim_\Pi t-1$ and $\pi(t-1) < \pi(t)$, then $\pi(i) < \pi(t)$ for all $i < t$ and $\pi(i) > \pi(t)$ for all $i > t$.
\end{claim}

\begin{proof}[Proof of Claim~\ref{clm:t2t1asc}]
Write $a := \pi(t-2)$, $b := \pi(t-1)$, $c := \pi(t)$, $d := \pi(t+1)$.
Our hypotheses assert that $b < c$.
Suppose, to the contrary, that $d < c$.
We proceed by case analysis on the different possible orderings of $a$, $b$, $c$, $d$; note that we must have $a < b < c$ or $b < a < c$ or $b < c < a$.
\begin{itemize}
\item If $a < b < d < c$,
then on the one hand $\subpermsimple{\pi}{t-1}(t-2) = a$, $\subpermsimple{\pi}{t-1}(t-1) = c-1$, so $a \sim_\Pi c-1$;
and on the other hand $\subpermsimple{\pi}{t+1}(t-2) = a$, $\subpermsimple{\pi}{t+1}(t) = c-1$, so $a \not\sim_\Pi c-1$,
a contradiction.

\item If $a < d < b < c$ or $d < a < b < c$,
then $\subpermsimple{\pi}{t-2}(t-2) = b-1$, $\subpermsimple{\pi}{t-2}(t-1) = c-1$, so $b-1 \sim_\Pi c-1$;
and $\subpermsimple{\pi}{t+1}(t-1) = b-1$, $\subpermsimple{\pi}{t+1}(t) = c-1$, so $b-1 \not\sim_\Pi c-1$,
a contradiction.

\item If $b < a < d < c$,
then $\subpermsimple{\pi}{t-2}(t-2) = b$, $\subpermsimple{\pi}{t-2}(t-1) = c-1$, so $b \sim_\Pi c-1$;
and $\subpermsimple{\pi}{t+1}(t-1) = b$, $\subpermsimple{\pi}{t+1}(t) = c-1$, so $b \not\sim_\Pi c-1$,
a contradiction.

\item If $b < d < a < c$ or $d < b < a < c$ or $b < d < c < a$ or $d < b < c < a$,
then $\subpermsimple{\pi}{t-1}(t-2) = a-1$, $\subpermsimple{\pi}{t-1}(t-1) = c-1$, so $a-1 \sim_\Pi c-1$;
and $\subpermsimple{\pi}{t+1}(t-2) = a-1$, $\subpermsimple{\pi}{t+1}(t) = c-1$, so $a-1 \not\sim_\Pi c-1$,
a contradiction.
\end{itemize}
We have thus established that $c < d$.

Suppose now, to the contrary, that $\pi(i) > c$ for some $i < t$.
Since we are assuming that $b < c$, we have $i < t-1$.
Then $\subpermsimple{\pi}{i}(t-2) = b$, $\subpermsimple{\pi}{i}(t-1) = c$, so $b \sim_\Pi c$;
and $\subpermsimple{\pi}{t+1}(t-1) = b$, $\subpermsimple{\pi}{t+1}(t) = c$, so $b \not\sim_\Pi c$,
a contradiction.

Finally, suppose, to the contrary, that $\pi(i) < c$ for some $i > t$.
Since $c < d$, we have $i > t+1$.
Let $e := \pi(i)$.
We proceed by case analysis on the possible orderings of $a$, $b$, $e$.
\begin{itemize}
\item If $a < b < e$ or $a < e < b$,
then $\subpermsimple{\pi}{t-1}(t-2) = a$, $\subpermsimple{\pi}{t-1}(t-1) = c-1$, so $a \sim_\Pi c-1$;
and $\subpermsimple{\pi}{i}(t-2) = a$, $\subpermsimple{\pi}{i}(t) = c-1$, so $a \not\sim_\Pi c-1$,
a contradiction.

\item If $b < a < e$ or $b < e < a$,
then $\subpermsimple{\pi}{t-2}(t-2) = b$, $\subpermsimple{\pi}{t-2}(t-1) = c-1$, so $b \sim_\Pi c-1$;
and $\subpermsimple{\pi}{i}(t-1) = b$, $\subpermsimple{\pi}{i}(t) = c-1$, so $b \not\sim_\Pi c-1$,
a contradiction.

\item If $e < a < b$,
then $\subpermsimple{\pi}{t-2}(t-2) = b-1$, $\subpermsimple{\pi}{t-2}(t-1) = c-1$, so $b-1 \sim_\Pi c-1$;
and $\subpermsimple{\pi}{i}(t-1) = b-1$, $\subpermsimple{\pi}{i}(t) = c-1$, so $b-1 \not\sim_\Pi c-1$,
a contradiction.

\item If $e < b < a$,
then $\subpermsimple{\pi}{t-1}(t-2) = a-1$, $\subpermsimple{\pi}{t-1}(t-1) = c-1$, so $a-1 \sim_\Pi c-1$;
and $\subpermsimple{\pi}{i}(t-2) = a-1$, $\subpermsimple{\pi}{i}(t) = c-1$, so $a-1 \not\sim_\Pi c-1$,
a contradiction.
\qedhere
\end{itemize}
\end{proof}

\begin{claim}
\label{clm:tt1asc}
If $t \sim_\Pi t+1$ and $\pi(t-1) < \pi(t)$, then $\pi(i) < \pi(t)$ for all $i < t$ and $\pi(i) > \pi(t)$ for all $i > t$.
\end{claim}

\begin{proof}[Proof of Claim~\ref{clm:tt1asc}]
Write $a := \pi(t-1)$, $b := \pi(t)$, $c := \pi(t+1)$, $d := \pi(t+2)$.
Our hypotheses assert that $a < b$.
Suppose, to the contrary, that $b > c$ or $b > d$.
We proceed by case analysis on the possible orderings of $a$, $b$, $c$, $d$.
\begin{itemize}
\item If $a < d < b < c$ or $a < d < c < b$ or $d < a < b < c$ or $d < a < c < b$,
then $\subpermsimple{\pi}{t-1}(t-1) = b-1$, $\subpermsimple{\pi}{t-1}(t) = c-1$, so $b-1 \not\sim_\Pi c-1$;
and $\subpermsimple{\pi}{t+2}(t) = b-1$, $\subpermsimple{\pi}{t+2}(t+1) = c-1$, so $b-1 \sim_\Pi c-1$,
a contradiction.

\item If $a < c < b < d$ or $a < c < d < b$ or $c < a < b < d$ or $c < a < d < b$,
then $\subpermsimple{\pi}{t-1}(t-1) = b-1$, $\subpermsimple{\pi}{t-1}(t+1) = d-1$, so $b-1 \not\sim_\Pi d-1$;
and $\subpermsimple{\pi}{t+1}(t) = b-1$, $\subpermsimple{\pi}{t+1}(t+1) = d-1$, so $b-1 \sim_\Pi d-1$,
a contradiction.

\item If $c < d < a < b$,
then $\subpermsimple{\pi}{t-1}(t-1) = b-1$, $\subpermsimple{\pi}{t-1}(t) = c$, so $b-1 \not\sim_\Pi c$;
and $\subpermsimple{\pi}{t+2}(t) = b-1$, $\subpermsimple{\pi}{t+2}(t+1) = c$, so $b-1 \sim_\Pi c$,
a contradiction.

\item If $d < c < a < b$,
then $\subpermsimple{\pi}{t-1}(t-1) = b-1$, $\subpermsimple{\pi}{t-1}(t+1) = d$, so $b-1 \not\sim_\Pi d$;
and $\subpermsimple{\pi}{t+1}(t) = b-1$, $\subpermsimple{\pi}{t+1}(t+1) = d$, so $b-1 \sim_\Pi d$,
a contradiction.
\end{itemize}
We have thus established that $b < c$ and $b < d$.

Suppose now, to the contrary, that $\pi(i) < b$ for some $i > t$.
Then in fact $i > t + 2$, and we have
$\subpermsimple{\pi}{t-1}(t-1) = b-1$, $\subpermsimple{\pi}{t-1}(t) = c-1$, so $b-1 \not\sim_\Pi c-1$;
and $\subpermsimple{\pi}{i}(t) = b-1$, $\subpermsimple{\pi}{i}(t+1) = c-1$, so $b-1 \sim_\Pi c-1$,
a contradiction.

Finally, suppose, to the contrary, that $\pi(i) > b$ for some $i < t$.
Then in fact $i < t-1$.
Let $e := \pi(i)$.
We proceed by case analysis on the possible orderings of $c$, $d$, $e$.
\begin{itemize}
\item If $e < c < d$ or $c < e < d$,
then $\subpermsimple{\pi}{i}(t-1) = b$, $\subpermsimple{\pi}{i}(t+1) = d-1$, so $b \not\sim_\Pi d-1$;
and $\subpermsimple{\pi}{t+1}(t) = b$, $\subpermsimple{\pi}{t+1}(t+1) = d-1$, so $b \sim_\Pi d-1$,
a contradiction.

\item If $e < d < c$ or $d < e < c$,
then $\subpermsimple{\pi}{i}(t-1) = b$, $\subpermsimple{\pi}{i}(t) = c-1$, so $b \not\sim_\Pi c-1$;
and $\subpermsimple{\pi}{t+2}(t) = b$, $\subpermsimple{\pi}{t+2}(t+1) = c-1$, so $b \sim_\Pi c-1$,
a contradiction.

\item If $c < d < e$,
then $\subpermsimple{\pi}{i}(t-1) = b$, $\subpermsimple{\pi}{i}(t) = c$, so $b \not\sim_\Pi c$;
and $\subpermsimple{\pi}{t+2}(t) = b$, $\subpermsimple{\pi}{t+2}(t+1) = c$, so $b \sim_\Pi c$,
a contradiction.

\item If $d < c < e$,
then $\subpermsimple{\pi}{i}(t-1) = b$, $\subpermsimple{\pi}{i}(t+1) = d$, so $b \not\sim_\Pi d$;
and $\subpermsimple{\pi}{t+1}(t) = b$, $\subpermsimple{\pi}{t+1}(t+1) = d$, so $b \sim_\Pi d$,
a contradiction.
\qedhere
\end{itemize}
\end{proof}

\begin{claim}
\label{clm:t2t1desc}
If $t-2 \sim_\Pi t-1$ and $\pi(t-1) > \pi(t)$, then $\pi(i) > \pi(t)$ for all $i < t$ and $\pi(i) < \pi(t)$ for all $i > t$.
\end{claim}

\begin{claim}
\label{clm:tt1desc}
If $t \sim_\Pi t+1$ and $\pi(t-1) > \pi(t)$, then $\pi(i) > \pi(t)$ for all $i < t$ and $\pi(i) < \pi(t)$ for all $i > t$.
\end{claim}

Claims~\ref{clm:t2t1desc} and \ref{clm:tt1desc} are proved in a similar way as Claims~\ref{clm:t2t1asc} and \ref{clm:tt1asc}, with the following changes:
in statements about the ordering of elements $a$, $b$, $c$, $d$, $e$, the order is reversed, i.e., the relation symbols ``$<$'' and ``$>$'' must be interchanged. Additionally, when we speak of the values of $\subpermsimple{\pi}{i}$ at various points, we must interchange $a$ with $a-1$, $b$ with $b-1$, $c$ with $c-1$, and $d$ with $d-1$.
\end{quotation}

By Claims~\ref{clm:t2t1asc}--\ref{clm:tt1desc}, we conclude that either
\begin{itemize}
\item $\pi(\interval{1}{t-1}) \subseteq \interval{1}{\pi(t)-1}$ and $\pi(\interval{t+1}{n+1} \subseteq \interval{\pi(t)+1}{n+1}$, or
\item $\pi(\interval{1}{t-1}) \subseteq \interval{\pi(t)+1}{n+1}$ and $\pi(\interval{t+1}{n+1} \subseteq \interval{1}{\pi(t)-1}$.
\end{itemize}
Since $\pi$ is a permutation of a finite set, a simple counting argument shows that either
\begin{enumerate}[label=\upshape{(\alph*)}]
\item\label{case:ext:t} $\pi(t) = t$, $\pi(\interval{1}{t-1}) = \interval{1}{t-1}$, and $\pi(\interval{t+1}{n+1}) = \interval{t+1}{n+1}$, or
\item\label{case:ext:n-t+2} $\pi(t) = n - t + 2$, $\pi(\interval{1}{t-1}) = \interval{n - t + 3}{n + 1}$, and $\pi(\interval{t+1}{n+1}) = \interval{1}{n - t + 1}$.
\end{enumerate}
If \ref{case:ext:t} holds, then we are done.
Assume that \ref{case:ext:n-t+2} holds.
If $t = n - t + 2$, then $\pi(n - t + 2) = \pi(t) = n - t + 2 = t$; furthermore, $\pi$ interchanges the intervals $\interval{1}{t-1} = \interval{1}{n - t + 1}$ and $\interval{n - t + 3}{n+1} = \interval{t+1}{n+1}$.
If $t \neq n - t + 2$, then consider the permutation $\tau := \pi \circ \pi$, which is also a member of $\Compn[n+1]{\Aut \Pi}$ by Proposition~\ref{prop:S-subgroup}.
It is not possible that $\tau(t) = n - t + 2$, because then we would have $\pi(t) = n - t + 2 = \tau(t) = \pi(\pi(t)) = \pi(n - t + 2)$, which would contradict the injectivity of $\pi$.
By what we have shown above, it must then hold that $\tau(t) = t$, $\tau(\interval{1}{t-1}) = \interval{1}{t-1}$, and $\tau(\interval{t+1}{n+1}) = \interval{t+1}{n+1}$.
Then
\begin{align*}
& \pi(n - t + 2) = \pi(\pi(t)) = \tau(t) = t, \\
& \pi(\interval{n - t + 3}{n+1}) = \pi(\pi(\interval{1}{t-1})) = \tau(\interval{1}{t-1}) = \interval{1}{t-1}, \\
& \pi(\interval{1}{n-t+1}) = \pi(\pi(\interval{t+1}{n+1})) = \tau(\interval{t+1}{n+1}) = \interval{t+1}{n+1}.
\qedhere
\end{align*}
\end{proof}

\begin{lemma}
\label{lem:AutPi-cycles}
Let $\Pi$ be a partition of $\nset{n}$, and assume that $\Aut \Pi$ contains the cycle $(a \; a+1 \; \cdots \; b)$ for some $a, b \in \nset{n}$ with $a < b$.
Then the interval $\interval{a}{b}$ is either contained in a $\Pi$\hyp{}block or it is a union of interwoven $\Pi$\hyp{}blocks.
\end{lemma}

\begin{proof}
Let $\pi := (a \; a+1 \; \cdots \; b)$.
Assume that $\interval{a}{b}$ is not contained in a $\Pi$\hyp{}block.
Then $\interval{a}{b}$ contains elements from at least two distinct $\Pi$\hyp{}blocks.
In particular, there must exist $\ell \in \interval{a}{b}$ such that $\ell \not\sim_\Pi \ell + 1$, where addition is done modulo $b - a + 1$ and representatives of equivalence classes are chosen from the interval $\interval{a}{b}$.
Since $\pi \in \Aut \Pi$, it follows that $\ell + 1 = \pi(\ell) \not\sim_\Pi \pi(\ell + 1) = \ell + 2$ (with modular addition, as explained above).
Repeating this argument, we conclude that $i \not\sim_\Pi i + 1$ for all $i \in \interval{a}{b}$.

It also follows that if $B$ is a $\Pi$\hyp{}block satisfying $B \cap \interval{a}{b} \neq \emptyset$, then $B \subseteq \interval{a}{b}$. (For, if there exist $B \in \Pi$ and $i \in B \cap \interval{a}{b}$, $j \in B \setminus \interval{a}{b}$, then $i \sim_\Pi j$ but $\pi(i) = i + 1 \not\sim_\Pi i \sim_\Pi j = \pi(j)$, a contradiction.)
Furthermore, if $B$ is a $\Pi$\hyp{}block satisfying $B \subseteq \interval{a}{b}$, then
also $\pi(B)$
is a $\Pi$\hyp{}block contained in $\interval{a}{b}$.
Since $\interval{a}{b}$ is a finite interval, there exists a positive integer $m$ such that $\pi^m(B) \cap B \neq \emptyset$.
Let $k$ be the smallest such number; it is clear that $2 \leq k \leq b - a + 1$.
Then in fact $\pi^k(B) = B$, so for every $i \in B$, also $\pi^k(i) = i + k \pmod{b - a + 1}$ is in $B$.
It is then easy to see that $k$ is a divisor of $b - a + 1$, and $B$ is of the form \eqref{eq:interval-blocks}.
Consequently, $\pi(B)$ is also of the form \eqref{eq:interval-blocks}.
We conclude that $\interval{a}{b}$ is partitioned into $\Pi$\hyp{}blocks of the form prescribed in \eqref{eq:interval-blocks}, that is, $\interval{a}{b}$ is a union of interwoven $\Pi$\hyp{}blocks.
\end{proof}

\begin{lemma}
\label{lem:intervals-blocks}
Let $\Pi$ be a partition of $\nset{n}$ with no trivial blocks.
Assume that $a, b, c, d \in \nset{n}$ are numbers such that $a < b$ and $c < d$ and $\interval{a}{b} \cap \interval{c}{d} \neq \emptyset$.
If each one of $\interval{a}{b}$ and $\interval{c}{d}$ is either contained in a $\Pi$\hyp{}block or is a union of interwoven $\Pi$\hyp{}blocks, then $\interval{a}{b} \cup \interval{c}{d}$ is contained in a $\Pi$\hyp{}block or $\interval{a}{b} = \interval{c}{d}$.
\end{lemma}

\begin{proof}
Consider first the case where each one of $\interval{a}{b}$ and $\interval{c}{d}$ is contained in a $\Pi$\hyp{}block, say $\interval{a}{b} \subseteq B$, $\interval{c}{d} \subseteq B'$, where $B, B' \in \Pi$.
Since $\emptyset \neq \interval{a}{b} \cap \interval{c}{d} \subseteq B \cap B'$ and since the blocks of the partition $\Pi$ are pairwise disjoint, it follows that $B = B'$.
Consequently, $\interval{a}{b} \cup \interval{c}{d} \subseteq B$.

The case where both $\interval{a}{b}$ and $\interval{c}{d}$ are unions of interwoven $\Pi$\hyp{}blocks is settled by
Lemma~\ref{lem:interwoven-overlap}, which asserts that in this case it holds that $\interval{a}{b} = \interval{c}{d}$.

Consider finally the case where $\interval{a}{b}$ is a union of interwoven $\Pi$\hyp{}blocks and $\interval{c}{d}$ is contained in a single $\Pi$\hyp{}block $B$.
Since $\interval{a}{b}$ and $\interval{c}{d}$ are not disjoint, the intersection $\interval{a}{b} \cap \interval{c}{d}$ is a nonempty interval.
If $\interval{a}{b} \cap \interval{c}{d}$ has at least two elements, then the interval $\interval{a}{b}$ contains two consecutive elements belonging to the block $B$; hence $\interval{a}{b} \notinterwoven \Pi$, a contradiction.
We may thus assume that $\interval{a}{b} \cap \interval{c}{d}$ has exactly one element.
Since $\Pi$ has no trivial blocks, this implies that the $\Pi$\hyp{}block $B$ contains an element that is in $\interval{a}{b}$ and another element that is not in $\interval{a}{b}$.
Therefore, $\interval{a}{b}$ is not a union of $\Pi$\hyp{}blocks, again a contradiction.

In the case where $\interval{a}{b}$ is contained in a single $\Pi$\hyp{}block and $\interval{c}{d}$ is a union of interwoven $\Pi$\hyp{}blocks, we reach a contradiction in a similar way as above.
\end{proof}

\begin{corollary}
\label{cor:pi-jumps}
Let $\Pi$ be a partition of $\nset{n}$.
Let $\pi \in \Compn[n+1]{\Aut \Pi}$.
Let $t, u \in \nset{n}$ and assume that $a, b, c, d \in \nset{n+1}$ are numbers such that $a \leq b-2$ and $c \leq d-2$.
Then the following statements hold.
\begin{enumerate}[label=\upshape{(\roman*)}]
\item
If $\pi$ has a jump between $a$ and $b$ at $t$, then either the interval $\interval{a}{b-1}$ is contained in a single $\Pi$\hyp{}block or $\interval{a}{b-1}$ is a union of interwoven $\Pi$\hyp{}blocks.
\item
Assume that $\Pi$ has no trivial blocks,
$\pi$ has a jump between $a$ and $b$ at $t$, and $\pi$ has a jump between $c$ and $d$ at $u$.
If $\interval{a}{b-1} \cap \interval{c}{d-1} \neq \emptyset$, then either $\interval{a}{b-1} = \interval{c}{d-1}$ \textup{(}and hence $t = u$\textup{)} or $\interval{a}{b-1} \cup \interval{c}{d-1}$ is contained in a single $\Pi$\hyp{}block.
\end{enumerate}
\end{corollary}

\begin{proof}
By Lemma~\ref{lem:AtkBea-L11}, the group $\Aut \Pi$ contains the cycles $(a \; a+1 \; \cdots \; b-1)$ and $(c \; c+1 \; \cdots \; d-1)$.
By Lemma~\ref{lem:AutPi-cycles}, each one of the intervals $\interval{a}{b-1}$ and $\interval{c}{d-1}$ is either contained in a $\Pi$\hyp{}block or is a union of interwoven $\Pi$\hyp{}blocks.
If $\interval{a}{b-1} \neq \interval{c}{d-1}$, then it follows from Lemma~\ref{lem:intervals-blocks} that $\interval{a}{b-1} \cup \interval{c}{d-1}$ is contained in a $\Pi$\hyp{}block.
(For the statement in round brackets, it clearly holds that $\interval{a}{b-1} = \interval{c}{d-1}$ if and only if $t = u$, because $\pi$ is a permutation.)
\end{proof}

\begin{lemma}
\label{lem:short-jump}
Let $p, q, r, s \in \nset{n+1}$ with $p < q$, $r < s$, and assume that $\pi(\interval{p}{q}) = \interval{r}{s}$.
If there exist $a, b \in \interval{r}{s}$ with $r < a < b < s$ and $a \leq b-2$ such that $\pi$ has a jump between $a$ and $b$ in $\interval{p}{q-1}$, then there exist $c, d \in \interval{r}{s}$ with $c \leq d-2$ and $a \neq c$ and $b \neq d$ such that $\pi$ has a jump between $c$ and $d$ in $\interval{p}{q-1}$ and $\interval{a}{b-1} \cap \interval{c}{d-1} \neq \emptyset$.
\end{lemma}

\begin{proof}
Let $t \in \interval{p}{q-1}$ be the point at which $\pi$ has a jump between $a$ and $b$.
The three intervals $I_1 := \interval{r}{a-1}$, $I_2 := \interval{a+1}{b-1}$, $I_3 := \interval{b+1}{s}$ are nonempty, and at least one of the intervals $J_1 := \interval{p}{t-1}$ and $J_2 := \interval{t+2}{q}$ is nonempty.
We have
\[
\pi(J_1) \cup \pi(J_2) = \pi(J_1 \cup J_2)) = \pi(\interval{p}{q} \setminus \{t, t+1\}) = \interval{r}{s} \setminus \{a,b\} = I_1 \cup I_2 \cup I_2.
\]
Therefore, $\pi(J_1)$ or $\pi(J_2)$ must have a nonempty intersection with at least two of the sets $I_1$, $I_2$, $I_3$.
Consequently, there exist numbers $c \in I_i$, $d \in I_j$ with $1 \leq i < j \leq 3$ such that $\pi$ has a jump between $c$ and $d$ in $\interval{p}{q-1}$.
We also have that $\interval{a}{b-1} \cap \interval{c}{d-1} \neq \emptyset$.
\end{proof}

We say that an interval $\interval{a}{b} \subseteq \nset{n}$ is a \emph{singleton run} of $\maxintervals{\Pi}$ if for every $i \in \interval{a}{b}$, the singleton $\{i\}$ is a $\maxintervals{\Pi}$\hyp{}block. We say that a singleton run $\interval{a}{b}$ of $\maxintervals{\Pi}$ is \emph{maximal}
if no proper superset of $\interval{a}{b}$ is a singleton run of $\maxintervals{\Pi}$.

\begin{lemma}
\label{lem:pi-jumps-extremes}
Let $\Pi$ be a partition of the set $\nset{n}$ with no trivial blocks, and let $\pi \in \Compn[n+1]{\Aut \Pi}$.
Let $p, q, r, s \in \nset{n+1}$ be such that $p < q$, $r < s$, and assume that $\pi(\interval{p}{q}) = \interval{r}{s}$ and $\interval{p}{q-1}$ is a singleton run of $\maxintervals{\Pi}$.
Assume that $a, b, c, d \in \interval{r}{s}$ satisfy $a < b$ and $c < d$.
Then the following statements hold.
\begin{enumerate}[label=\upshape{(\roman*)}]
\item\label{item:jump-interwoven}
If $\pi$ has a jump between $a$ and $b$ in $\interval{p}{q-1}$, then $\interval{a}{b-1} \interwoven \Pi$.

\item\label{item:overlapping-jumps}
If $\pi$ has a jump between $a$ and $b$ and between $c$ and $d$ in $\interval{p}{q-1}$ and $\interval{a}{b-1} \cap \interval{c}{d-1} \neq \emptyset$, then $(a,b) = (c,d)$.

\item\label{item:jump-between-r-or-s}
If $\pi$ has a jump between $a$ and $b$ in $\interval{p}{q-1}$, then $a = r$ or $b = s$.

\item\label{item:no-jump-at-p-q}
If $\pi(p) \in \{r, s\}$ and $\pi$ has a jump between $a$ and $b$ at $p$, then $(a,b) = (r,s)$.
If $\pi(q) \in \{r, s\}$ and $\pi$ has a jump between $a$ and $b$ at $q-1$, then $(a,b) = (r,s)$.
\end{enumerate}
\end{lemma}

\begin{proof}
\ref{item:jump-interwoven}
By Corollary~\ref{cor:pi-jumps}, $\interval{a}{b-1} \subseteq B$ for some $B \in \Pi$ or $\interval{a}{b-1} \interwoven \Pi$.
Suppose, to the contrary, that $\interval{a}{b-1} \subseteq B$ for some $B \in \Pi$; hence $a \sim_\Pi b-1$.
Let $i := \pi^{-1}(a+1)$.
If $\pi(t) = a$ and $\pi(t+1) = b$, then either $\subpermsimple{\pi}{i}(t-1) = a$ and $\subpermsimple{\pi}{i}(t) = b-1$ (if $i < t$) or $\subpermsimple{\pi}{i}(t) = a$ and $\subpermsimple{\pi}{i}(t+1) = b-1$ (if $i > t$).
If $\pi(t) = b$ and $\pi(t+1) = a$, then either $\subpermsimple{\pi}{i}(t-1) = b-1$ and $\subpermsimple{\pi}{i}(t) = a$ (if $i < t$) or $\subpermsimple{\pi}{i}(t) = b-1$ and $\subpermsimple{\pi}{i}(t+1) = a$ (if $i > t$).
In either case, $\subpermsimple{\pi}{i}$ maps points from distinct $\Pi$\hyp{}blocks to the same $\Pi$\hyp{}block.
Therefore, $\subpermsimple{\pi}{i} \notin \Aut \Pi$, so $\pi \notin \Compn[n+1]{\Aut \Pi}$, a contradiction.
We conclude that $\interval{a}{b-1} \interwoven \Pi$.

\ref{item:overlapping-jumps}
By Corollary~\ref{cor:pi-jumps}, either $\interval{a}{b-1} = \interval{c}{d-1}$ or $\interval{a}{b-1} \cup \interval{c}{d-1}$ is contained in a single $\Pi$\hyp{}block.
By part~\ref{item:jump-interwoven}, we have $\interval{a}{b-1} \interwoven \Pi$ and $\interval{c}{d-1} \interwoven \Pi$.
Therefore, $\interval{a}{b-1} \cup \interval{c}{d-1}$ is not contained in a single $\Pi$\hyp{}block.
It follows that $\interval{a}{b-1} = \interval{c}{d-1}$, that is, $(a,b) = (c,d)$.

\ref{item:jump-between-r-or-s}
Suppose, to the contrary, that $r < a$ and $b < s$.
By Lemma~\ref{lem:short-jump}, there exist $c, d \in \interval{r}{s}$ such that $(a, b) \neq (c, d)$, $\interval{a}{b-1} \cap \interval{c}{d-1} \neq \emptyset$ and $\pi$ has a jump between $c$ and $d$ in $\interval{p}{q-1}$.
It follows from part~\ref{item:overlapping-jumps} that $(a,b) = (c,d)$.
We have reached a contradiction.

\ref{item:no-jump-at-p-q}
Suppose, to the contrary, that $\pi(p) = r$ and $\pi(p+1) = b$ or $\pi(q) = r$ and $\pi(q-1) = b$ for some $b$ such that $r+2 \leq b < q$.
Then there must exist $c, d \in \interval{r}{s}$ such that $c < b < d$ and $\pi$ has a jump between $c$ and $d$ in $\interval{p}{q-1}$.
Since $\interval{r}{b-1} \cap \interval{c}{d-1} \neq \emptyset$, part~\ref{item:overlapping-jumps} implies that $(r,b) = (c,d)$.
We have reached a contradiction.
In a similar way, we reach a contradiction if we suppose that $\pi(p) = s$ and $\pi(p+1) = a$ or $\pi(q) = s$ and $\pi(q-1) = s$ for some $a$ such that $r < a \leq s-2$.
\end{proof}

\begin{lemma}
\label{lem:singleton-intervals}
Let $\Pi$ be a partition of $\nset{n}$ with no trivial blocks.
Assume that $\interval{p}{q} \subseteq \nset{n}$ is a nonempty maximal singleton run of $\maxintervals{\Pi}$.
Let $\pi \in \Compn[n+1]{\Aut \Pi}$.
Then one of the following conditions is satisfied.
\begin{enumerate}[label=\upshape{(\alph*)}]
\item\label{item:nojumps}
$\pi(i) = i$ for all $i \in \interval{p}{q+1}$ or $\pi(i) = \desc{n+1}(i)$ for all $i \in \interval{p}{q+1}$.

\item\label{item:jump-1q}
$p = 1$, there exists $\ell \in \interval{1}{q}$ such that $\interval{1}{\ell} \interwoven \Pi$, and
\begin{itemize}
\item $\pi(i) = \dja{n+1}{\ell}(i)$ for all $i \in \interval{1}{q+1}$, or
\item $\pi(i) = (\desc{n+1} \circ \dja{n+1}{\ell})(i)$ for all $i \in \interval{1}{q+1}$ and $\interval{n - \ell + 1}{n} \interwoven \Pi$.
\end{itemize}

\item\label{item:jump-pn+1}
$q = n+1$, there exists $\ell \in \interval{p}{n}$ such that $\interval{\ell}{n} \interwoven \Pi$, and
\begin{itemize}
\item $\pi(i) = \ajd{n+1}{n-\ell+1}(i)$ for all $i \in \interval{p}{n+1}$, or
\item $\pi(i) = (\desc{n+1} \circ \ajd{n+1}{n-\ell+1})(i)$ for all $i \in \interval{p}{n+1}$ and $\interval{1}{n - \ell + 1} \interwoven \Pi$.
\end{itemize}

\item\label{item:jump-1n+1}
$p = 1$, $q = n+1$, there exist $\ell, \ell' \in \nset{n}$ with $\ell < \ell'$ such that $\interval{1}{\ell} \interwoven \Pi$, $\interval{\ell'}{n} \interwoven \Pi$, and
\begin{itemize}
\item $\pi = \dja{n+1}{\ell} \circ \ajd{n+1}{n-\ell'+1}$, or
\item $\pi = \desc{n+1} \circ \dja{n+1}{\ell} \circ \ajd{n+1}{n-\ell'+1}$ and $\ell = n - \ell' + 1$.
\end{itemize}

\item\label{item:jump-cycle}
$p = 1$, $q = n+1$, $\interval{1}{n} \interwoven \Pi$, and $\pi = \natcycle{n+1}^j$ or $\pi = \desc{n+1} \circ \natcycle{n+1}^j$ for some $j \in \nset{n}$.
\end{enumerate}
\end{lemma}

\begin{proof}
Assume first that $p > 1$ and $q < n$.
From the assumption that $\interval{p}{q}$ is a maximal singleton run of $\maxintervals{\Pi}$, we can conclude using Lemmas~\ref{lem:AtkBea-L6-2} and \ref{lem:AtkBea-L6-ext} that either
\begin{enumerate}[label=\upshape{(\roman*)}]
\item $\pi(p) = p$, $\pi(q+1) = q+1$, $\pi(\interval{p+1}{q}) = \interval{p+1}{q}$, or
\item $\pi(p) = \desc{n+1}(p)$, $\pi(q+1) = \desc{n+1}(q+1)$, $\pi(\interval{p+1}{q}) = \desc{n+1}(\interval{p+1}{q})$.
\end{enumerate}
Suppose, to the contrary, that $\pi$ has a jump between $a$ and $b$ at some $t \in \interval{p}{q}$ ($a \leq b-2$).
Let us apply Lemma~\ref{lem:pi-jumps-extremes} with $r := p$ and $s := q+1$ when case (i) holds, or with $r := \desc{n+1}(q+1)$ and $s := \desc{n+1}(p)$ when case (ii) holds.
By Lemma~\ref{lem:pi-jumps-extremes}\ref{item:jump-between-r-or-s}, $a = r$ or $b = s$.
Since $\pi(p) = r$ and $\pi(q+1) = s$, this means that $\pi$ has a jump between $a$ and $b$ at $p$ or at $q$.
Then it follows from Lemma~\ref{lem:pi-jumps-extremes}\ref{item:no-jump-at-p-q} that $a = r$ and $b = s$.
This is clearly not possible, because $r$ and $s$ are nonconsecutive integers.
We have reached a contradiction, and we conclude that $\pi$ does not have any jumps in $\interval{p}{q}$, that is, condition \ref{item:nojumps} is satisfied.

Assume then that $p = 1$ and $q < n$.
As above, we can conclude using Lemmas~\ref{lem:AtkBea-L6-ext} and \ref{lem:AtkBea-L6-2} that either
\begin{enumerate}[label=\upshape{(\roman*)}]
\item $\pi(q+1) = q+1$, $\pi(\interval{1}{q}) = \interval{1}{q}$, or
\item $\pi(q+1) = \desc{n+1}(q+1)$, $\pi(\interval{1}{q}) = \desc{n+1}(\interval{1}{q})$.
\end{enumerate}
If $\pi$ does not have any jumps in $\interval{1}{q}$, then either $\pi(i) = i$ for all $i \in \interval{1}{q+1}$ or $\pi(i) = \desc{n+1}(i)$ for all $i \in \interval{1}{q+1}$, so condition \ref{item:nojumps} is satisfied. We may thus assume that $\pi$ has a jump between $a$ and $b$ at some $t \in \interval{1}{q}$.

If case (i) applies, then application of Lemma~\ref{lem:pi-jumps-extremes} with $r := 1$ and $s := q+1$ yields $a = 1$ or $b = q+1$; moreover, if $b = q+1$, then $a = 1$.
Thus $\pi$ has a jump between $1$ and $\ell + 1$ in $\interval{1}{q}$ for some $\ell \in \interval{1}{q}$, and this is the only jump of $\pi$ in the interval $\interval{1}{q}$. Consequently, $\pi(i) = \dja{n+1}{\ell}(i)$ for all $i \in \interval{1}{q+1}$ and $\interval{1}{\ell} \interwoven \Pi$, so condition \ref{item:jump-1q} is satisfied.

If case (ii) applies, then application of Lemma~\ref{lem:pi-jumps-extremes} with $r := \desc{n+1}(q+1)$ and $s := n+1$ yields $a = \desc{n+1}(q+1)$ or $b = n+1$; moreover, if $a = \desc{n+1}(q+1)$, then $b = n+1$ by Lemma~\ref{lem:pi-jumps-extremes}\ref{item:no-jump-at-p-q}.
We conclude that $\pi$ has a jump between $\desc{n+1}(\ell + 1)$ and $n+1$ in $\interval{1}{q}$ for some $\ell \in \interval{1}{q}$, and this is the only jump of $\pi$ in $\interval{1}{q}$. Consequently, $\pi(i) = \desc{n+1} \circ \dja{n+1}{\ell}(i)$ for all $i \in \interval{1}{q+1}$ and $\interval{\desc{n+1}(\ell+1)}{n} = \interval{n-\ell+1}{n} \interwoven \Pi$.

Consider the $n$\hyp{}pattern $\subpermsimple{\pi}{q+1}$ of $\pi$.
We have that $\subpermsimple{\pi}{q+1}(i) = \desc{n} \circ \dja{n}{\ell}(i) = n - \ell + i$ for all $i \in \interval{1}{\ell}$.
Since $\subpermsimple{\pi}{q+1} \in \Aut \Pi$, the preimage of every $\Pi$\hyp{}block under $\subpermsimple{\pi}{q+1}$ is a $\Pi$\hyp{}block.
In particular, for each $\Pi$\hyp{}block $B$ contained in the interval $\interval{n-\ell+1}{n}$ (which is a union of interwoven $\Pi$\hyp{}blocks), we have $\subpermsimple{\pi}{q+1}^{-1}(B) = B - (n - \ell)$. It is thus easy to see that $\interval{1}{\ell} \interwoven \Pi$.
We conclude that condition \ref{item:jump-1q} is satisfied.

Assume then that $p > 1$ and $q = n$.
A similar proof as in the previous case shows that condition \ref{item:nojumps} or \ref{item:jump-pn+1} is satisfied.

Finally, assume that $p = 1$ and $q = n$.
If $\pi$ does not have any jumps, then clearly $\pi = \asc{n+1}$ or $\pi = \desc{n+1}$, so condition \ref{item:nojumps} is satisfied.
If $\pi$ has exactly one jump, say, between $a$ and $b$ ($a \leq b-2$) and $(a,b) \neq (1, n+1)$, then the situation is similar to the ones considered above, and we can conclude that either condition \ref{item:jump-1q} or \ref{item:jump-pn+1} is satisfied.
If $\pi$ has a jump between $1$ and $n+1$, then $\pi$ cannot have any other jumps by Corollary~\ref{cor:pi-jumps}. In this case, it is easy to see that $\pi = \natcycle{n+1}^j$ or $\pi = \desc{n+1} \circ \natcycle{n+1}^j$ for some $j \in \{1, \dots, n\}$. Moreover, $\interval{1}{n} \interwoven \Pi$, so condition \ref{item:jump-cycle} is satisfied.
If $\pi$ has exactly two jumps, then by Lemma~\ref{lem:pi-jumps-extremes} and Corollary~\ref{cor:pi-jumps}, the jumps must be between $1$ and $a$ and $b$ and $n+1$ for some $a, b \in \nset{n}$ with $a \leq b$.
It is easy to see that then condition \ref{item:jump-1n+1} is satisfied.
It follows from Lemma~\ref{lem:pi-jumps-extremes} and Corollary~\ref{cor:pi-jumps} that $\pi$ has at most two jumps, so we have exhausted all possibilities.
\end{proof}

\begin{theorem}
\label{thm:CompAutPi}
Let $\Pi$ be a partition of $\nset{n}$ with no trivial blocks.
Then
\[
\Compn[n+1]{\Aut \Pi} =
\begin{cases}
\gensg{\symm{\Pi'}, E_\Pi}, & \text{if $\desc{n} \notin \Aut \Pi$,} \\
\gensg{\symm{\Pi'}, E_\Pi, \desc{n+1}}, & \text{if $\desc{n} \in \Aut \Pi$,}
\end{cases}
\]
where $E_\Pi$ is the set of permutations satisfying the following conditions:
\begin{itemize}
\item If $\interval{1}{\ell} \interwoven \Pi$ for some $\ell$ with $1 < \ell < n$, then $\dja{n+1}{\ell} \in E_\Pi$.
\item If $\interval{m}{n} \interwoven \Pi$ for some $m$ with $1 < m < n$, then $\ajd{n+1}{n-m+1} \in E_\Pi$.
\item If $\interval{1}{n} \interwoven \Pi$, then $\natcycle{n+1} \in E_\Pi$.
\item $E_\Pi$ does not contain any other elements than the ones implied by the previous conditions.
\end{itemize}
\end{theorem}

\begin{proof}
To prove that the group $\gensg{\symm{\Pi'}, E_\Pi}$ or $\gensg{\symm{\Pi'}, E_\Pi, \desc{n+1}}$, as the case may be, is included in $\Compn[n+1]{\Aut \Pi}$, it suffices to show that its generators are included. If $\desc{n} \in \Aut \Pi$, then $\desc{n+1} \in \Compn[n+1]{\Aut \Pi}$ by Lemma~\ref{lem:descending}.
Since $\symm{\Pi} \leq \Aut \Pi$, we have $\Compn[n+1]{\symm{\Pi}} \leq \Compn[n+1]{\Aut \Pi}$.
By Proposition~\ref{prop:CompSPi}, $\symm{\Pi'} \leq \Compn[n+1]{\symm{\Pi}}$.
Hence, $\symm{\Pi'} \subseteq \Compn[n+1]{\Aut \Pi}$.

Concerning the elements of $E_\Pi$, it is easy to verify that
\begin{itemize}
\item $\Patl[n]{\dja{n+1}{\ell}} = \{\dja{n}{\ell}, \dja{n}{\ell-1}\}$ and this is a subset of $\Aut \Pi$ if $\interval{1}{\ell} \interwoven \Pi$,
\item $\Patl[n]{\ajd{n+1}{\ell}} = \{\ajd{n}{\ell}, \ajd{n}{\ell-1}\}$ and this is a subset of $\Aut \Pi$ if $\interval{n - \ell + 1}{n} \interwoven \Pi$,
\item $\Patl[n]{\natcycle{n+1}} = \{\natcycle{n}, \asc{n}\}$ and this is a subset of $\Aut \Pi$ if $\interval{1}{n} \interwoven \Pi$.
\end{itemize}

We conclude that if $\desc{n} \notin \Aut \Pi$ then $\gensg{\symm{\Pi'}, E_\Pi} \subseteq \Compn[n+1]{\Aut \Pi}$, and if $\desc{n} \in \Aut \Pi$ then $\gensg{\symm{\Pi'}, E_\Pi, \desc{n+1}} \subseteq \Compn[n+1]{\Aut \Pi}$.

In order to prove the converse inclusion, let $\pi \in \Compn[n+1]{\Aut \Pi}$.
If $\maxintervals{\Pi}$ has no nontrivial blocks, then, by Lemma~\ref{lem:singleton-intervals}, one of the following conditions is satisfied:
\begin{itemize}
\item $\pi = \asc{n+1}$ or $\pi = \desc{n+1}$,
\item there exists $\ell \in \nset{n}$ such that $\interval{1}{\ell} \interwoven \Pi$ and $\pi = \dja{n+1}{\ell}$ or $\pi = \desc{n+1} \circ \dja{n+1}{\ell}$,
\item there exists $m \in \nset{n}$ such that $\interval{m}{n} \interwoven \Pi$ and $\pi = \ajd{n+1}{n-m+1}$ or $\pi = \desc{n+1} \circ \ajd{n+1}{n-m+1}$,
\item there exist $\ell, m \in \nset{n}$ such that $\interval{1}{\ell} \interwoven \Pi$, $\interval{m}{n} \interwoven \Pi$, and $\pi = \dja{n+1}{\ell} \circ \ajd{n+1}{n-m+1}$ or $\pi = \desc{n+1} \circ \dja{n+1}{\ell} \circ \ajd{n+1}{n-m+1}$,
\item $\interval{1}{n} \interwoven \Pi$ and $\pi = \natcycle{n+1}^j$ or $\pi = \desc{n+1} \circ \natcycle{n+1}^j$ for some $j \in \nset{n}$.
\end{itemize}
In each case, $\pi$ is in $\gensg{\symm{\Pi'},E_\Pi}$ or $\gensg{\symm{\Pi'},E_\Pi,\desc{n+1}}$, as appropriate.

Assume then that $\maxintervals{\Pi}$ has nontrivial blocks.
Then it follows from Lemma~\ref{lem:AtkBea-L6-ext} and \ref{lem:AtkBea-L6-2} that either statement \ref{stat:A} below holds for every nontrivial block $B$ of $\maxintervals{\Pi}$, or statement \ref{stat:B} below holds for every nontrivial block $B$ of $\maxintervals{\Pi}$.
\begin{enumerate}[label=\upshape{(\alph*)}]
\item\label{stat:A} Let $t := \min B$, $u := \max B + 1$.
\begin{itemize}
\item If $1, n \notin B$ then $\pi(t) = t$, $\pi(u) = u$, $\pi(B \setminus \{t\}) = B \setminus \{t\}$.
\item If $1 \in B$, $n \notin B$ then $\pi(u) = u$, $\pi(B) = B$.
\item If $1 \notin B$, $n \in B$ then $\pi(t) = t$, $\pi(B \setminus \{t\} \cup \{n+1\}) = B \setminus \{t\} \cup \{n+1\}$.
\end{itemize}

\item\label{stat:B} Let $t := \min B$, $u := \max B + 1$.
\begin{itemize}
\item If $1, n \notin B$ then $\pi(t) = \desc{n+1}(t)$, $\pi(u) = \desc{n+1}(u)$, $\pi(B \setminus \{t\}) = \desc{n+1}(B \setminus \{t\})$.
\item If $1 \in B$, $n \notin B$ then $\pi(u) = \desc{n+1}(u)$, $\pi(B) = \desc{n+1}(B)$.
\item If $1 \notin B$, $n \in B$ then $\pi(t) = \desc{n+1}(t)$, $\pi(B \setminus \{t\} \cup \{n+1\}) = \desc{n+1}(B \setminus \{t\} \cup \{n+1\})$.
\end{itemize}
\end{enumerate}

Writing
\[
T := \bigcup \{ B \cup (B+1) \mid \text{$B \in \maxintervals{\Pi}$ and $\card{B} > 1$}\},
\]
we can rephrase the above statements \ref{stat:A} and \ref{stat:B} as follows:
\begin{enumerate}
\item[\ref{stat:A}] The restriction of $\pi$ to $T$ coincides with the restriction of $\tau$ to $T$ (i.e., $\pi|_T = \tau|_T$) for some $\tau \in \symm{\Pi'}$.
\item[\ref{stat:B}] The restriction of $\pi$ to $T$ coincides with the restriction of $\desc{n+1} \circ \tau$ to $T$ (i.e., $\pi|_T = (\desc{n+1} \circ \tau)|_T$) for some $\tau \in \symm{\Pi'}$.
\end{enumerate}

Note that for every $i \in \nset{n+1} \setminus T$, the singleton $\{i\}$ is a $\Pi'$\hyp{}block. Hence, for every $\tau \in \symm{\Pi'}$, the restriction of $\tau$ to $\nset{n+1} \setminus T$ is the identity map on $\nset{n+1} \setminus T$ (i.e., $\tau|_{\nset{n+1} \setminus T} = \asc{n+1}|_{\nset{n+1} \setminus T}$).
Note also that if $\interval{1}{\ell} \interwoven \Pi$, then $\interval{1}{\ell} \cap T = \emptyset$.
Similarly, if $\interval{n-m+1}{n} \interwoven \Pi$, then $\interval{n-m+2}{n+1} \cap T = \emptyset$.

Based on this, we can deduce from Lemma~\ref{lem:singleton-intervals} that 
if statement \ref{stat:A} is true, then one of the following holds:
\begin{itemize}
\item $\pi|_{\nset{n+1} \setminus T} = \asc{n+1}|_{\nset{n+1} \setminus T}$,
\item $\pi|_{\nset{n+1} \setminus T} = \dja{n+1}{\ell}|_{\nset{n+1} \setminus T}$ and $\interval{1}{\ell} \interwoven \Pi$,
\item $\pi|_{\nset{n+1} \setminus T} = \ajd{n+1}{\ell}|_{\nset{n+1} \setminus T}$ and $\interval{n-\ell+1}{n} \interwoven \Pi$,
\item $\pi|_{\nset{n+1} \setminus T} = (\dja{n+1}{\ell} \circ \ajd{n+1}{m})|_{\nset{n+1} \setminus T}$ and $\interval{1}{\ell} \interwoven \Pi$ and $\interval{n-m+1}{n} \interwoven \Pi$.
\end{itemize}
Hence $\pi$ is of the form
$\tau$, $\dja{n+1}{\ell} \circ \tau$, $\ajd{n+1}{\ell} \circ \tau$, or $\dja{n+1}{\ell} \circ \ajd{n+1}{m} \circ \tau$ for some $\tau \in \symm{\Pi'}$.
Moreover, $E_\Pi$ contains $\dja{n+1}{\ell}$ or $\ajd{n+1}{\ell}$ or both or neither, as appropriate, so that we can conclude that $\pi \in \gensg{\symm{\Pi'}, E_\Pi}$.

Similarly, if statement \ref{stat:B} is true, then one of the following holds:
\begin{itemize}
\item $\pi|_{\nset{n+1} \setminus T} = \desc{n+1}|_{\nset{n+1} \setminus T}$,
\item $\pi|_{\nset{n+1} \setminus T} = (\desc{n+1} \circ \dja{n+1}{\ell})|_{\nset{n+1} \setminus T}$ and $\interval{1}{\ell} \interwoven \Pi$,
\item $\pi|_{\nset{n+1} \setminus T} = (\desc{n+1} \circ \ajd{n+1}{\ell})|_{\nset{n+1} \setminus T}$ and $\interval{n-\ell+1}{n} \interwoven \Pi$,
\item $\pi|_{\nset{n+1} \setminus T} = (\desc{n+1} \circ \dja{n+1}{\ell} \circ \ajd{n+1}{m})|_{\nset{n+1} \setminus T}$, $\interval{1}{\ell} \interwoven \Pi$, and $\interval{n-m+1}{n} \interwoven \Pi$.
\end{itemize}
Hence $\pi$ is of one of the forms
$\desc{n+1} \circ \tau$, $\desc{n+1} \circ \dja{n+1}{\ell} \circ \tau$, $\desc{n+1} \circ \ajd{n+1}{\ell} \circ \tau$, or $\desc{n+1} \circ \dja{n+1}{\ell} \circ \ajd{n+1}{m} \circ \tau$ for some $\tau \in \symm{\Pi'}$.
Moreover, $E_\Pi$ contains $\dja{n+1}{\ell}$ or $\ajd{n+1}{\ell}$ or both or neither, as appropriate, so that we can conclude that $\pi \in \gensg{\symm{\Pi'}, E_\Pi, \desc{n+1}}$.
We still need to verify that this case occurs only if $\desc{n} \in \Aut \Pi$.
We have shown above that $\gensg{\symm{\Pi'}, E_\Pi} \leq \Aut \Pi$.
Consequently, if $\pi = \desc{n+1} \circ \tau$ for some $\tau \in \symm{\Pi'}$, then $\tau^{-1} \in \symm{\Pi'} \leq \Aut \Pi$, so $\desc{n+1} = \pi \circ \tau^{-1} \in \Aut \Pi$.
Similarly, if $\pi$ is of the form $\desc{n+1} \circ \dja{n+1}{\ell} \circ \tau$, $\desc{n+1} \circ \ajd{n+1}{\ell} \circ \tau$, or $\desc{n+1} \circ \dja{n+1}{\ell} \circ \ajd{n+1}{m} \circ \tau$, then $(\dja{n+1}{\ell})^{-1}$ or $(\ajd{n+1}{\ell})^{-1}$ or both or neither, as appropriate, are in $\Aut \Pi$, and we can conclude in the same way that $\desc{n} \in \Aut \Pi$.

This establishes the inclusion $\Compn[n+1]{\Aut \Pi} \subseteq \gensg{\symm{\Pi'}, E_\Pi}$ if $\desc{n} \notin \Aut \Pi$ and $\Compn[n+1]{\Aut \Pi} \subseteq \gensg{\symm{\Pi'}, E_\Pi, \desc{n+1}}$ if $\desc{n} \in \Aut \Pi$.
\end{proof}

Theorem~\ref{thm:CompAutPi} describes $\Compn[n+1]{\Aut \Pi}$.
We are now going to find out how the level sequence of $\Aut \Pi$ continues.
Note that $\natcycle{n} \in \Aut \Pi$ if and only if $\interval{1}{n} \interwoven \Pi$ (see Lemma~\ref{lem:natcycle-1ninterwoven}).
In this case we also have $\desc{n} \in \Aut \Pi$, so $\dihed{n} \leq \Aut \Pi$.
Furthermore, if $n > 2$, then $\Aut \Pi \notin \{\alt{n}, \symm{n}\}$, so Theorem~\ref{thm:Comp-cycle} implies that $\Compn[n+1]{\Aut \Pi} = \dihed{n+1}$.
Moreover, $\Pi'$ is the partition of $\nset{n+1}$ into trivial blocks, so $\symm{\Pi'}$ is the trivial group, and Theorem~\ref{thm:CompAutPi} yields $\Compn[n+1]{\Aut \Pi} = \gensg{\natcycle{n+1}, \desc{n+1}} = \dihed{n+1}$, which is consistent with what we already know about groups containing the natural cycle (see Theorem~\ref{thm:Comp-cycle}).
Then we can apply Theorem~\ref{thm:Comp-cycle} repeatedly to conclude that $\Compn[n+i]{\Aut \Pi} = \dihed{n+i}$ for every $i \geq 1$.
In all other cases, we apply our results on intransitive groups.

\begin{lemma}
\label{lem:natcycle-1ninterwoven}
Let $\Pi$ be a nontrivial partition of $\nset{n}$ with no trivial blocks.
Then $\natcycle{n} \in \Aut \Pi$ if and only if $\interval{1}{n} \interwoven \Pi$.
\end{lemma}

\begin{proof}
Under our assumption that $\Pi$ has no trivial blocks, the condition $\interval{1}{n} \interwoven \Pi$ is equivalent to the following:
there exist numbers $k \geq 2$ and $\ell \geq 2$ such that $k \ell = n$ and the $\Pi$\hyp{}blocks are precisely the sets
\begin{equation}
\eqclass{i}{\Pi} = \{i + m k \mid 0 \leq m < \ell\} \qquad (1 \leq i \leq k).
\label{eq:1n-blocks}
\end{equation}
This condition clearly implies $\natcycle{n} \in \Aut \Pi$, since $\natcycle{n}(\eqclass{i}{\Pi}) = \eqclass{(i+1)}{\Pi}$ for each $i \in \nset{k-1}$ and $\natcycle{n}(\eqclass{k}{\Pi}) = \eqclass{1}{\Pi}$.

Assume then that $\natcycle{n} \in \Aut \Pi$.
This implies that for every $\Pi$\hyp{}block $B$, we have $\natcycle{n}(B) \in \Pi$.
Fix a $\Pi$\hyp{}block $B$.
Then all of the following are $\Pi$\hyp{}blocks:
\[
B, \natcycle{n}(B), \natcycle{n}^2(B), \natcycle{n}^3(B), \dots, \natcycle{n}^{n-1}(B).
\]
Since $\natcycle{n}$ is an $n$\hyp{}cycle, the above list contains all $\Pi$\hyp{}blocks, possibly with repetitions.
From this we can conclude that the $\Pi$\hyp{}blocks have equal size, say $\ell$.
Since $\Pi$ has no trivial blocks, $\ell \geq 2$.
Let $k$ be the number of $\Pi$\hyp{}blocks.
Since $\Pi$ is nontrivial, $k \geq 2$.
Then clearly $k \ell = n$ holds.
It is then easy to verify that the $\Pi$\hyp{}blocks must be as in \eqref{eq:1n-blocks}.
This means that $\interval{1}{n} \interwoven \Pi$.
\end{proof}

\begin{theorem}
\label{thm:CompAutPi-2}
Let $\Pi$ be a nontrivial partition of $\nset{n}$ with no trivial blocks.
Then for $i \geq 2$, we have
\[
\Compn[n+i]{\Aut \Pi} =
\begin{cases}
\symm{\Pi^{(i)}}, & \text{if $\desc{n} \notin \Aut \Pi$ and $\natcycle{n} \notin \Aut \Pi$,} \\
\gensg{\symm{\Pi^{(i)}}, \desc{n+1}}, & \text{if $\desc{n} \in \Aut \Pi$ and $\natcycle{n} \notin \Aut \Pi$,} \\
\dihed{n+i}, & \text{if $\natcycle{n} \in \Aut \Pi$.}
\end{cases}
\]
\end{theorem}

\begin{proof}
Consider first the case that $\natcycle{n} \in \Aut \Pi$.
Then $\interval{1}{n} \interwoven \Pi$ by Lemma~\ref{lem:natcycle-1ninterwoven}.
Then we clearly have $\desc{n} \in \Aut \Pi$ as well.
Consequently, $\dihed{n} \leq \Aut \Pi$.
Since $\Aut \Pi \notin \{\symm{n}, \alt{n}\}$, Theorem~\ref{thm:Comp-cycle} implies $\Compn[n+1]{\Aut \Pi} = \dihed{n+1}$.
An easy inductive argument using Theorem~\ref{thm:Comp-cycle} then shows that $\Compn[n+i]{\Aut \Pi} = \dihed{n+i}$ for all $i \geq 1$.

Consider then the case that $\natcycle{n} \notin \Aut \Pi$, i.e., $\interval{1}{n} \notinterwoven \Pi$.
If there is no $\ell$ with $1 < \ell < n$ such that $\interval{1}{\ell} \interwoven \Pi$ and no $m$ with $1 < m < n$ such that $\interval{m}{n} \interwoven \Pi$, then Theorem~\ref{thm:CompAutPi} gives $\Compn[n+1]{\Aut \Pi} = \symm{\Pi'}$ if $\desc{n} \notin \Aut \Pi$ and $\Compn[n+1]{\Aut \Pi} = \gensg{\symm{\Pi'}, \desc{n+1}}$ if $\desc{n} \in \Aut \Pi$.
The result then follows immediately from Theorem~\ref{thm:CompSPi-general}.

Assume then that $\interval{1}{\ell} \interwoven \Pi$ for some $\ell$ with $1 < \ell < n$ but there is no $m$ with $1 < m < n$ such that $\interval{m}{n} \interwoven \Pi$.
Since $\ell$ must be a composite number, we have $\ell \geq 4$.
In this case clearly $\Pi \neq \desc{n}(\Pi)$, so $\desc{n} \notin \Aut \Pi$, and
Theorem~\ref{thm:CompAutPi} gives
$\Compn[n+1]{\Aut \Pi} = \gensg{\symm{\Pi'}, \dja{n+1}{\ell}}$.
Note that the singletons $\{1\}, \{2\}, \dots, \{\ell\}, \{\ell + 1\}$ are $\Pi'$\hyp{}blocks, and that any permutation in $\gensg{\symm{\Pi'}, \dja{n+1}{\ell}}$ is of the form
$\asc{\ell} \oplus 1 \oplus \tau$ or $\desc{\ell} \oplus 1 \oplus \tau$
for some $\tau \in \symm{n-\ell}$.
In particular, every permutation $\sigma \in \gensg{\symm{\Pi'}, \dja{n+1}{\ell}}$ satisfies $\sigma(\ell + 1) = \ell + 1$ and $\sigma(\nset{\ell}) = \nset{\ell}$.

By the monotonicity of $\Compn[n+2]{}$ and Proposition~\ref{prop:CompSPi}, we have
\[
\symm{\Pi^{(2)}} = \Compn[n+2]{\symm{\Pi'}} \subseteq \Compn[n+2]{\gensg{\symm{\Pi'}, \dja{n+1}{\ell}}}.
\]
Suppose, to the contrary, that the above inclusion is proper, i.e., there exists $\pi \in (\Compn[n+2]{\gensg{\symm{\Pi'}, \dja{n+1}{\ell}}}) \setminus \symm{\Pi^{(2)}}$.
Since $\Compn[n+2]{\symm{\Pi'}} = \symm{\Pi^{(2)}}$, the permutation $\pi$ has an $(n+1)$\hyp{}pattern in $\gensg{\symm{\Pi'}, \dja{n+1}{\ell}} \setminus \symm{\Pi'}$; say $\subpermsimple{\pi}{i} \in \gensg{\symm{\Pi'}, \dja{n+1}{\ell}} \setminus \symm{\Pi'}$.

If $i \leq \ell$, then $\pi(\ell + 1) \in \{1, 2\}$; hence $\subpermsimple{\pi}{\ell + 2}(\ell + 1) \leq 2 < \ell + 1$, a contradiction.

If $i = \ell + 1$, then we distinguish between three cases according to the value of $\pi$ at $\ell + 1$.
If $\pi(\ell + 1) < \ell + 1$, then we must have $\pi(\ell + 2) = \ell + 2$; hence $\subpermsimple{\pi}{\ell + 2}(\ell + 1) < \ell + 1$, a contradiction.
If $\pi(\ell + 1) = \ell + 1$, then we must have $\pi(\ell + 2) = \ell + 2$; hence $\subpermsimple{\pi}{1}(\ell + 1) = \ell < \ell + 1$, a contradiction.
If $\pi(\ell + 1) > \ell + 1$, then we must have $\pi(\ell + 2) = \ell + 1$; hence $\subpermsimple{\pi}{1}(\ell + 1) = \pi(\ell + 1) - 1 > \ell$, a contradiction.

Finally, if $i > \ell + 1$, then we distinguish between two cases according to the value of $\pi$ at $i$.
If $\pi(i) \geq \ell + 2$, then $\subpermsimple{\pi}{1} = \desc{\ell - 1} \oplus 1 \oplus \tau'$ for some $\tau' \in \symm{n - \ell - 1}$, a contradiction.
If $\pi(i) \leq \ell + 1$, then $\pi(\ell + 1) = \ell + 2$; hence $\subpermsimple{\pi}{1}(\ell) = \ell + 1$, a contradiction.

This completes the proof that $\Compn[n+2]{\Aut \Pi} = \symm{\Pi^{(2)}}$ when $\interval{1}{\ell} \interwoven \Pi$ for some $\ell$ with $1 < \ell < n$ but there is no $m$ with $1 < m < n$ such that $\interval{m}{n} \interwoven \Pi$.
An entirely analogous argument shows that $\Compn[n+2]{\Aut \Pi} = \symm{\Pi^{(2)}}$ also in the case where $\interval{m}{n} \interwoven \Pi$ for some $m$ with $1 < m < n$ but there is no $\ell$ with $1 < \ell < n$ such that $\interval{1}{\ell} \interwoven \Pi$.
A similar argument, with a few obvious modifications to take care of the descending permutation, also works in the case where $\interval{1}{\ell} \interwoven \Pi$ and $\interval{m}{n} \interwoven \Pi$, and we obtain
\[
\Compn[n+2]{\Aut \Pi} = 
\begin{cases}
\symm{\Pi^{(2)}}, & \text{if $\desc{n} \notin \Aut \Pi$,} \\
\gensg{\Pi^{(2)}, \desc{n+2}}, & \text{if $\desc{n} \in \Aut \Pi$.}
\end{cases}
\]
The result then follows immediately from Theorem~\ref{thm:CompSPi-general}.
\end{proof}

Application of Theorem~\ref{thm:general-intransitive} gives us upper bounds for $\Compn[n+i]{G}$ for $i \geq 2$, when $G$ is an arbitrary imprimitive group.
We can also determine which steady sequence is eventually reached by the level sequence of $G$ and how fast this happens.

\begin{theorem}
\label{thm:general-imprimitive}
Let $G \leq \symm{n}$ be an imprimitive group with $\natcycle{n} \notin G$, and let $\Pi$ be the finest partition of $\nset{n}$ such that $G \leq \Aut \Pi$.
Let $a$ and $b$ be the largest numbers $\alpha$ and $\beta$, respectively, such that $\symm{n}^{\alpha,\beta} \leq G$.
Then for all $\ell \geq \max(\maxblocksize_{a,b}(\Pi), 2)$, it holds that $\Compn[n + \ell]{G} = \symm{n + \ell}^{a,b}$ or $\Compn[n + \ell]{G} = \gensg{\symm{n + \ell}^{a,b}, \desc{n + \ell}}$.
\end{theorem}

\begin{proof}
Follows immediately from Theorems \ref{thm:CompSPi-general}, \ref{thm:general-intransitive}, and \ref{thm:CompAutPi-2}.
\end{proof}


\section{Primitive groups}
\label{sec:primitive}

In this section, we investigate level sequences of primitive groups.
Recall that a transitive subgroup of $\symm{n}$ is \emph{primitive} if it preserves no nontrivial partition of $\nset{n}$.

The following classical results are due to Jordan; for proofs, see, e.g., \cite[Theorems~8.17 and 8.23]{Isaacs}.

\begin{theorem}[Jordan]
\label{thm:Jordan-prim-trans}
Let $G \leq \symm{n}$. Assume that $G$ is primitive and contains a transposition. Then $G$ is the full symmetric group $\symm{n}$.
\end{theorem}

\begin{theorem}[Jordan]
\label{thm:Jordan-prim-cycle}
Let $G \leq \symm{n}$. Assume that $G$ is primitive and contains a cycle of length $p$, where $p$ is prime and $p \leq n - 3$. Then $G$ is either the alternating group $\alt{n}$ or the full symmetric group $\symm{n}$.
\end{theorem}

Jones showed that the primality condition can be omitted from Theorem \ref{thm:Jordan-prim-cycle}.
His theorem (Theorem~\ref{thm:Jones1.2} below) also brings together the characterization of primitive groups with a cycle fixing one point that is due to M\"uller~\cite[Theorem~6.2]{Muller} and the description of primitive groups with a fixed\hyp{}point\hyp{}free cycle from an earlier paper of Jones~\cite{Jones2002}, which completed the work of Feit~\cite{Feit}.

The following well\hyp{}known groups are referred to in what follows:
the cyclic group of order $p$, denoted by $C_p$,
the affine general linear group of degree $d$ over the finite field $\IF_q$ of order $q$, denoted by $\AGL{d}{q}$,
the affine semilinear group $\AGammaL{d}{q}$,
the projective special linear group $\PSL{d}{q}$,
the projective general linear group $\PGL{d}{q}$,
the projective semilinear group $\PGammaL{d}{q}$,
Mathieu groups $M_{11}$, $M_{12}$, $M_{24}$.

\begin{theorem}[{Jones~\cite[Theorem~1.2]{Jones2014}}]
\label{thm:Jones1.2}
Let $G \leq \symm{n}$. Assume that $G$ is primitive and $\alt{n} \nleq G$.
Assume that $G$ contains a cycle fixing $k$ points, where $0 \leq k \leq n-2$.
Then one of the following holds:
\begin{enumerate}[label=\upshape{(\arabic*)}]
\item\label{Jones:k=0} $k = 0$ and either
\begin{enumerate}[label=\upshape{(\alph*)}]
\item $C_p \leq G \leq \AGL{1}{p}$ with $n = p$ prime; or
\item $\PGL{d}{q} \leq G \leq \PGammaL{d}{q}$ with $n = (q^d - 1) / (q - 1)$ and $d \geq 2$ for some prime power $q$; or
\item $G = \PSL{2}{11}$, $M_{11}$ or $M_{23}$ with $n = 11$, $11$ or $23$, respectively.
\end{enumerate}
\item\label{Jones:k=1} $k = 1$ and either
\begin{enumerate}[label=\upshape{(\alph*)}]
\item $\AGL{d}{q} \leq G \leq \AGammaL{d}{q}$ with $n = q^d$ and $d \geq 1$ for some prime power $q$; or
\item $G = \PSL{2}{p}$ or $\PGL{2}{p}$ with $n = p + 1$ for some prime $p \geq 5$; or
\item $G = M_{11}$, $M_{12}$ or $M_{24}$ with $n = 12$, $12$ or $24$, respectively.
\end{enumerate}
\item\label{Jones:k=2} $k = 2$ and $\PGL{2}{q} \leq G \leq \PGammaL{2}{q}$ with $n = q + 1$ for some prime power $q$.
\end{enumerate}
\end{theorem}

\begin{theorem}[{Jones~\cite[Corollary~1.3]{Jones2014}}]
\label{thm:Jones1.3}
Let $G \leq \symm{n}$. Assume that $G$ is primitive and contains a cycle with $k$ fixed points.
Then $G \geq \alt{n}$ if $k \geq 3$, or if
$k = 0$, $1$ or $2$ and $n$ avoids the values listed in the respective parts of Theorem~\ref{thm:Jones1.2}.
\end{theorem}

\begin{proposition}
\label{prop:Comp-n+2-primitive}
Let $G \leq \symm{n}$.
Assume that $G$ is primitive and does not contain the natural cycle $\natcycle{n}$.
Then $\Compn[n+2]{G} \leq \gensg{\desc{n+2}}$.
\end{proposition}

\begin{proof}
Write $H := \Compn[n+2]{G}$.
It follows from Lemma~\ref{lem:cyclic-dihedral} that $H$ does not contain any permutation in $\dihed{n+2} \setminus \{\asc{n+2}, \desc{n+2}\}$.
Suppose, to the contrary, that $H$ contains a permutation $\pi$ that is neither $\asc{n+2}$ not $\desc{n+2}$.
By Fact~\ref{fact:dihedral}, there exists $t \in \nset{n+1}$ such that $\pi(t)$ and $\pi(t+1)$ are not consecutive modulo $n+2$.
By Lemma~\ref{lem:AtkBea-L11}, $\gPatl[n]{H}$ contains a transposition.
Since $\gPatl[n]{H} \subseteq G$, it follows from Jordan's Theorem~\ref{thm:Jordan-prim-trans} that $G = \symm{n}$.
This contradicts the assumption that $\natcycle{n} \notin G$.
\end{proof}

For $\sigma \in \symm{m}$, we write
\[
\jumps{\sigma} := \{(a,b) \in \nset{m} \times \nset{m} \mid
\text{$a \leq b-1$ and $\sigma$ has a jump between $a$ and $b$}\},
\]
and for $S \subseteq \symm{m}$, we write
\[
\jumps{S} := \bigcup_{\sigma \in S} \jumps{\sigma}.
\]

\begin{lemma}
\label{lem:jumpsets}
Assume that $G \leq \symm{n}$ is a primitive group that does not contain the natural cycle $\natcycle{n}$ and $\alt{n} \nleq G$.
Then $\jumps{\Compn[n+1]{G}}$ equals one of the following sets:
\begin{itemize}
\item $\emptyset$, $\{(1,n)\}$, $\{(2,n+1)\}$, $\{(1,n-1)\}$, $\{(3,n+1)\}$,
\item $\{(1,5), (2,7)\}$, $\{(1,6), (3,7)\}$, $\{(1,5), (3,7)\}$ \textup{(}the last three are possible only if $n = 6$\textup{)}.
\end{itemize}
\end{lemma}

\begin{proof}
The claim is vacuously true whenever $n \leq 4$, because the only primitive groups of degree at most $4$ are the alternating and symmetric groups.
Therefore, we assume from now on that $n \geq 5$.

Let $G \leq \symm{n}$ be a primitive group, and write $H := \Compn[n+1]{G}$.
By Lemma~\ref{lem:AtkBea-L11}, $G$ contains the cycle $(a \; a+1 \; \cdots \; b-1)$ for every $(a,b) \in \jumps{H}$.
Making use of this fact, we are going to show that if $\jumps{H}$ is not one of the sets listed in the statement of the lemma, then a contradiction arises (usually by means of denying our premise that $\natcycle{n} \notin G$ and $\alt{n} \nleq G$).

If $(1,n+1) \in \jumps{H}$, then $\natcycle{n} \in G$.
If $\jumps{H}$ contains a pair $(a,b)$ such that $b - a = 2$, then $G$ contains the transposition $(a \; a+1)$ and consequently $G = \symm{n}$ by Theorem~\ref{thm:Jordan-prim-trans}.
If $\jumps{H}$ contains a pair $(a,b)$ such that $b - a \leq n - 3$, then $G$ contains a cycle of length at most $n-3$ and consequently $\alt{n} \leq G$ by Theorem~\ref{thm:Jones1.3}.
By our premises, none of the above is possible.
Therefore, we may assume that every pair $(a,b) \in \jumps{H}$ satisfies $n - 2 \leq b - a \leq n - 1$ and $b - a \geq 3$, that is,
\begin{equation}
\jumps{H} \subseteq \{(1,n-1), (1,n), (2,n), (2,n+1), (3,n+1)\}.
\label{eq:J}
\end{equation}

Consider first the case where $\card{\jumps{H}} = 1$, say $\jumps{H} = \{(a,b)\}$.
Then each permutation in $H$ either has no jumps or has a jump between $a$ and $b$.
By Lemma~\ref{lem:one-jump}, we must have $a = 1$ or $b = n+1$.
A contradiction arises if $\jumps{H} = \{(2,n)\}$.

Consider then the case where $\card{\jumps{H}} \geq 2$.
If $\jumps{H}$ contains pairs $(a,b)$ and $(a,b+1)$ for some $a$ and $b$, then $G$ contains the cycles $(a \; a+1 \; \cdots \; b-1)$ and $(a \; a+1 \; \cdots \; b)$ and hence also
\[
(a \; a+1 \; \cdots \; b)^{-1} \circ (a \; a+1 \; \cdots \; b-1) = (b-1 \; b) \in G. 
\]
Consequently, $G = \symm{n}$ by Theorem~\ref{thm:Jordan-prim-trans}, and we have reached a contradiction with the premise that $\alt{n} \nleq G$.
In a similar way, we can conclude that if $\jumps{H}$ contains pairs $(a,b)$ and $(a+1,b)$ for some $a$ and $b$, then $(a \; b-1) \in G$ and hence $G = \symm{n}$, again a contradiction.

If $\jumps{H}$ contains pairs $(a,b)$ and $(a+1,b+1)$ for some $a$ and $b$, then $G$ contains the cycles $(a \; a+1 \; \cdots \; b-1)$ and $(a+1 \; a+2 \; \cdots \; b)$ and hence also
\[
(a \; a+1 \; \cdots \; b-1)^{-1} \circ (a+1 \; a+2 \; \cdots \; b) = (a \; b-1 \; b) \in G.
\]
If $n \geq 6$, then $\alt{n} \leq G$ by Theorem~\ref{thm:Jordan-prim-cycle}, and we have reached again a contradiction.
If $n = 5$, then one of the following cases holds:
\begin{itemize}
\item Case 1: $(1,4), (2,5) \in \jumps{H}$ and consequently $(1 \; 2 \; 3), (2 \; 3 \; 4) \in G$. Since these permutations generate $\alt{\interval{1}{4}}$ and since the only transitive overgroups of $\alt{\interval{1}{4}}$ in $\symm{5}$ are $\alt{5}$ and $\symm{5}$, we conclude that $\alt{5} \leq G$, a contradiction.
\item Case 2: $(2,5), (3,6) \in \jumps{H}$ and consequently $(2 \; 3 \; 4), (3 \; 4 \; 5) \in G$.
By considering conjugates with respect to the descending permutation, the previous case implies that $\alt{5} \leq G$, a contradiction.
\item Case 3: $(1,5), (2,6) \in \jumps{H}$ and consequently $(1 \; 2 \; 3 \; 4), (2 \; 3 \; 4 \; 5) \in G$. These permutations generate $\symm{5}$, so we have $G = \symm{5}$, a contradiction.
\end{itemize}

If $\jumps{H}$ contains pairs $(1, n-1)$ and $(2, n+1)$, then $G$ contains the cycles $(1 \; 2 \; \cdots \; n-2)$ and $(2 \; 3 \; \cdots \; n)$ and hence also
\[
(1 \; 2 \; \cdots \; n-2)^{-1} \circ (2 \; 3 \; \cdots \; n) = (1 \; n-2 \; n-1 \; n) \in G.
\]
If $n \geq 7$, then $\alt{n} \leq G$ by Theorem~\ref{thm:Jones1.3}, a contradiction.
If $n = 5$, then $G = \symm{5}$, because the cycles $(1 \; 2 \; 3)$ and $(2 \; 3 \; 4 \; 5)$ generate the full symmetric group $\symm{5}$, again a contradiction.
If $n = 6$, then $G$ contains the cycles $(1 \; 2 \; 3 \; 4)$ and $(2 \; 3 \; 4 \; 5 \; 6)$; this is fine.

By considering conjugates with respect to the descending permutation, a contradiction arises in a similar way if we assume that $\jumps{H}$ contains pairs $(1,n)$ and $(3,n+1)$, unless $n = 6$.

If $\jumps{H}$ contains pairs $(1,n-1)$ and $(3,n+1)$, then $G$ contains the cycles $(1 \; 2 \; \cdots \; n-2)$ and $(3 \; 4 \; \cdots \; n)$ and hence also
\[
(1 \; 2 \; \cdots \; n-2)^{-1} \circ (3 \; 4 \; \cdots \; n) = (1 \; n-2 \; n-1 \; n \; 2) \in G.
\]
If $n \geq 8$, then $\alt{n} \leq G$ by Theorem~\ref{thm:Jones1.3}.
If $n = 7$, then we also have $\alt{n} \leq G$ by Theorem~\ref{thm:Jones1.3}, because the above\hyp{}mentioned cycles have two fixed points, $n = 6 + 1$, and $6$ is not a prime power.
If $n = 5$, then $\alt{5} \leq G$, because the cycles $(1 \; 2 \; 3)$ and $(3 \; 4 \; 5)$ generate the alternating group $\alt{5}$.
In all the above cases, we reached a contradiction.
If $n = 6$, then $G$ contains the cycles $(1 \; 2 \; 3 \; 4)$ and $(3 \; 4 \; 5 \; 6)$; this is fine.

We have exhausted all possible two\hyp{}element subsets of the set on the right side of \eqref{eq:J}.
It is easy to see that any subset with three or more elements contains pairs of the form $(a,b)$ and $(a,b+1)$; or $(a,b)$ and $(a+1,b)$; or $(a,b)$ and $(a+1,b+1)$ for some $a$ and $b$, and we have seen above that this leads to a contradiction.
We conclude that $\card{\jumps{H}} \geq 2$ only if $n = 6$ and $\jumps{H}$ is one of the following two\hyp{}element sets: $\{(1,5), (2,7)\}$, $\{(1,6), (3,7)\}$, $\{(1,5), (3,7)\}$.

Finally, we note that it is also possible that $\jumps{H} = \emptyset$.
This concludes the proof of the lemma.
\end{proof}

\begin{fact}
\label{fact:primitive6}
The primitive permutation groups of degree $6$ are well known, and they are listed, up to isomorphism, in Table~\ref{table:primitive6}. The list also provides a generating set for each, and it indicates the point stabilizer of each group. (Source:~\cite[Section~VI.1.4]{CRC-Combinatorial-Designs}.)
\end{fact}

\begin{table}
\begin{center}
\begin{tabular}{lll}
\toprule
Group & Generators & Point stabilizer \\
\midrule
$\PSL{2}{5} \cong \alt{5}$ & $(1 \; 2 \; 3 \; 4 \; 5), (1 \; 3 \; 4)(2 \; 5 \; 6)$ & $\dihed{5}$ \\
$\PGL{2}{5} \cong \symm{5}$ & $(1 \; 2 \; 3 \; 4 \; 5), (1 \; 5 \; 2 \; 4 \; 3 \; 6)$ & $\AGL{1}{5}$ \\
$\alt{6}$ & $(1 \; 2 \; 3 \; 4 \; 5), (1 \; 2 \; 3)(4 \; 5 \; 6)$ & $\alt{5}$ \\
$\symm{6}$ & $(1 \; 2 \; 3 \; 4 \; 5), (1 \; 2 \; 3 \; 4 \; 5 \; 6)$ & $\symm{5}$ \\
\bottomrule
\end{tabular}
\end{center}

\medskip
\caption{Primitive permutation groups of degree $6$.}
\label{table:primitive6}
\end{table}

\begin{fact}
\label{fact:PGL25-S5-S6}
The following equalities involving subgroups of $\symm{6}$ hold:
\begin{gather*}
\gensg{(1 \; 2 \; 3 \; 4), (2 \; 3 \; 4 \; 5 \; 6)} = \gensg{456123, 156234, 145623} = \gensg{432165, 143265, 154326}, \\
\gensg{(1 \; 2 \; 3 \; 4 \; 5), (3 \; 4 \; 5 \; 6)} = \gensg{154326, 215436, 216543} = \gensg{451236, 345126, 456123}, \\
\gensg{(1 \; 2 \; 3 \; 4), (3 \; 4 \; 5 \; 6)} = \gensg{143265, 214365, 215436} = \gensg{634512, 563412, 562341}.
\end{gather*}
Each one of the three groups mentioned above is isomorphic to $\PGL{2}{5} \cong \symm{5}$.
Hence, each one of them is a maximal subgroup of $\symm{6}$.
Each one is a primitive group that does not contain the natural cycle $\natcycle{6}$ and does not include the alternating group $\alt{6}$.
\end{fact}

\begin{lemma}
\label{lem:Comp-prim-n=6}
Let $n = 6$, and assume that $G \leq \symm{n}$ is a primitive group that does not contain the natural cycle $\natcycle{n}$ and $\alt{n} \nleq G$.
Then the following statements hold.
\begin{enumerate}[label=\upshape{(\roman*)}]
\item\label{lem:Comp-prim-n=6:i}
The following conditions are equivalent:
\begin{enumerate}[label=\upshape{(\arabic*)}]
\item\label{cond:jumpsG'} $\jumps{\Compn[n+1]{G}} = \{(1,5), (2,7)\}$,
\item\label{cond:G'description} $\Compn[n+1]{G} = \{1234567, 1276543, 1567234, 1543276\}$,
\item\label{cond:Gdescription} $G = \gensg{(1 \; 2 \; 3 \; 4), (2 \; 3 \; 4 \; 5 \; 6)}$.
\end{enumerate}

\item\label{lem:Comp-prim-n=6:ii}
The following conditions are equivalent:
\begin{enumerate}[label=\upshape{(\arabic*)}]
\item $\jumps{\Compn[n+1]{G}} = \{(1,6), (3,7)\}$,
\item $\Compn[n+1]{G} = \{1234567, 2165437, 4561237, 5432167\}$.
\item $G = \gensg{(1 \; 2 \; 3 \; 4 \; 5), (3 \; 4 \; 5 \; 6)}$.
\end{enumerate}

\item\label{lem:Comp-prim-n=6:iii}
The following conditions are equivalent:
\begin{enumerate}[label=\upshape{(\arabic*)}]
\item $\jumps{\Compn[n+1]{G}} = \{(1,5), (3,7)\}$,
\item $\Compn[n+1]{G} = \{1234567, 2154376, 6734512, 7654321\}$.
\item $G = \gensg{(1 \; 2 \; 3 \; 4), (3 \; 4 \; 5 \; 6)}$.
\end{enumerate}
\end{enumerate}
\end{lemma}

\begin{proof}
\ref{lem:Comp-prim-n=6:i}
We prove the implications
(1)$\implies$(2)$\implies$(1)$\implies$(3)$\implies$(1).
Write $H := \Compn[n+1]{G}$.
To prove (1)$\implies$(2), assume that $\jumps{H} = \{(1,5), (2,7)\}$.
Then $\gensg{(1 \; 2 \; 3 \; 4), (2 \; 3 \; 4 \; 5 \; 6)} \leq G$ by Lemma~\ref{lem:AtkBea-L11}.
Then $H$ is a subset of
\begin{multline*}
K := \{\sigma \in \symm{n+1} \mid \jumps{\sigma} \subseteq \{(1,5), (2,7)\}\} \\
\begin{array}{rlllll}
= \{1234567, & 4321567, & 1276543, & 1543276, & 1567234, & 3451276, \\
    7654321, & 7651234, & 3456721, & 6723451, & 4327651, & 6721543 \}.
\end{array}
\end{multline*}
However, since $H$ is a group and hence closed under composition, not all elements of $K$ can be members of $H$.
Namely, for each
\[
\sigma \in \{3451276, 3456721, 4327651, 6721543, 6723451, 7651234\},
\]
we have $\sigma \circ \sigma \notin K$ or $\sigma \circ \sigma \circ \sigma \notin K$, so $\sigma$ cannot be in $H$.
Furthermore, for $\sigma = 4321567$ and for each $\tau \in \{1276543, 1567234, 1543276, 7654321\}$, we have $\sigma \circ \tau \notin K$, so we cannot have both $\sigma$ and $\tau$ in $H$.
From this it follows that $4321567 \notin H$, because otherwise we would have $H = \{1234567, 4321567\}$ and consequently $(2,7) \notin \jumps{H}$, contradicting our assumptions.
Similarly, for $\sigma = 7654321$ and for each $\tau \in \{1276543, 1543276, 1567234\}$, we have $\sigma \circ \tau \notin K$, so we cannot have both $\sigma$ and $\tau$ in $H$.
From this it follows that $7654321 \notin H$, because otherwise we would have $H \subseteq \{1234567, 7654321\}$ and hence $\jumps{H} = \emptyset$.
Therefore, $H \subseteq \{1234567, 1276543, 1543276, 1567234\}$.
Furthermore, either $1567234$ or $1563276$ must be in $H$, because otherwise we would have $(1,5) \notin \jumps{H}$.

Observe that
\begin{align*}
\Patl[6]{1234567} &= \{123456\}, \\
\Patl[6]{1276543} &= \{165432, 126543\}, \\
\Patl[6]{1543276} &= \{432165, 143265, 154326\}, \\
\Patl[6]{1567234} &= \{456123, 156234, 145623\}.
\end{align*}
By Fact~\ref{fact:PGL25-S5-S6}, we have
\[
\gensg{456123, 156234, 145623} = \gensg{432165, 143265, 154326} = \gensg{(1 \; 2 \; 3 \; 4), (2 \; 3 \; 4 \; 5 \; 6)},
\]
and it also holds that
\[
\{1234567, 165432, 126543\} \subseteq \gensg{(1 \; 2 \; 3 \; 4), (2 \; 3 \; 4 \; 5 \; 6)}.
\]
Consequently,
\[
\{1234567, 1276543, 1543276, 1567234\}
\subseteq \Compn[n+1]{\gensg{(1 \; 2 \; 3 \; 4), (2 \; 3 \; 4 \; 5 \; 6)}} \subseteq H,
\]
and we conclude that $H = \{1234567, 1276543, 1543276, 1567234\}$.

Implication (2)$\implies$(1) is obvious.

For implication (1)$\implies$(3), assume that $\jumps{H} = \{(1,5), (2,7)\}$.
As observed above, this implies that $\gensg{(1 \; 2 \; 3 \; 4), (2 \; 3 \; 4 \; 5 \; 6)} \leq G$ by Lemma~\ref{lem:AtkBea-L11}.
By Fact~\ref{fact:PGL25-S5-S6}, $\gensg{(1 \; 2 \; 3 \; 4), (2 \; 3 \; 4 \; 5 \; 6)}$ is primitive and it does not contain $\natcycle{6}$ nor $\alt{6}$, but the only proper overgroup of $\gensg{(1 \; 2 \; 3 \; 4), (2 \; 3 \; 4 \; 5 \; 6)}$, namely $\symm{6}$, fails to have these properties.
Therefore, the subgroup inclusion must be satisfied as an equality, i.e., $G = \gensg{(1 \; 2 \; 3 \; 4), (2 \; 3 \; 4 \; 5 \; 6)}$.

Finally, for implication (3)$\implies$(1), assume that $G = \gensg{(1 \; 2 \; 3 \; 4), (2 \; 3 \; 4 \; 5 \; 6)}$.
As observed above, $\Patl[6]{1543276} \subseteq \gensg{(1 \; 2 \; 3 \; 4), (2 \; 3 \; 4 \; 5 \; 6)}$, so $1543276 \in H$.
Consequently, $\{(1,5), (2,7)\} \subseteq \jumps{H}$.
By Lemma~\ref{lem:jumpsets}, $\jumps{H}$ cannot contain more than two pairs, so we conclude that $\jumps{H} = \{(1,5), (2,7)\}$.

\ref{lem:Comp-prim-n=6:ii}
The claim follows from part \ref{lem:Comp-prim-n=6:i} by considering conjugates with respect to the descending permutation and noting the following two facts:
(1) For every $\sigma \in \symm{n+1}$, we have $(a,b) \in \jumps{\sigma}$ if and only if $(\desc{n+1}(b), \desc{n+1}(a)) \in \jumps{\desc{n+1} \circ \sigma \circ \desc{n+1}}$.
(2) Replacing generators of a group with their inverses does not change the group generated.

\ref{lem:Comp-prim-n=6:iii}
The proof is similar to part \ref{lem:Comp-prim-n=6:i}.
\end{proof}

For an interval $\interval{a}{b} \subseteq \nset{n}$, we denote by $\dihed{\interval{a}{b}}$ the subgroup of $\symm{n}$ comprising all permutations of the form $\asc{a-1} \oplus \pi \oplus \asc{n-b}$, where $\pi \in \dihed{b-a+1}$.

\begin{lemma}
\label{lem:Comp-prim-general}
Let $n \geq 5$, and assume that $G \leq \symm{n}$ is a primitive group that does not contain the natural cycle $\natcycle{n}$ and $\alt{n} \nleq G$.
Then the following statements hold.
\begin{enumerate}[label=\upshape{(\roman*)}]
\item\label{lem:Comp-prim-general:i}
The following conditions are equivalent:
\begin{enumerate}[label=\upshape{(\arabic*)}]
\item $\jumps{\Compn[n+1]{G}} = \{(1,n)\}$,
\item $\Compn[n+1]{G} = \gensg{\desc{n-1} \oplus \asc{2}}$,
\item $\dihed{\interval{1}{n-1}} \leq G$ and if $n = 6$ then $G = \gensg{(1 \; 2 \; 3 \; 4 \; 5), (1 \; 3 \; 4)(2 \; 5 \; 6)}$.
\end{enumerate}

\item\label{lem:Comp-prim-general:ii}
The following conditions are equivalent:
\begin{enumerate}[label=\upshape{(\arabic*)}]
\item $\jumps{\Compn[n+1]{G}} = \{(1,n-1)\}$,
\item $\Compn[n+1]{G} = \gensg{\desc{n-2} \oplus \asc{3}}$,
\item $\dihed{\interval{1}{n-2}} \leq G$.
\end{enumerate}

\item\label{lem:Comp-prim-general:iii}
The following conditions are equivalent:
\begin{enumerate}[label=\upshape{(\arabic*)}]
\item $\jumps{\Compn[n+1]{G}} = \{(2, n+1)\}$,
\item $\Compn[n+1]{G} = \gensg{\asc{2} \oplus \desc{n-1}}$,
\item $\dihed{\interval{2}{n}} \leq G$ and if $n = 6$ then $G = \gensg{(2 \; 3 \; 4 \; 5 \; 6), (1 \; 2 \; 5)(3 \; 4 \; 6)}$.
\end{enumerate}

\item\label{lem:Comp-prim-general:iv}
The following conditions are equivalent:
\begin{enumerate}[label=\upshape{(\arabic*)}]
\item $\jumps{\Compn[n+1]{G}} = \{(3, n+1)\}$,
\item $\Compn[n+1]{G} = \gensg{\asc{3} \oplus \desc{n-2}}$,
\item $\dihed{\interval{3}{n}} \leq G$.
\end{enumerate}
\end{enumerate}
\end{lemma}

\begin{proof}
Write $H := \Compn[n+1]{G}$.

\ref{lem:Comp-prim-general:i}
(1)$\implies$(2):
If $\jumps{H} = \{(1, n)\}$, then
\[
H \subseteq \{\sigma \in \symm{n+1} \mid \jumps{\sigma} \subseteq \{(1,n)\}\} = \{\asc{n+1}, \desc{n+1}, \desc{n-1} \oplus \asc{2}, \desc{2} \ominus \asc{n-1}\}.
\]
However, $\desc{2} \ominus \asc{n-1} \notin H$, because $(\desc{2} \ominus \asc{n-1}) \circ (\desc{2} \ominus \asc{n-1}) = (\desc{2} \oplus \desc{2}) \ominus \asc{n-3}$ and this permutation has a jump between $n-2$ and $n+1$.
Moreover, $\desc{n+1}$ and $\desc{n-1} \oplus \asc{2}$ cannot both be in $H$, because
$\desc{n+1} \circ (\desc{n-1} \oplus \asc{2}) = \asc{n-1} \ominus \desc{2}$, and this permutation has a jump between $2$ and $n+1$.
At the same time, we must have $\desc{n-1} \oplus \asc{2} \in H$, because otherwise we would have $\jumps{H} = \emptyset$.
Therefore, $H = \{\asc{n+1}, \desc{n-1} \oplus \asc{2}\} = \gensg{\desc{n-1} \oplus \asc{2}}$.

(2)$\implies$(3):
If $H = \{\asc{n+1}, \desc{n-1} \oplus \asc{2}\}$, then $\dihed{\interval{1}{n-1}} \leq G$, because it holds that $\Patl[n]{\desc{n-1} \oplus \asc{2}} = \{\desc{n-2} \oplus \asc{2}, \desc{n-1} \oplus \asc{1}\}$ and this set generates $\dihed{\interval{1}{n-1}}$.

We can conclude from Facts~\ref{fact:primitive6} and~\ref{fact:PGL25-S5-S6} that in the case where $n = 6$, the primitive groups that include $\dihed{\interval{1}{n-1}}$ are $\gensg{(1 \; 2 \; 3 \; 4 \; 5), (1 \; 3 \; 4)(2 \; 5 \; 6)} \cong \PSL{2}{5}$, $\gensg{(1 \; 2 \; 3 \; 4 \; 5), (3 \; 4 \; 5 \; 6)} \cong \PGL{2}{5}$, and $\symm{6}$.
But the last two options are not possible, because $\alt{6} \leq \symm{6}$ and we have shown in Lemma~\ref{lem:Comp-prim-n=6} that
\[
\Compn[n+1]{\gensg{(1 \; 2 \; 3 \; 4 \; 5), (3 \; 4 \; 5 \; 6)}} = \{1234567, 2165437, 4561237, 5432167\}.
\]
Therefore, $G = \gensg{(1 \; 2 \; 3 \; 4 \; 5), (1 \; 3 \; 4)(2 \; 5 \; 6)}$.

(3)$\implies$(1):
If $\dihed{\interval{1}{n-1}} \leq G$, then $\desc{n-1} \oplus \asc{2} \in H$ because $\Patl[n]{\desc{n-1} \oplus \asc{2}} = \{\desc{n-2} \oplus \asc{2}, \desc{n-1} \oplus \asc{1}\} \subseteq \dihed{\interval{1}{n-1}}$.
Consequently, $(1,n) \in \jumps{H}$.
It follows from Lemma~\ref{lem:jumpsets} that $\jumps{H} = \{(1,n)\}$, unless $n = 6$.
If $n = 6$, then $\jumps{H}$ equals either $\{(1,6)\}$ or $\{(1,6), (3,7)\}$.
Lemma~\ref{lem:Comp-prim-n=6} asserts that if $n = 6$ and $\jumps{H} = \{(1,6), (3,7)\}$, then $G = \gensg{(1 \; 2 \; 3 \; 4 \; 5), (3 \; 4 \; 5 \; 6)}$.
But we are assuming that $G = \gensg{(1 \; 2 \; 3 \; 4 \; 5), (1 \; 3 \; 4)(2 \; 5 \; 6)} \neq \gensg{(1 \; 2 \; 3 \; 4 \; 5), (3 \; 4 \; 5 \; 6)}$, so we have $\jumps{H} = \{(1,6)\}$ also in this case.

\ref{lem:Comp-prim-general:ii}
(1)$\implies$(2):
If $J = \{(1, n-1)\}$, then
\[
H \subseteq \{\sigma \in \symm{n+1} \mid \jumps{\sigma} \subseteq \{(1,n-1)\}\} = \{\asc{n+1}, \desc{n+1}, \desc{n-2} \oplus \asc{3}, \desc{3} \ominus \asc{n-2}\}.
\]
However, $\desc{3} \ominus \asc{n-2} \notin H$, because
\[
(\desc{3} \ominus \asc{n-2}) \circ (\desc{3} \ominus \asc{n-2}) =
\begin{cases}
(\desc{3} \oplus \desc{3}) \ominus \asc{n-5}, & \text{if $n \geq 6$,} \\
\desc{3} \oplus \desc{3}, & \text{if $n = 6$,} \\
21354, & \text{if $n = 5$,}
\end{cases}
\]
and in each case we have a permutation that has a jump that is not allowed.
Moreover, $\desc{n+1}$ and $\desc{n-2} \oplus \asc{3}$ cannot both be in $H$, because $\desc{n+1} \circ (\desc{n-2} \oplus \asc{3}) = \asc{n-2} \ominus \desc{3}$, and this permutation has a jump between $3$ and $n+1$.
At the same time, we must have $\desc{n-2} \oplus \asc{3} \in H$, because otherwise we would have $\jumps{H} = \emptyset$.
Therefore, $H = \{\asc{n+1}, \desc{n-2} \oplus \asc{3}\}$.

(2)$\implies$(3):
If $H = \{\asc{n+1}, \desc{n-2} \oplus \asc{3}\}$, then $\dihed{\interval{1}{n-2}} \leq G$, because it holds that $\Patl[n]{\desc{n-2} \oplus \asc{3}} = \{\desc{n-3} \oplus \asc{3}, \desc{n-2} \oplus \asc{2}\}$, and this set generates $\dihed{\interval{1}{n-2}}$.

(3)$\implies$(1):
Assume that $\dihed{\interval{1}{n-2}} \leq G$.
Then we have $\desc{n-2} \oplus \asc{3} \in H$, because $\Patl[n]{\desc{n-2} \oplus \asc{3}} = \{\desc{n-3} \oplus \asc{3}, \desc{n-2} \oplus \asc{2}\} \subseteq \dihed{\interval{1}{n-2}}$.
Therefore, $(1,n-1) \in \jumps{H}$.
It follows from Lemma~\ref{lem:jumpsets} that $\jumps{H} = \{(1,n-1)\}$, unless $n = 6$.
If $n = 6$, then $\jumps{H}$ equals either $\{(1,5)\}$ or $\{(1,5), (2,7)\}$.
Suppose, to the contrary, that $n = 6$ and $\jumps{H} = \{(1,5), (2,7)\}$. Then $G = \gensg{(1 \; 2 \; 3 \; 4), (2 \; 3 \; 4 \; 5 \; 6)}$ by Lemma~\ref{lem:Comp-prim-n=6}. But $432156 \notin \gensg{(1 \; 2 \; 3 \; 4), (2 \; 3 \; 4 \; 5 \; 6)}$, which contradicts the assumption that $\dihed{\interval{1}{n-2}} \leq G$.
We conclude that $\jumps{H} = \{(1,5)\}$ also in the case where $n = 6$.

Statements \ref{lem:Comp-prim-general:iii} and \ref{lem:Comp-prim-general:iv}
follow from \ref{lem:Comp-prim-general:i} and \ref{lem:Comp-prim-general:ii} by considering conjugates with respect to the descending permutation.
\end{proof}

\begin{theorem}
\label{thm:Comp-primitive}
Assume that $G \leq \symm{n}$ is a primitive group that does not contain the natural cycle $\natcycle{n}$ and $\alt{n} \nleq G$.
Then the following statements hold.
\begin{enumerate}[label=\upshape{(\roman*)}]
\item\label{item:prim-1}
\begin{enumerate}[label=\upshape{(\alph*)}]
\item\label{item:prim-1-6}
If $n = 6$ and $G$ is one of the groups listed in Table~\ref{table:primitive-eq-6}, then $\Compn[n+1]{G}$ equals the group listed on the corresponding row of the table.

\item\label{item:prim-1-not6}
If $n \neq 6$ and $G$ is an overgroup of one of the groups $H$ listed in Table~\ref{table:primitive-neq-6}, then $\Compn[n+1]{G}$ equals the group listed on the corresponding row of the table.

\item Otherwise $\Compn[n+1]{G} \leq \gensg{\desc{n+1}}$.
\end{enumerate}

\item\label{item:prim-2}
$\Compn[n+2]{G} \leq \gensg{\desc{n+2}}$.
\end{enumerate}
\end{theorem}

\begin{table}
\begin{center}
\begin{tabular}{ll}
\toprule
$G$ & $\Compn[7]{G}$ \\
\midrule
$\gensg{(1 \; 2 \; 3 \; 4), (3 \; 4 \; 5 \; 6)}$ & $\{1234567, 2154376, 6734512, 7654321\}$ \\
$\gensg{(1 \; 2 \; 3 \; 4), (2 \; 3 \; 4 \; 5 \; 6)}$ & $\{1234567, 1276543, 1543276, 1567234\}$ \\
$\gensg{(1 \; 2 \; 3 \; 4 \; 5), (3 \; 4 \; 5 \; 6)}$ & $\{1234567, 2165437, 4561237, 5432167\}$ \\
$\gensg{(1 \; 2 \; 3 \; 4 \; 5), (1 \; 3 \; 4)(2 \; 5 \; 6)}$ & $\gensg{\dja{7}{5}}$ \\
$\gensg{(2 \; 3 \; 4 \; 5 \; 6), (1 \; 2 \; 5)(3 \; 4 \; 6)}$ & $\gensg{\ajd{7}{5}}$ \\
\bottomrule
\end{tabular}
\end{center}

\medskip
\caption{The sets $\Compn[7]{G}$ for some primitive groups $G \leq \symm{6}$.}
\label{table:primitive-eq-6}
\end{table}

\begin{table}
\begin{center}
\begin{tabular}{ll}
\toprule
$H$ & $\Compn[n+1]{G}$ \\
\midrule
$\dihed{\interval{1}{n-1}}$ & $\gensg{\dja{n+1}{n-1}}$ \\
$\dihed{\interval{1}{n-2}}$ & $\gensg{\dja{n+1}{n-2}}$ \\
$\dihed{\interval{2}{n}}$ & $\gensg{\ajd{n+1}{n-1}}$ \\
$\dihed{\interval{3}{n}}$ & $\gensg{\ajd{n+1}{n-2}}$ \\
\bottomrule
\end{tabular}
\end{center}

\medskip
\caption{The sets $\Compn[n+1]{G}$ for primitive groups $G$ of degree $n \neq 6$ satisfying $H \leq G$.}
\label{table:primitive-neq-6}
\end{table}

\begin{proof}
\ref{item:prim-1}
This follows easily from Lemmas~\ref{lem:jumpsets}, \ref{lem:Comp-prim-n=6}, and \ref{lem:Comp-prim-general}.
More precisely, Lemma~\ref{lem:jumpsets} describes which subsets of $\nset{n+1} \times \nset{n+1}$ may appear as sets of the form $\jumps{\Compn[n+1]{G}}$, where $G \leq \symm{n}$ is a primitive group satisfying the prescribed conditions.
For each such nonempty feasible subset,
Lemmas~\ref{lem:Comp-prim-n=6} and~\ref{lem:Comp-prim-general} characterize the corresponding groups $G$ and their compatibility sets $\Compn[n+1]{G}$.
Finally, note that $\jumps{\Compn[n+1]{G}} = \emptyset$ if and only if $\Compn[n+1]{G} \leq \gensg{\desc{n+1}}$.

\ref{item:prim-2}
This is proved in Lemma~\ref{prop:Comp-n+2-primitive}.
\end{proof}

\begin{remark}
The first three groups listed in Theorem~\ref{thm:Comp-primitive}\ref{item:prim-1}\ref{item:prim-1-6} (Table~\ref{table:primitive-eq-6}) are isomorphic to the projective general linear group $\PGL{2}{5}$ and the last two are isomorphic to the projective special linear group $\PSL{2}{5}$.
\end{remark}

\begin{remark}
\label{rem:primitive}
It remains a bit out of the scope of the current paper to determine which primitive groups satisfy the conditions of Theorem~\ref{thm:Comp-primitive}\ref{item:prim-1}\ref{item:prim-1-not6}.
Note that if $G \leq \symm{n}$ ($n \neq 6$) is a primitive group that is an overgroup of
$\dihed{\interval{1}{n-1}}$, $\dihed{\interval{1}{n-2}}$, $\dihed{\interval{2}{n}}$, or $\dihed{\interval{3}{n}}$, then $G$ contains a cycle fixing one or two points, so either $\alt{n} \leq G$ or $G$ is isomorphic to one of the groups listed in Theorem~\ref{thm:Jones1.2} parts \ref{Jones:k=1} and \ref{Jones:k=2}.
We leave it as an open question which ones of these groups, if any, do actually contain an appropriate dihedral subgroup.
\end{remark}


\section{Concluding remarks}
\label{sec:concluding}

By repeatedly applying the theorems of previous sections, we can determine for every permutation group $G \leq \symm{n}$
the level sequence
\[
G, \Compn[n+1]{G}, \Compn[n+2]{G}, \dots, \Compn[n+i]{G}, \dots,
\]
as in \eqref{eq:comp-sequence},
or, in the case of intransitive or imprimitive groups, at least have good upper and lower bounds for the members of the sequence.
We have also determined which one of the stable sequences of Theorem~\ref{thm:AtkBea-asymptotic} the level sequence eventually reaches.

Concerning the question how fast the level sequence reaches one of the possible eventual steady sequences, our results can be summarized quickly as follows.
The upper bounds for $m$ given here are quite rough, and sharper bounds can be found in the theorems of previous sections.

\begin{corollary}
\label{cor:nutshell}
Let $G \leq \symm{n}$ and let $m$ be the smallest number $i$ such that $\Compn[n+i]{G}$ is one of the groups $\symm{n+i}$, $\dihed{n+i}$, $\cycl{n+i}$, $\symm{n+i}^{a,b}$, $\gensg{\symm{n+i}^{a,a}, \desc{n+i}}$ \textup{(}see Theorem~\ref{thm:AtkBea-asymptotic}\textup{)}.
\begin{enumerate}[label=\upshape{(\roman*)}]
\item If $G$ is neither an intransitive group nor an imprimitive group with $\natcycle{n} \notin G$, then $m \leq 2$.
\item If $G$ is intransitive, then $m \leq n-1$.
\item If $G$ is imprimitive and $\natcycle{n} \notin G$, then $m \leq p$, where $p$ is the largest proper divisor of $n$.
\end{enumerate}
\end{corollary}

\begin{proof}
This can be read off from Theorems
\ref{thm:symmetric-trivial},
\ref{thm:Comp-alt-n+2},
\ref{thm:Comp-cycle},
\ref{thm:general-intransitive},
\ref{thm:general-imprimitive},
\ref{thm:Comp-primitive}.
Note that if $\Pi$ is a nontrivial partition of $\nset{n}$, then $\maxblocksize_{a,b}(\Pi) \leq n-1$.
Moreover, if the blocks of $\Pi$ have equal size $b$, then $\maxblocksize_{a,b}(\Pi) \leq b \leq p$, where $p$ is the largest proper divisor of $n$.
\end{proof}


\section*{Acknowledgments}

The author would like to thank Reinhard P\"oschel, Sven Reichard, Nik Ru\v{s}kuc, and Tobias Schlemmer for insightful discussions.


\end{document}